\documentclass[12pt]{amsart}
\usepackage{fullpage}
\usepackage{amsmath}
\usepackage{amsfonts}
\usepackage{amssymb}
\usepackage{amsthm}
\usepackage{graphicx,xy}
\usepackage{enumerate}
\usepackage{latexsym}
\usepackage{tikz}
\usetikzlibrary{arrows,calc,matrix,topaths,positioning,scopes,shapes,decorations}
\usepackage{microtype} 
\usepackage{float}

\DeclareMathAlphabet{\mathpzc}{OT1}{pzc}{m}{it}

\def\cal{\mathcal}
\def\d{{\mathrm d}}
\def\F2{\mathbb F_2}
\def\A2{{\mathcal A}_2}
\def\As{{\mathcal As}}

\def\Z{{\mathbb Z}}

\def\R{{\mathbb R}}

\def\SC{{\cal{SC}}}

\def\sc{{\mathrm{SC}}}
\def\hsc{{\mathrm{\bf{sc}}}}
\def\rBr{{\mathcal{RB}}}

\def\SCvor{{\cal{SC}}^{\rm vor}}

\def\dgvs{{\bf dgvs}}

\def\Def{{\rm Def}}
\def\det{{\rm det}}

\def\Fin{{\bf Fin}}

\def\Coder{{\rm Coder}}

\def\Hom{{\rm Hom}}

\def\End{{\rm End}}
\def\sgn{{\rm sgn}}

\def\dd{{\mathbf d}}
\def\su{{\mathbf s}}
\def\cl{\mathpzc c}
\def\op{\mathpzc o}
\def\kfield{\mathbf k}

\def\ie{\emph{i.e.} }
\def\Ainf{\As_\infty}
\def\ot{\otimes} 
\def\Ch{\mathrm{\bf Ch}(\Z)} 
\def\Ba{\mathcal{B}}
\newcommand{\RS}{\mathcal{RS}}
\newcommand{\un}[1]{\underline{#1}}

\newcommand{\colim}{\text{colim}}
\newcommand{\hocolim}{\text{hocolim}}

\newcommand{\delZ}{\delta_{\Z}}

\newcommand{\cat}{\mathbf}
\newcommand{\Tas}{\mathcal{T}_{As}}
\newcommand{\Tasc}{\mathcal{T}_{As,\sq}}
\newcommand{\Tasr}{\mathcal{T}_{As,\bullet}}

\newcommand{\lft}{\text{left}}
\newcommand{\rgt}{\text{right}}
\newcommand{\Nest}{\mathrm{Nest}}
\newcommand{\Nestas}{\mathrm{Nest}_{As}}

\def\E{{\mathrm{E}}}
\def\Dc{\mathcal {D^{\cal P}_{\cl}}}
\def\Do{\mathcal {D^{\cal P}_{\op}}}
\def\Cor{\mathrm{Cor}}

\newcommand{\ac}{\scriptstyle \text{\rm !`}}

\newcommand{\wiggly}{\xymatrix@1@C=15pt{  && \ar@{~}[ll]}}

\newcommand{\straight}{\xymatrix@1@C=30pt{  & \ar@{-}[l]}}

\newtheorem{thm}{Theorem}[subsection]
\newtheorem{lem}[thm]{Lemma}
\newtheorem{prop}[thm]{Proposition}
\newtheorem{cor}[thm]{Corollary}
\theoremstyle{definition}
\newtheorem{defn}[thm]{Definition}

\newtheorem{ex}[thm]{Example}

\newtheorem{nota}[thm]{Notation}
\newtheorem*{nota*}{Notation}

\theoremstyle{remark}
\newtheorem{rem}[thm]{Remark}

\theoremstyle{remark}
\newtheorem{Convention}[thm]{Convention}

\xyoption{all}
\CompileMatrices

\begin{document}

\title[On the deformation complex of homotopy affine actions]{On the deformation complex of homotopy \\ affine actions}
\author{Eduardo Hoefel}
\address{Universidade Federal do Paran\'a, Departamento de Matem\'atica C.P. 019081, 81531-990 Curitiba, PR - Brazil }
\email{hoefel@ufpr.br}
\author{Muriel Livernet}
\address{Universit\'e Paris Diderot, Sorbonne Paris Cit\'e, IMJ-PRG, UMR 7586 CNRS, 75013 
  \mbox{Paris}, France}
\email{livernet@math.univ-paris-diderot.fr}
\author{Alexandre Quesney}
\address{Universidade de S\~ao Paulo,
Instituto de Ci\^encias Matem\'aticas e de Computa\c c\~ao,
Avenida Trabalhador S\~ao-carlense, 400 -
CEP: 13566-590 - S\~ao Carlos, SP -
Brazil }
\email{math@quesney.org} 
\thanks{This collaboration is funded by the Capes-Cofecub program MA 763-13 "Factorization algebras in Mathematical Physics and algebraic topology". The third author was partially supported by "CAPES - Projeto 88881.030367/2013-01 Bolsa PDJ".}
\date{\today}
\begin{abstract}
An affine action of an associative algebra $A$ on a vector space $V$ is an algebra morphism 
$A \to V \rtimes {\rm End}(V)$, where $V$ is a vector space and 
$V \rtimes {\rm End}(V)$ is the algebra of affine transformations of $V$.
The one dimensional version of the Swiss-Cheese operad, denoted
$\hsc_1$, is the operad that governs affine actions of associative algebras. 
This operad is Koszul and admits a minimal
model denoted by $(\hsc_1)_\infty$. Algebras over this minimal model
are called {\it Homotopy Affine Actions}, they consist of an 
$A_\infty$-morphism $A \to V \rtimes {\rm End}(V)$, where $A$ is an
$A_\infty$-algebra. In this paper we prove a relative version of Deligne's conjecture. In other words, we show that the deformation complex of
a homotopy affine action has the structure of an algebra over an ${\rm SC}_2$ operad. That structure is naturally compatible with the ${\rm E}_2$ structure on the deformation complex of the $A_\infty$-algebra.
\end{abstract}
\maketitle

%
%
%
%

\section*{Introduction}

The original Deligne's conjecture, now proved by many authors, states that there is an action of an operad weakly equivalent to the operad of chains
of the 2-little discs operad $\mathcal D_2$ on the Hochschild cochain complex of an associative algebra. This question naturally arose as a topological counterpart of the fact that the Hochschild cohomology of an associative algebra have a Gerstenhaber algebra structure, that is, an action of the homology of $\mathcal D_2$.

The Hochschild cochain complex of an algebra controls deformations of the algebra and is also known as the deformation complex of the associative algebra. In this paper we are interested in understanding the structure of the deformation complex associated to a pair $(A,V)$ where $A$ is an associative algebra and $V$ is a vector space endowed with an affine action $A \to V \rtimes {\rm End}(V)$. In this relative version of the original Deligne's conjecture, the topological operads in actions are the Swiss-cheese operads, which are $2$-colored operads based on the little discs operads. Indeed we go further in this analysis since we prove that
the deformation complex of an operad under $\sc_1$ is an $\sc_2$-algebra. 
Before going through the details of the paper, let us give an historical background.

\medskip

{\it State of the art.} 
An operad is said to be $\E_n$ if it is weakly equivalent to the operad of chains of the little discs operad $\mathcal D_n$.
In the case $n=1$ the operad for associative algebra $\As$ is $\E_1$ as well as its minimal model $\As_\infty$. Any $\E_1$-algebra is in particular an $\As_\infty$-algebra and admits a deformation complex, similar to the Hochschild complex of an associative algebra. Indeed an $\As_\infty$-algebra is defined through its bar construction, which is its deformation complex.
An $\As_\infty$-version of Deligne's conjecture states that the deformation complex of any $\E_1$-algebra is an $\E_2$-algebra.
This was done by Kontsevich and Soibelman in \cite{KS00}, where they build an explicit operad $\mathcal M$ acting on the deformation complex of an $\As_\infty$-algebra (see also  \cite{CamWill16} for a very nice proof of this result).
This operad will correspond to the closed part of ours. We refer also to Kauffman and Schwell in \cite{KS10} for a topological proof of this fact.

The $2$-dimensional Swiss cheese operad  $\SCvor_2$ has been introduced by S. Voronov in \cite{Voronov99} as a topological 2-colored operad, in order to understand spaces of configuration of points used by M. Kontsevich in deformation quantization and by B. Zwiebach in \cite{Zwie97} in Open-Closed string field theory. The two colors are commonly called closed and open.
It is shown by A. Voronov that its homology, $\hsc_2^{vor}$, governs Gerstenhaber algebras (the closed structure) acting on associative algebras (the open structure).

In  \cite{Hoefel09}, the first author studied another version of the Voronov's Swiss-Cheese operad introduced by M. Kontsevich in \cite{Kontse99}, that we call here the Swiss-cheese operad and denote $\SC_2$.  This operad admits operations having only closed input and an open output. In particular, there is an operation that transforms a closed variable into an open one, called the whistle. The operad $\SCvor_2$ is a suboperad of $\SC_2$.
This operad is more accurate in compactification and deformation theory  though it has the disadvantage of not being  formal, as proved in \cite{Liv15}. 
Given an associative algebra, the pair $(HH^*(A),A)$  is  an algebra over $\hsc_2$, the homology of $\SC_2$.
The map $HH^*(A,A)\rightarrow A$ sends $HH^*(A,A)$ onto $HH^0(A,A)$, the center of $A$. 

As in the classical case, an operad is said to be $\sc^{(vor)}_n$ if it is weakly equivalent to the operad of chains of the topological operad $\SC^{(vor)}_n$. A differential graded vector space is said to be an $\sc^{(vor)}_n$-algebra if it has an action of an $\sc^{(vor)}_n$ operad. 
 In \cite{DTT11}, V. Dolgushev, D. Tamarkin and B. Tsygan proved  that the pair $(CH^*(A,A),A)$ is an $\sc^{vor}_2$-algebra, where $CH^*(A,A)$ denotes the Hochschild cochain complex of the associative algebra $A$.

\medskip

{\it Main result.} In this paper we prove a  relative $\As_\infty$ version of the Deligne's conjecture. More precisely, we prove that the deformation complex of an $\sc_1$-algebra is an $\sc_2$-algebra. Indeed we prove that the deformation complex of any operad under an $\sc_1$-operad is an $\sc_2$-algebra. 
Notice that our approach is different in spirit  
from \cite{DTT11}:  via deformation theory, we obtain an $\sc_2$-algebra $(CH^*(A,A),CH^*(A,M))$, where $M$ is an $A$-bimodule, such that its homology is the  $\hsc_2$-algebra $(CH^*(A,A),A)$.  
However, the $\sc_2$ operad that we build here does not act on the pair $(CH^*(A,A),A)$, as we explain in Remark \ref{R:DTT}.

\medskip

{\it Contents of the paper.} 
The algebra of affine transformations of a vector space $V$ will be denoted by: 
${\rm End}^+(V) := V \rtimes {\rm End}(V)$. It is the deformation complex of the $\E_0$-algebra $V$, as pointed out by Kontsevich in \cite{Kontse99}. 
The first parts of the paper are concerned with computing the deformation complex of an $\sc_1$-algebra. In dimension 1, the relative case behaves like the $\E_1$-case. An $\hsc_1$-algebra is a pair $(A,V)$ where $A$ is associative and $V$ is endowed with an algebra morphism $A \to {\rm End}^+(V)$, hence $V$ is a (left) $A$-module and there is a map $f:A\rightarrow V$ such that $f(ab)=af(b)$. The operad $\hsc_1$ is Koszul (Proposition \ref{P:hsc1Koszul}) and admits a minimal
model $(\hsc_1)_\infty$.  Thus any $\sc_1$-algebra is an $(\hsc_1)_\infty$-algebra.
Proposition \ref{P:maurercartan} says that an  $(\hsc_1)_\infty$-algebra $(A,V)$ is given by an $\As_\infty$-algebra $A$ and an $\As_\infty$-morphism $A \to \End(V)^+$.

The main result of Section 2,  Theorem \ref{T:mapcone} states that the deformation complex of an $(\hsc_1)_\infty$-algebra $(A,V)$  is the mapping cone of a map induced by the whistle from the Hochschild cochain complex of $A$ to the Hochschild cochain complex of $A$ with coefficients in $\End^+(V)$.  An operad $\mathcal Q$ is an operad under $\mathcal P$ if there is a map of operads $\eta:\cal P\rightarrow \cal Q$. The endomorphism operad of a $\cal P$-algebra is a special case of an operad  under $\mathcal P$. In section 2, we also define  the deformation complex of an operad under $(\hsc_1)_\infty$ and prove in
Theorem \ref{T:coneoperad} that it is the mapping cone of a map between Hochschild cochain complexes of operads as defined by Turchin 
in \cite{Tu05}.

Section 3 is devoted to the construction of the operad $\rBr$, called the operad of relative braces. This operad is built in the category of chain complexes over $\Z$. In Theorem \ref{T:action}. we prove that $\rBr$ acts on the deformation complex of any operad $\cal P$ under
$(\hsc_1)_\infty$. The operad of relative braces has for closed part the operad $\cal M$ described by Kontsevich and Soibelman in \cite{KS00}, and that it acts on the deformation complex of an operad under $\As_\infty$ is the result obtained therein. Here we describe the open part and the interaction with the closed part.

Section 4 is devoted to the proof of the Swiss-cheese Deligne's conjecture, that is, the proof that the operad $\rBr$ is an $\sc_2$-operad, weakly equivalent to the operad of chains of the Swiss cheese topological operad $\SC_2$, over $\Z$ (Theorem \ref{th: equiv}). In \cite{Q-SCMRL}, the third author proved that  a third operad $\RS$ is an $\sc_2$-operad. Hence,
in Section 4 we prove a weak homotopy equivalence $\rBr\rightarrow \RS$. This is performed by using cellular decompositions as in \cite{Batanin-Berger-Lattice}.

That the closed part of $\RS$ and $\rBr$ have isomorphic homologies over a field of characteristic $0$ has been obtained by D. Dolgushev and T. Willwacher in \cite{DolWil14}. Our result is a generalization of their, since we have a weak equivalence over $\Z$.
\tableofcontents
\section{Preliminaries}

\subsection{On differential graded vector spaces} 

We work on a ground field $\kfield$ of characteristic $0$.
 The category  $\dgvs$ is the category of lower $\Z$-graded $\kfield$-vector spaces together with a differential of degree $-1$.  Objects in $\dgvs$ are called for short  dgvs.  The degree of $x\in V$, where $V$ is a dgvs is  denoted by $|x|$. 
 This category is enriched over $\dgvs$ and  the differential graded vector space of maps from $V$ to $W$ is denoted by $\Hom(V,W)$. The dgvs
$\Hom(V,V)$ is denoted $\End(V)$.

  The suspension of a dgvs $V$ is denoted by $\su V$ and defined as $(\su V)_n=V_{n-1}$. The suspension $\su$ is also seen as a map
 $\su: V\rightarrow \su V$ of 
  $\dgvs$ of degree $+1$.

  All along the text, we use Koszul sign convention for morphisms, that is, if $f:V_1\rightarrow W_1$ and $g: V_2\rightarrow W_2$ are morphisms
  of degree $|f|$ and $|g|$ respectively, then $f\otimes g:V_1\otimes V_2\rightarrow W_1\otimes W_2$ is a morphism of degree $|f|+|g|$ defined by
  $$(f\otimes g)(v_1\otimes v_2)=(-1)^{|g||v_1|} f(v_1)\otimes g(v_2).$$

\subsection{Notation for operads} We assume the reader to be familiar with operad theory.  For notation and definition on colored operads and Koszul duality for colored operads we refer to \cite{vanderlaan}. We refer also to \cite[Section 2]{HoeLiv12} for details on $2$-colored quadratic operads and their Koszul dual.

Consequently, this section is devoted to notations rather than to definitions. 

\subsubsection{$S$-operads}
Let $(\cal C,\otimes,\star)$ be a symmetric monoidal category.  Let $S$ be a finite set of colors.  
The category $\Fin_S$ is the category whose objects $(X,x_0;i:X\rightarrow S)$ are pointed, 
non-empty finite sets together with a map $i$, and whose morphisms are pointed bijections commuting with $i$. 
An {\sl $S$-collection} in $\mathcal C$ is a contravariant functor from $\Fin_S$ to $\cal C$.
There is a notion of plethysm in this category, denoted $\circ_S$, and 
an {\sl $S$-operad}  in $\cal C$ is an $S$-collection $\cal P$ that is a monoid for $\circ_S$.

We will use most frequently the following notation:  giving  an object in $\Fin_S$ of the form
$(X=\{0,\ldots, k\},x_0=0,i:X\rightarrow S)$  amounts to giving a sequence of elements $(s_1,\ldots,s_n)$ of $S$ together with an element $s=i(0)$.
An $S$-collection $M$ is written as a family $M(I;s)\in\cal C$, where $I$ runs over the aforementioned sequences, and $s=i(0)$.
We will denote by $|I|$ the number of elements in the sequence and for $u\in S$ the set $\{s_i|s_i=u\}$ is denoted  $I_u$.

With this interpretation in mind, there is a right action of the symmetric group 
\begin{align*}
M(s_1,\ldots,s_n;x)&\rightarrow M(s_{\sigma(1)},\ldots,s_{\sigma(n)};x)\\
m & \mapsto  m\cdot\sigma
\end{align*}
induced by the bijections in $\Fin_S$.

Let $\cal P$ be an $S$-operad. For $p\in \cal P(I;s)$ and $q\in \cal P(J;s_i)$, such that the $i$-th term of the sequence $I$ is $s_i$, the partial composition is denoted, as usual, by 
$p\circ_i q$.

\subsubsection{Regular $S$-operads}  Functors from the discrete category whose objects are pointed sequences $(I;s)$ to $\mathcal C$ are called $S$-graded objects. Similarly to the case of uncolored operads, plethysm can be defined in this category, and then monoids, called {\sl non-symmetric $S$-operads}.
The forgetful functor from $S$-operads to non-symmetric $S$-operads admits a left adjoint ${\mathcal L}_{S}$. 
An $S$-operad $\mathcal P$ is  {\sl regular} if there exists a non-symmetric $S$-operad $\cal P^{ns}$ so that
$$\mathcal P={\mathcal L}_{S}(\cal P^{ns})$$

\subsubsection{$S$-operads described by generators and relations} 
We will denote by $\mathcal F(M)$ the free $S$-operad generated by the $S$-collection $M$ and by $\mathcal F(M;R)$ its quotient by the ideal
generated by $R\subset \mathcal F(M)$.

\subsubsection{Algebras over an $S$-operads} Let $A=(A_s)_{s\in S}$ be a family 
of objects of $\cal C$. Let $I$ be a sequence $(s_1,\ldots,s_k)$ of elements of $S$.
We denote by $A^{I}$ the object $A_{s_1}\otimes A_{s_2}\otimes \ldots\otimes A_{s_k}$.
The endomorphism operad $\End_A$ is the $S$-operad defined by $\End_A(I;s)=\Hom(A^I;A_s)$. By definition,
$A$ is {\sl an algebra over the $S$-operad $\mathcal P$} (or $\mathcal P$-algebra) providing there is a map of operads $\gamma:\mathcal P\rightarrow \End_A$.

For a sequence $I$ such that $|I|=k$, for  $p\in \cal P(I;s)$ and $\underline a=(a_1,\ldots,a_k)\in A^{\otimes I}$ we denote by
$p(a_1,\ldots,a_k)$ the element of $A_s$ obtained as $\gamma(p)(\underline a)$.

In the category $\dgvs$,
the Koszul sign convention implies that if $A$ is a $\cal P$-algebra, then for $p\in \cal P((x_1,\ldots,x_k);s)$ and $q\in \cal P(J;t)$, such that
$t=x_i$ and $|J|=l$

$$(p\circ_i q)(a_1,\ldots,a_{k+l-1})=(-1)^{|q|\sum_{r=1}^{i-1} |a_r|}p(a_1,\ldots,a_{i-1},q(a_i,\ldots,a_{i+l-1}),a_{i+l},\ldots, a_{k+l-1}).$$

\subsubsection{Suspension of $S$-collections and $S$-operads}\label{suspension}
The suspension of the $S$-collection $\cal P$ is
$$\Lambda\cal P(I;x)=s^{1-|I|}\cal P(I;x)\otimes \sgn_{|I|}.$$
If $\cal P$ is a $\{\cl,\op\}$-colored operad, then the structure of $\cal P$-algebra on the pair $(V_\cl,V_\op)$ is equivalent to the structure of $\Lambda\cal P$-algebra on the pair $(sV_\cl,sV_\op)$.

 The suspension of the $\{\cl,\op\}$-collection $\cal P$ with respect to the color $\cl$ is
 $$\Lambda_\cl\cal P(I;x)=s^{\delta_{x,\cl}-|I_\cl|}\cal P(I;x)\otimes \sgn_{|I_\cl|},$$
 where $\delta$ denotes the Kronecker symbol.
If $\cal P$ is an operad, then the structure of $\cal P$-algebra on the pair $(V_\cl,V_\op)$ is equivalent to the structure of $\Lambda_\cl\cal P$-algebra on the pair $(sV_\cl,V_\op)$.

\subsubsection{Convention for the $\circ$-product}\label{S:preLie}
Let $\cal P$ be an $S$-operad in the category $\dgvs$. For $p\in \cal P((x_1,\ldots,x_k);s)$ and $q\in \cal P(J;t)$, define:
$$p\circ q=\sum_{i| x_i=t} p\circ_i q.$$
Notice that {\sl a priori} this notation is not well constrained since each $p\circ_i q$ might live in different components of $\mathcal P$. However, if
one considers the (weight) graded vector space 
$$\mathcal O(\mathcal P)_k=\bigoplus_{|I|=k+1,s\in S} \mathcal P(I;s),$$
then one has 
$$p\in \mathcal O(\mathcal P)_k, q\in O(\mathcal P)_l \implies p\circ q\in O(\mathcal P)_{k+l}.$$
Furthermore, the $\circ$ product is a graded pre-Lie product.

\subsection{On the topological Swiss cheese operad} There exist two versions of the Swiss cheese operad, the original one has been defined by Voronov in \cite{Voronov99}. The one we consider in this paper is a slightly different version that  has been considered by Kontsevich in \cite{Kontse99}.

Let $\mathcal D_n$ be the $n$-little discs operad. The topological operad $\SC_n$ is a $2$-colored operad, where the set of colors is $\{\cl,\op\}$.

For $k\geq 1$ and $J\in \{\cl,\op\}^k$ , set $J_\cl=\{x\in J, x=\cl\}$ and  $J_\op=\{x\in J, x=\op\}$ .

The topological space $\SC_n(J;\cl)$ is $\cal D_n(J)$, if $J_\cl=J$ and is the empty set otherwise.

The topological space $\SC_n(J;\op)$ is the space of labeled configurations of  
non overlapping $|J_\cl|$ balls and $|J_\op|$ semiballs  in the unit  upper semiball of dimension $n$, the semiballs being centered 
on the hyperplane of $\R^n$ whose equation is $x_{n}=0$. With our convention, $J_\op$ can be empty, that is, we allow operations having only closed inputs and an open output, contrary to the operad built by Voronov. However, for $x\in\{\cl,\op\}$, $\SC_n(\emptyset;x)$ is the empty space.

We refer to \cite{Kontse99} and \cite{HoeLiv13} for more details on $\SC_n$.

\subsection{On $\sc_n$ and $\hsc_n$-algebras} By definition an $\sc_n$-algebra is an algebra over an operad which is weakly equivalent to the
operad of singular chains with coefficients in $\kfield$ of $\SC_n$. Two differential graded operads are said to be weakly equivalent if there is a zig-zag
of quasi-isomorphisms of operads relating them. The homology operad of $\SC_n$ is denoted by $\hsc_n$. 

We recall that an $\hsc_2$-algebra structure on a pair of dgvs $(G,A)$ is given by the following data:
\begin{itemize}
\item[-] $G$ is a differential graded Gerstenhaber algebra.
\item[-] $A$ is a differential graded associative algebra.
\item[-] These two dgvs are related by  a central morphism of differential graded associative algebras $f:G\rightarrow A$.
\end{itemize}

\section{Deformation complex of $\sc_1$-algebras and its variants}

The operad $\SC_1$ deserves a special attention because it behaves differently than the operads $\SC_n$ for $n\geq 2$. 
The first difference relies in the fact that each space $\SC_1(J;x)$ has contractible path connected components and that $\SC_1$ is a regular operad.
The second difference relies in the fact that for $n\geq 2$ the operad $\SC_n$ is not formal as proved in \cite{Liv15}, though $\hsc_n$ is Koszul (in the quadratic-linear sense) as proved in \cite{HoeLiv13},  whereas $\SC_1$ is formal and $\hsc_1$ is Koszul quadratic.

As a consequence one can consider, as a model for $\sc_1$-algebras, either $\hsc_1$-algebras or $(\hsc_1)_\infty$-algebras. Since the $\hsc_1$-algebras are precisely affine actions, the $(\hsc_1)_\infty$-algebras are also called {\it Homotopy Affine Actions}.

In this section, we will compute the deformation complex of an $\hsc_1$-algebra, deduce the definition of an $(\hsc_1)_\infty$-algebra and compute the
deformation complex of an $(\hsc_1)_\infty$-algebra. These deformation complexes will be expressed in terms of Hochschild cohomology of an associative algebra or $\Ainf$-algebra with coefficients in a module. We end the section by considering the deformation complex of an operad under
$(\hsc_1)_\infty$.

\subsection{Notation for the operad $\hsc_1$}  Notice that because $\hsc_1$ is regular, that is, $\hsc_1=\mathcal L_{\{\cl,\op\}}(\hsc_1^{ns})$, algebras over $\hsc_1^{ns}$ in the context of 
non-symmetric  $\{\cl,\op\}$-operads coincide with algebras over $\hsc_1$ in the context of $\{\cl,\op\}$-operads.
Hence, in all the text $\hsc_1$ will denote the non-symmetric operad, that is
$$\hsc_1(\cl^{\times n};\cl)=\kfield,\; \hsc_1(\cl^{\times n};\op)=\kfield,\; \hsc_1(\cl^{\times n-1},\op;\op)=\kfield,\ \forall n\geq 1$$
and is $0$ otherwise.

\subsection{Algebras over the operad $\hsc_1$}\label{S:opsc1}
An algebra over $\hsc_1$ in the category of vector spaces is a pair $(A,V)$ subject to the following conditions:
\begin{enumerate}
\item $(A,*)$ is an associative algebra,
\item $V$ is an $A$-left module, that is, there exists an action $\rho:A\otimes V\rightarrow V$ such that $\rho(a*b,v)=\rho(a,\rho(b,v)),\forall a,b\in A,\forall v\in V$,
\item\label{R_f} there is a map $f:A\rightarrow V$ such that $f(a*b)=\rho(a,f(b)),\forall a,b\in A.$
\end{enumerate}
Equivalently the operad $\hsc_1$ is described as
$$\hsc_1=\mathcal F(\mu,\rho,f)/<R>$$
where $\mu\in\hsc_1(\cl,\cl;\cl), \rho\in\hsc_1(\cl,\op;\op)$ and $f\in\hsc_1(\cl;\op)$. The ideal $R$ is generated by the quadratic relations
$$\mu\circ_1\mu-\mu\circ_2\mu,\quad \rho\circ_1\mu-\rho\circ_2\rho,\quad  f\circ_1\mu-\rho\circ_2 f.$$

\begin{prop}\label{P:hsc1Koszul} The operad $\hsc_1$ is Koszul.
\end{prop}

\begin{proof} The operad $\hsc_1$ is a regular colored operad, and one can adapt the rewriting rule explained in \cite[Chapter 8]{LodVal}.

The order $\mu>\rho>f$ induces a total order on planar colored trees. For instance in the weight 2 part of the operad, one has
$$\mu\circ_1\mu>\rho\circ_1\mu>f\circ_1\mu>\mu\circ_2\mu>\rho\circ_2\rho>\rho\circ_2 f$$
and the three critical pairs $\mu\circ_1\mu\circ_1\mu, \rho\circ_1\mu\circ_1\mu, f\circ_1\mu\circ_1\mu$ are all confluent (they all give rise to a pentagon relation).
\end{proof}

Assume $(A,V)$ is a dg $\hsc_1$-algebra, and let $d_A$ be the differential on $A$ and $d_V$ be the differential on $V$. Following Kontsevich  
in \cite{Kontse99}, one gets
the following definition.

\begin{defn}\label{D:endvplus} Let $V\in\dgvs$.  The Lie algebra of the group of affine transformations of $V$,  denoted by $\End^+(V)=V \rtimes \End(V)$, is a differential graded associative algebra for the following multiplication and differential:

\begin{align*}
(v,\varphi)\cdot(w,\psi)=&(\varphi(w),\varphi\psi), \\
d_{End^+(V)}(v,\varphi)=&(d_V v,d_V\varphi-(-1)^{|\varphi|}\varphi d_V).
\end{align*}
\end{defn}

\begin{prop}\label{P:dgsc1algebras} Differential graded algebras over the operad $\hsc_1$ are triples $(A,V,\alpha)$ where $A$ is a differential graded associative algebra, $V$ a dgvs
and $\alpha:A\rightarrow \End^+(V)$ is a map of differential graded associative algebras.
\end{prop}

\begin{proof}  The relations satisfied by $\rho$ and $f$ in the definition of an $\hsc_1$-algebra implies that $\alpha:A\rightarrow \End^+(V)$ defined by
$\alpha(a)=(f(a),\rho(a))$ satisfies $\alpha(ab)=(f(ab),\rho(ab))=(\rho(a)(f(b)),\rho(a)\rho(b))=(f(a),\rho(a))\cdot(f(b),\rho(b))$
A similar computation gives the result for the differential.
\end{proof}

\subsection{Coderivations} Following general operad theory (see  \cite[Chapter10]{LodVal}),  the operad $(\hsc_1)_\infty$ is defined as
$\Omega(\hsc_1^{\ac})$ and the deformation complex of an $\hsc_1$-algebra $(A,V)$ is given by
the complex of coderivations $(\Coder(\hsc_1^{\ac}(A,V)),[\delta_\varphi,-])$ where $\delta_\varphi$ is the square zero coderivation induced by the $\hsc_1$-algebra structure on $(A,V)$. Furthermore an $(\hsc_1)_\infty$-algebra structure on the pair $(A,V)$ is equivalent to a square zero coderivation $\delta$
on $\Coder(\hsc_1^{\ac}(A,V))$ and the complex of deformation of the $(\hsc_1)_\infty$-algebra  $(A,V)$ is
$(\Coder(\hsc_1^{\ac}(A,V)),[\delta,-])$.

In this section we  compute the complex of deformation of an $(\hsc_1)_\infty$-algebra, which is equivalent to the structure of a homotopy affine action.

\begin{lem} Let $(A,V)$ be a pair of graded vector spaces. The $\hsc_1^{\ac}$-coalgebra $\hsc_1^{\ac}(A,V)$ is the pair $(B,W)$ where $B=\su^{-1}\overline{T}^c(\su A)$
and $W=\overline{T}^c(\su A)\oplus T^c(\su A)\otimes V$. The structure maps are the following
$$\Delta: B\rightarrow B \otimes B, \quad \Delta_l: W\rightarrow B\otimes W,\quad h: W\rightarrow B,$$
where $\su B$ is the free coassociative coalgebra generated by $\su A$ and $\Delta'=(\su\otimes \su)\Delta \su^{-1}$ is the deconcatenation, $W$ is a cofree left $\hsc_1^{\ac}$-$B$-comodule, that is, $\Delta'_l=(\su\otimes 1_W)\Delta_l$ is the deconcatenation, and $h'=\su h$ is the projection onto the first factor of $W$.
\end{lem}

\begin{proof} This case being similar to the case of the associative operad, we refer to  \cite[Chapter 7]{LodVal} for computational details.

From Section \ref{S:opsc1},  the cooperad $\hsc_1^{\ac}$ is $\cal C(\su \mu,\su \rho,\su f; \su ^2R)$ and satisfies:

\begin{equation}\label{D:sc1dual}
\hsc_1^{\ac}(\cl^{\times n};\cl)=\kfield\;\su ^{n-1}\mu_n,\quad \hsc_1^{\ac}(\cl^{\times m};\op)=\kfield\; \su ^{m}f_m,\quad
 \hsc_1^{\ac}(\cl^{\times p},\op;\op)=\kfield\, \su ^{p}\rho_{p+1},
 \end{equation}
 
 for  $n, m\geq 1$ and $p\geq 0$ with  $\mu_n,f_m,\rho_{p}$ of degree $0$.
As a consequence,

\begin{align*}
B=&\bigoplus_{n\geq 1} \hsc_1^{\ac}(\cl^{\times n};\cl)\otimes A^{\otimes n}=\bigoplus_{n\geq 1} \su ^{n-1} A^{\otimes n}=\su ^{-1}\overline{T}^c(\su A)\\
W=&\bigoplus_{m\geq 1} \hsc_1^{\ac}(\cl^{\times m};\op)\otimes A^{\otimes m}\oplus 
\bigoplus_{m\geq 0} \hsc_1^{\ac}(\cl^{\times m},\op;\op)\otimes A^{\otimes m}\otimes V\\
=&\bigoplus_{m\geq 1} \su ^m A^{\otimes m}\oplus 
\bigoplus_{m\geq 0}\su ^m A^{\otimes m}\otimes V=\overline{T}^c(\su A)\oplus T^c(\su A)\otimes V.
\end{align*}
The maps $\Delta,\Delta_l$ and $h$ follows similarly to the case of the operad $\As$.
We refer to  \cite[Section 3]{ALRWZ} for the signs convention.
\end{proof}

\begin{rem} In the above lemma,  $\Delta'=(\su\otimes \su)\Delta \su ^{-1},\Delta'_l=(\su \otimes 1_W)\Delta_l$ and $h'=\su h$, 
satisfy coassociative relations, that is
$$(\Delta'\otimes 1_{\su B})\Delta'=(1_{\su B}\otimes \Delta')\Delta',\quad (\Delta'\otimes 1_W)\Delta'_l=(1_{\su B}\otimes \Delta'_l)\Delta'_l,\quad
\Delta' h'=(1_{\su B}\otimes h')\Delta'_l.$$
The maps $\Delta'$ and $\Delta'_l$ are deconcatenation and we use the following notation: if $\underline a\in T^c(\su A)$ 
then
$$\Delta'(\underline a)=\sum a_{(1)}\otimes a_{(2)},$$
same notation for $\Delta'_l$.
\end{rem}

\begin{lem} Let $(A,V)$ be a pair of graded vector spaces.  There is an isomorphism
$$\delta: \Hom(\overline{T}^c(\su A),\su A)\times \Hom(\overline{T}^c(\su A)\oplus T^c(\su A)\otimes V ,V)\rightarrow \Coder(\hsc_1^{\ac}(A,V))$$
given by
$$\delta(\varphi=(\varphi_\cl,\varphi_\op))=((\delta\varphi)_\cl,(\delta\varphi)_\op)$$
with
\begin{equation}\label{E:coder}
\begin{aligned}
(\delta\varphi)_\cl(\underline a)=&\sum\pm a_{(1)}\varphi_\cl(a_{(2)})a_{(3)}\\
(\delta\varphi)_\op(\underline a+\underline b\otimes v)=&\sum\pm a_{(1)}\varphi_\cl(a_{(2)})a_{(3)}+
\sum\pm a_{(1)}\otimes\varphi_\op(a_{(2)})\\
+&\sum\pm b_{(1)}\varphi_\cl(b_{(2)})b_{(3)}\otimes v+ \sum\pm b_{(1)}\otimes\varphi_\op(b_{(2)}\otimes v),
\end{aligned}
\end{equation}
where the signs are determined by the Koszul sign rule.
The graded Lie structure on the space of coderivations induces a graded pre-Lie structure on the space 
$\Hom(\overline{T}^c(\su A),\su A))\times \Hom(\overline{T}^c(\su A)\oplus T^c(\su A)\otimes V ,V)$
given by
\begin{equation}\label{E:preLie}
(\varphi\circ \psi)_\cl=\varphi_\cl\circ \psi_\cl,\quad (\varphi\circ \psi)_\op=\varphi_\op\circ(\psi_\cl+\psi_\op).
\end{equation}
\end{lem}

\begin{proof} We recall our notation: $B=\su^{-1}\overline{T}^c(\su A)$
and $W=\overline{T}^c(\su A)\oplus T^c(\su A)\otimes V$. Hence $\Hom(\hsc_1^{\ac}(A,V),(A,V))=\Hom(B,A)\times\Hom(W,V).$
Notice that $\Hom(B,A)=\Hom(\overline{T}^c(\su A),\su A)$.

Let $\varphi=(\varphi_\cl,\varphi_\op)\in  \Hom(B,A)\times \Hom(W,V)$, that is
$\varphi_\cl: \overline{T}^c(\su A)\rightarrow \su A$ and $\varphi_\op:\overline{T}^c(\su A)\oplus T^c(\su A)\otimes V\rightarrow V.$ 
The map $\delta\varphi$  is the unique coderivation of $\hsc_1^{\ac}$-coalgebras such that its composition  with the projection
$\hsc_1^{\ac}(A,V)\rightarrow (A,V)$ coincides with $\varphi$. Being a coderivation of $\hsc_1^{\ac}$-coalgebras amounts to the following commutative diagrams:

$$\xymatrix{\su B\ar[r]^{\Delta'}\ar[d]_{(\delta\varphi)_\cl}&\su  B\otimes  \su B\ar[d]^{(\delta\varphi)_\cl \otimes 1_{\su B}+1_{\su B}\otimes(\delta\varphi)_\cl }
&&& W\ar[r]^{\Delta'_l}\ar[d]_{(\delta\varphi)_\op}& \su B\otimes W\ar[d]^{(\delta\varphi)_\cl \otimes 1_W+1_{\su B}\otimes(\delta\varphi)_\op } &&&
W\ar[d]_{(\delta\varphi)_\op} \ar[r]^{h'}& \su B\ar[d]^{(\delta\varphi)_\cl}  \\
\su B\ar[r]^{\Delta'}& \su B\otimes \su B &&& W\ar[r]^{\Delta'_l}& \su B\otimes W && &W \ar[r]^{h'}& \su B }$$

and equations \eqref{E:coder} make the diagrams commute. The pre-Lie product $\varphi\circ \psi$ is given by the composite 
of $\varphi$ with $\delta(\psi)$, yielding  formulas \eqref{E:preLie}.
\end{proof}

\begin{nota} Let $\overline{\Hom}(T^c(\su A),\End^+(V))$ be the sub-vs of $\Hom(T^c(\su A),\End^+(V))$ of maps $f$ satisfying  $f(\kfield)\subset\End(V)$.
The graded vector space $ \Hom(\overline{T}^c(\su A)\oplus T^c(\su A)\otimes V ,V)$ is isomorphic to $\overline{\Hom}(T^c(\su A),\End^+(V))$.
Furthermore, the latter graded vector space is a  graded associative 
algebra using the convolution product, where $T^c(\su A)$ is the free conilpotent coalgebra with deconcatenation and 
$\End^+(V)$ is the algebra defined in Definition \ref{D:endvplus}.
We denote by $*$ this convolution product. More precisely
\begin{equation}\label{E:star}
\varphi*\psi=\mu(\varphi\otimes\psi)\Delta',
\end{equation}
where $\mu$ denotes the product in $\End^+(V)$.
\end{nota}

\begin{lem} Let $(A,V)$ be a pair of graded vector spaces. There is an isomorphism of graded pre-Lie algebras
$$\Hom(\overline{T}^c(\su A),\su A))\times \overline{\Hom}(T^c(\su A),\End^+(V))\simeq \Hom(\overline{T}^c(\su A),\su A))\times \Hom(\overline{T}^c(\su A)\oplus T^c(\su A)\otimes V ,V)$$
where the pre-Lie structure on the left hand side is given by
\begin{equation}\label{E:prelieproduct}
(\varphi\odot \psi)_\cl=\varphi_\cl\circ \psi_\cl,\quad (\varphi\odot \psi)_\op=\varphi_\op\circ \psi_\cl+\varphi_\op* \psi_\op.
\end{equation}
\end{lem}

\begin{proof} Let $\varphi=(\varphi_\cl,\varphi_\op)$ be an element of $\Hom(\overline{T}^c(\su A),\su A)\times \overline{\Hom}(T^c(\su A),\End^+(V))$
and let $\tilde\varphi$ be the corresponding element on the right hand side.  We write $(\varphi_\op^1,\varphi_\op^2)$ for the components
of $\varphi_\op$ on $V\oplus \End(V)$.

\begin{align*}
\tilde\varphi_\cl=&\varphi_\cl \\
\tilde\varphi_\op(\underline a+\underline b\otimes v)=&\varphi_\op^1(\underline a)+\varphi_\op^2(\underline b)(v).
\end{align*}
Let $\varphi,\psi \in \Hom(\overline{T}^c(\su A),\su A))\times \overline{\Hom}(T^c(\su A),\End^+(V))$. If
$\gamma=\varphi\odot\psi$ exists then $\tilde\gamma=\tilde\varphi\circ\tilde\psi$, that is, on the one hand

\begin{equation*}
\begin{aligned}
\tilde\gamma_\op(\underline a+\underline b\otimes v)=& 
\tilde\varphi_\op\circ(\tilde\psi_\cl+\tilde\psi_\op)(\underline a+\underline b\otimes v)\\
=&(\varphi_\op^1\circ\psi_\cl)(\underline a)+(\varphi_\op^2\circ\psi_\cl)(\underline b)(v)+
\sum\pm\tilde\varphi_\op(a_{(1)}\tilde\psi_\op(a_{(2)}))
+\sum\pm\tilde\varphi_\op(b_{(1)}\tilde\psi_\op(b_{(2)}\otimes v)) \\
=&
(\varphi_\op^1\circ\psi_\cl)(\underline a)+
\sum\pm\varphi_\op^2(a_{(1)})(\psi_\op^1(a_{(2)})) +\\
&(\varphi_\op^2\circ\psi_\cl)(\underline b)(v)+
\sum\pm\varphi_\op^2(b_{(1)})(\psi_\op^2(b_{(2)})(v))\\
=&\gamma_\op^1(\underline a)+\gamma_\op^2(\underline b)(v).
\end{aligned}
\end{equation*}
and on the other hand
\begin{multline*}
(\widetilde{\varphi_\op\circ \psi_\cl}+\widetilde{\varphi_\op * \psi_\op})(\underline a+\underline b\otimes v)=\varphi_\op^1\circ\psi_\cl(\underline a) +(\varphi_\op^2\circ\psi_\cl)(\underline b)(v)+ (\varphi_\op*\psi_\op)^1(\underline a)+(\varphi_\op*\psi_\op)^2(\underline b)(v)\\
=\varphi_\op^1\circ\psi_\cl(\underline a) +(\varphi_\op^2\circ\psi_\cl)(\underline b)(v)+ \sum\pm(\varphi_\op(a_{(1)})\cdot\psi_\op(a_{(2)}))^1+
\sum\pm(\varphi_\op(b_{(1)})\cdot\psi_\op(b_{(2)}))^2(v)\\
=\varphi_\op^1\circ\psi_\cl(\underline a)+\sum\pm \varphi_\op^2(a_{(1)})(\psi_\op^1(a_{(2)}))
+(\varphi_\op^2\circ\psi_\cl)(\underline b)(v)+
\sum\pm\varphi_\op^2(b_{(1)})(\psi_\op^2(b_{(2)})(v))
\end{multline*}
Furthermore, one has $(\varphi\odot\psi)_\cl=\varphi_\cl\circ\psi_\cl$ by definition.
\end{proof}

\begin{nota}\label{N:def} The graded pre-Lie algebra $(\Hom(\overline{T}^c(\su A),\su A)\times \overline{\Hom}(T^c(\su A),\End^+(V)),\odot)$ is denoted  $\Def(A,V)$.

\end{nota}

\subsection{Deformations of homotopy affine actions}  In the previous section we have interpreted the  pre-Lie structure on $\Coder(\hsc_1^{\ac}(A,V))$
in terms of a pre-Lie structure in $\Def(A,V)$.
If $(A,V)$ is a pair of dgvs, then $\Coder(\hsc_1^{\ac}(A,V))$ is a dg pre-Lie algebra, with the differential induced by those of $A$ and $V$ 
and so is $\Def(A,V)$. 
We denote by $\partial$ this differential. Namely $\partial(\varphi_\cl,\varphi_\op)=((\partial\varphi)_\cl,(\partial\varphi)_\op)$
with

\begin{equation*}
\begin{aligned}
(\partial\varphi)_\cl=&d_{\su A}\varphi_\cl-(-1)^{|\varphi_\cl|} \varphi_\cl\circ d_{\su A} \\
(\partial\varphi)_\op=&d_{\End^+(V)}\varphi_\op-(-1)^{|\varphi_\op|} \varphi_\op\circ d_{\su A}.
\end{aligned}
\end{equation*}

\begin{rem} Let $\dd\in\Def(A,V)$ be defined by
$\dd_\cl=d_{\su A}$ and $\dd_\op=d_V$, that is, $\dd_\op$ is the map such that $\dd_\op(1_{\kfield})=(0,d_V)$
and such that $\dd_\op((\su A)^{\otimes n})=0$, for $n\geq 1$.

Then
$$\dd\odot \dd=0 \text{ and } \partial \varphi= \dd\odot\varphi-(-1)^{|\varphi|} \varphi\odot\dd=[\dd,\varphi].$$
\end{rem}

\begin{defn}  Let $(A,V)$ be a pair of dgvs.  A {\sl Maurer Cartan element} $\varphi=(\varphi_\cl,\varphi_\op)$ in 
$\Def_{-1}(A,V)$, is an element such that $\varphi_\cl(\su A)=0$ and $\varphi_\op(\kfield)=0$,
and such that
$$\partial\varphi+\varphi\odot\varphi=(\dd+\varphi)\odot(\dd+\varphi)=0.$$
\end{defn}
As a consequence, if $\varphi$ is a Maurer Cartan element, given $\psi$ in $\Def(A,V)$, then 
\begin{equation}\label{E:fond}
\begin{aligned}
([\dd+\varphi,\psi])_\op=&(\dd+\varphi)\odot\psi-(-1)^{|\psi|}\psi\odot(\dd+\varphi) \\
=&d_{\End^+(V)}\psi_\op-(-1)^{|\psi]}\psi_\op\circ (d_{\su A}+\varphi_\cl)+\varphi_\op\circ\psi_\cl +\varphi_\op*\psi_\op-(-1)^{|\psi|}\psi_\op*\varphi_\op
\end{aligned}
\end{equation}
induces a differential on $\Def(A,V)$ denoted $\delta_{\varphi}$.

\begin{defn} Let $(A,V)$ be a pair of dgvs. An {\sl $(\hsc_1)_\infty$-algebra structure} on $(A,V)$ is given by
a Maurer Cartan element $\varphi$ in $\Def_{-1}(A,V)$. The complex $(\Def(A,V),\delta_\varphi)$ is called the {\sl deformation complex} of
the $(\hsc_1)_\infty$-algebra $(A,V;\varphi)$.
\end{defn}

\begin{prop}\label{P:maurercartan} The pair of dgvs $(A,V)$ is endowed with an $(\hsc_1)_\infty$-algebra structure if and only if $A$ is an $\Ainf$-algebra  endowed with an $\Ainf$-morphism
from $A$ to the differential graded associative algebra $\End^+(V)$.
\end{prop}

\begin{proof} Let $\varphi$ be the Maurer Cartan element in $\Def_{-1}(A,V)$ defining the $(\hsc_1)_\infty$-algebra
structure on $(A,V)$.  Then
$\varphi=(\varphi_\cl,\varphi_\op)$ with $\varphi_\cl\in\Hom(\overline{T}^c(\su A),\su A)$ and $\varphi_\op\in\overline{\Hom}(T^c(\su A),\End^+(V))$.
Equation
$\partial\varphi +\varphi\odot\varphi=0$ and the definition \eqref{E:prelieproduct} of the pre-Lie product give
\begin{equation*}
\begin{aligned}
d_{\su A}\circ \varphi_\cl+\varphi_\cl\circ d_{\su A}+\varphi_\cl\circ\varphi_\cl=&0 \\
d_{\End^+(V)}\varphi_\op+\varphi_\op\circ d_{\su A}+\varphi_\op*\varphi_\op+\varphi_\op\circ\varphi_\cl=&0.
\end{aligned}
\end{equation*}
The first relation says that $\varphi_\cl$ is a Maurer Cartan element in $\Def_{-1}(A)$ yielding an $\Ainf$-structure on $A$.

Recall that if $(R,\mu,d_R)$ is a differential graded associative algebra, then $(T^c(sR),\partial)$ is a dg coalgebra where the coderivation is induced by
$\partial_1(\su r)=-\su d_R(r)$ and $\su^{-1}\partial_2(\su\otimes\su)=\mu$.
Let $R$ be the  differential graded associative $\End^+(V)$ and define $F:\overline{T}^c(\su A)\rightarrow \overline{T}^c(\su R)$ as the unique morphism of coalgebras extending
$\su\varphi_\op$.
In particular,  $F$ composed with the projection onto $\su R$ is $\su \varphi_\op$ and $F$ composed with the projection
onto $(\su R)^{\otimes 2}$ is $(\su\varphi_\op\otimes\su\varphi_\op)\Delta'$ where $\Delta'$ is the deconcatenation product of $\overline{T}^c(\su A)$.
Therefore, since $\varphi_\op$ has degree $-1$:
$$\su(\varphi_\op*\varphi_\op)=\su \mu(\varphi_\op\otimes\varphi_\op)\Delta'=
\partial_2(\su\otimes\su)(\varphi_\op\otimes\varphi_\op)\Delta'=-\partial_2(\su\varphi_\op\otimes\su\varphi_\op)\Delta'
=-\partial_2F.$$

Hence $(\partial\varphi+\varphi\odot\varphi)_\op=0$  writes

\begin{multline*}
\su d_{\End^+(V)}\varphi_\op+(\su \varphi_\op)\circ d_{\su A}+
\su( \varphi_\op*\varphi_\op)+(\su \varphi_\op)\circ\varphi_\cl=\\
-\partial_1(\su \varphi_\op) -\partial_2(F)+F\circ (\varphi_\cl+d_{\su A})=-(\partial_1+\partial_2)(F)+F\circ (\varphi_\cl+d_{\su A})=0.
\end{multline*}
which amounts to say that $F$ is an $\Ainf$-morphism.
\end{proof}

\begin{rem} We recover that if $(A,V)$ is an $\hsc_1$-algebra, then the Maurer-Cartan element $\varphi$ is such that
$\varphi_\cl$ vanishes except on $(\su A)^{\otimes 2}$ while
 $\varphi_\op$ vanishes except on $\su A$, hence the $\Ainf$-morphism $F$ reduces to a morphism of differential graded associative algebras.\end{rem}

\subsection{Modules over {$\Ainf$}-algebras and Hochschild cochain complex} 
In this section we refer to \cite{ALRWZ} and  \cite{GetJon90}
for the theory of representations of a $\mathcal P_\infty$-algebra and more specifically for representations of $\Ainf$-algebras.

In this section  $A$ is an $\Ainf$-algebra and $\varphi_\cl:\overline{T}^c(\su A)\rightarrow \su A$ is its structure morphism, that is, 
$(d_{\su A}+\varphi_\cl)\circ (d_{\su A}+\varphi_\cl)=0$.

\begin{defn} A dgvs $(N,d_N)$ is said to be an {\it $\Ainf$-$A$-module} (or $A$-bimodule in the terminology of Getzler and Jones) if there exists a map
$$K: T^c(\su A;\su N):=T^c(\su A) \otimes \su N\otimes T^c(\su A)\rightarrow \su N$$ of degree $-1$ such that $K(\su N)=0$, and such that 
the induced coderivation $\delta_{K+d_{sN}}$ of $\overline{T}^c(sA)$-dg comodule squares to $0$.
\end{defn}

\begin{prop}{\cite[Proposition 3.4]{GetJon90}} If $F$ is an $\Ainf$-morphism from $A_1$ to $A_2$ then $A_2$ is an $\Ainf$-$A_1$-module.
\end{prop}

The structure goes as follows: if $\tilde F$ denotes the dg coalgebra map $T^c(\su A_1)\rightarrow T^c(\su A_2)$ induced by $F$, 
and $\varphi_2$ denotes the $\Ainf$-structure on $A_2$, then $K$ is obtained as the composite:
\begin{equation}\label{E:K}
K=\varphi_{2}(\tilde F\otimes 1_{\su A_2} \otimes \tilde F)
\end{equation}

\begin{prop}\label{P:casparticulierK} If $F$ is an $\Ainf$-morphism from an $\Ainf$-algebra $A$ to a differential graded associative  algebra $(R,\mu,d_R)$, 
then the structure of $\Ainf$-$A$-module over $R$ takes the following form: for $\underline a, \underline b\in \overline{T}^c(\su A)$
and $r\in \su R$

\begin{equation*}
\begin{aligned}
K(\underline a\otimes r)=&\partial_2(F(\underline a),r) \\
K(m\otimes\underline b)=&\partial_2(r,F(\underline b)) \\
K(\underline a\otimes r\otimes \underline b)=&0,
\end{aligned}
\end{equation*}
where $\partial_2(\su\otimes\su)=\su\mu.$

\end{prop}

\begin{proof}Since $K=\varphi_R(\tilde F\otimes 1_{\su R}\otimes \tilde F)$ with $\tilde F$ the coalgebra morphism $T^c(\su A)\rightarrow T^c(\su R)$
induced by $F$
and since $\varphi_R=\partial_2$, then $K((\su A)^{\otimes n}\otimes \su R\otimes (\su A)^{\otimes m})=0$ if $nm\not= 0$ and
Equation \eqref{E:K} gives the result.
\end{proof}

\begin{defn} Let $(A,\varphi_\cl,d_A)$ be an $\Ainf$-algebra and $(N,K,d_N)$ be an $\Ainf$-$A$-module. The Hochschild cochain complex
of $A$ with coefficients in $N$ is
$$CH(A;N)=(\Hom(T^c(\su A),\su N),\dd_{H})$$
where $$\dd_H(\psi)=(K+d_{\su N})\circ \psi-(-1)^{|\psi|}\psi\circ(\varphi_\cl+d_A)$$

\end{defn}

\begin{rem} The Hochschild cochain complex so defined coincide (up to a regrading) with the usual Hochschild cochain complex, that is,
\begin{enumerate}
\item  If $A$ is an associative algebra and $N$ is a bimodule, then $CH(A;N)$ is the Hochschild cochain complex of $A$ with coefficients in $N$.
Indeed for $n\in\Z$,  
$$CH^n(A;N)=(\Hom_{n}(T^c(\su A),\su N)=\Hom(A^{\otimes (1-n)},N).$$
\item If $A$ is an $\Ainf$-algebra then $A$ is an $\Ainf$-$A$-module and $CH(A;A)$ is the  Hochschild cochain complex of the $\Ainf$-algebra $A$.
We denote by $\overline{CH}(A;A)=(\Hom(\overline{T}^c(\su A),\su A),\dd_{H})$ the reduced Hochschild cochain complex.
\end{enumerate}
\end{rem}

\subsection{The deformation complex of an $(\hsc_1)_\infty$-algebra}\label{S:defcomplex} 

Recall from Proposition \ref{P:maurercartan} that an $(\hsc_1)_\infty$-algebra structure on $(A,V)$ amounts to
\begin{enumerate}
\item $(A,\varphi_\cl, d_A)$ is an $\Ainf$-algebra and $\varphi_\cl:\overline{T}^c(\su A)\rightarrow \su A$ is the structure map.
\item $R=\End^+(V)$ is a differential graded associative  algebra with differential $d_{\End^+(V)}$ and product $\mu$. The associated $\Ainf$-structure map $\partial_2:\overline{T}^c(\su R)\rightarrow \su R$ is defined by  $\partial_2(\su\otimes\su)=\su\mu$.
\item There is an $\Ainf$-map from $A$ to $R$ defined by $F:T^c(\su A)\rightarrow \su R$ and we denote by $\varphi_\op=s^{-1} F$
the open part of the Maurer-Cartan element defining the $(\hsc_1)_\infty$-structure.
\end{enumerate}

Moreover  Proposition \ref{P:casparticulierK} gives the  $\Ainf$-$A$-module structure on $R$. In this context, one has

\begin{prop}  The Hochschild cochain complex $(CH(A,\End^+(V)),\dd_H)$ takes the following form 
$$\dd_H(\su \alpha)=-\su (d_{\End^+(V)}\alpha-(-1)^{|\alpha|}\alpha\circ (\varphi_\cl+d_{\su A})+\varphi_\op*\alpha-
(-1)^{|\alpha|}\alpha*\varphi_\op).$$
\end{prop}

\begin{proof} Since $\End^+(V)$ is a differential graded associative algebra one has from Proposition \ref{P:casparticulierK}

\begin{equation}\label{E:KenF}
K\circ\psi=\partial_2(F\otimes\psi)\Delta'+\partial_2(\psi\otimes F)\Delta'
\end{equation}

Let us compute
\begin{multline*}
\dd_H(\su\alpha)=(K+d_{\su R})\circ \su\alpha-(-1)^{|\su\alpha|}\su\alpha\circ(\varphi_\cl+d_{\su A})=\\
-\su (d_R\alpha-(-1)^{|\alpha|}\alpha\circ (\varphi_\cl+d_{\su A}))+\partial_2(\su\varphi_\op\otimes\su\alpha)\Delta'+
\partial_2(\su\alpha\otimes \su\varphi_\op)\Delta'=\\
-\su (d_R\alpha-(-1)^{|\alpha|}\alpha\circ (\varphi_\cl+d_{\su A}))-\partial_2(\su\otimes\su)(\varphi_\op\otimes\alpha)\Delta'+
(-1)^{|\alpha|}\partial_2(\su\otimes\su)(\alpha\otimes\varphi_\op)\Delta'=\\
-\su (d_R\alpha-(-1)^{|\alpha|}\alpha\circ (\varphi_\cl+d_{\su A})+\varphi_\op*\alpha-
(-1)^{|\alpha|}\alpha*\varphi_\op).
\end{multline*}
\end{proof}

\begin{nota} Let $ (\overline{CH}(A,\End^+(V)),\dd_H)$ be the subcomplex of $(CH(A,\End^+(V)),\dd_H)$ formed by the maps $\psi:T^c(sA)\rightarrow
\su\End^+(V)$ such that $\psi(\kfield)\subset \su\End(V)$.
\end{nota}

\begin{thm}\label{T:mapcone} Let $(A,V)$ be an $(\hsc_1)_\infty$-algebra. The deformation complex $\Def(A,V)$ is the mapping cone of the map
\begin{align*}
\hat F:(\overline{CH}(A;A),\dd_H)&\rightarrow (\overline{CH}(A,\End^+(V)),\dd_H)\\
 \psi_\cl &\mapsto   F\circ\psi_\cl
\end{align*}
where $F$ is the $\Ainf$-map from $A$ to $\End^+(V)$.
\end{thm}

\begin{proof}Recall from Notation \ref{N:def} that $\Def(A,V)=\overline{CH}(A,A)\times \su^{-1}\overline{CH}(A,\End^+(V))$.
Recall also that the differential on $\Def(A,V)$ is given by
$$\delta_\varphi\psi=[\dd+\varphi,\psi]$$
where $\varphi=(\varphi_\cl,\varphi_\op)$
with $\varphi_\cl$  the $\Ainf$-structure on $A$ and $\su\varphi_\op$  the $\Ainf$-map from $A$ to $\End^+(V)$.
Let $\psi=(\psi_\cl,\su^{-1}\psi_\op)$ be  in $\Def(A,V)$.
It is clear that $(\delta_\varphi\psi)_\cl=d_H\psi_\cl$. Computing $(\delta_\varphi\psi)_\op$ with Relations \eqref{E:fond} gives

\begin{align*}
[\dd+\varphi,\psi]_\op=&d_{\End^+(V)}\su^{-1}\psi_\op-(-1)^{|\psi]}\su^{-1}\psi_\op\circ (d_{\su A}+\varphi_\cl)\\
&+\varphi_\op\circ\psi_\cl +\varphi_\op*\su^{-1}\psi_\op-(-1)^{|\psi|}\su^{-1}\psi_\op*\varphi_\op\\
=&-\su^{-1}\dd_H(\psi_\op)+\su^{-1}F\circ \psi_\cl.
\end{align*}
As a consequence, the differential in $\Def(A,V)$ coincides with the one of the mapping cone construction of the map $\hat F$, providing $\hat F$ is a morphism in $\dgvs$.  Since  $(A,V)$ is an $(\hsc_1)_\infty$-algebra, then  $\partial_\varphi^2=0$ and $\hat F$ is a morphism in $\dgvs$.
\end{proof}

\begin{rem} As pointed out by Kontsevich and Soibelman in \cite{KS00} in the case of $\As_\infty$-algebra, Theorem \ref{T:mapcone} states that  the deformation of an $(\hsc_1)_\infty$-algebra $(A,V)$ is controlled by the truncated complex
of the mapping cone of the map $\hat F$ from the Hochschild complex of $A$ to the Hochschild complex of $A$ with coefficients in $\End^+(V)$.
\end{rem}

\subsection{On the deformation complex of an operad under $(\hsc_1)_\infty$}\label{S:defcomplexop}

Let $\cal P$ be a non-symmetric differential graded $\{\cl,\op\}$-operad $\cal P$, together with an operad morphism
\[\eta:(\hsc_1)_{\infty}\rightarrow \cal P.\]
From \cite{LodVal}, there is a sequence of  isomorphisms in $\dgvs$
\[\Hom_{dgop}((\hsc_1)_\infty,\cal P)=  \Hom_{dgop}(\Omega(\hsc_1^{\ac}),\cal P)\simeq Tw(\hsc_1^{\ac},\cal P),\]
where a twisting cochain is a map $\alpha$ of degree $-1$ satisfying $d_{\cal P}\alpha+\alpha\circ \alpha=0$.
The description of the cooperad $\hsc_1^{\ac}$ in Equation (\ref{D:sc1dual}) implies that
$$\Def(\cal P,\eta)=\Hom(\hsc_1^{\ac},\cal P)=\prod_{n\geq 1} s^{1-n} \cal P(\cl^{\times n};\cl)
\prod_{m\geq 1} s^{-m} \cal P(\cl^{\times m};\op)\prod_{p\geq 0}s^{-p} \cal P(\cl^{\times p},\op;\op)$$
is a graded pre-Lie algebra, where the pre-Lie product is induced by the one defined in Section \ref{S:preLie}.
Indeed, with the notation of Section \ref{suspension}, one has 
$$\Def(\cal P,\eta)=\Hom(\hsc_1^{\ac},\cal P)=\prod_{n\geq 1}\Lambda_\cl \cal P(\cl^{\times n};\cl)
\prod_{m\geq 1} \Lambda_\cl \cal P(\cl^{\times m};\op)\prod_{p\geq 0} \Lambda_\cl \cal P(\cl^{\times p},\op;\op).$$

Let us write $\Def(\cal P,\eta)$ as $\Dc\times \Do$ where 

$$\Dc=\prod_{n\geq 1} \Lambda_\cl \cal P(\cl^{\times n};\cl) \text{ and } \Do=
\prod_{m\geq 1} \Lambda_\cl \cal P(\cl^{\times m};\op)\prod_{p\geq 0}\Lambda_\cl  \cal P(\cl^{\times p},\op;\op)$$

Furthermore the operad morphism $\eta: (\hsc_1)_{\infty}\rightarrow \cal P$ yields a twisting element, 
also denoted $\eta=(\eta_\cl,\eta_\op)$, in $
\Def(\cal P,\eta)_{-1}$ such that
$d_{\cal P}(\eta)+\eta\circ\eta =0$. This equality decomposes into
\begin{alignat*}{2}
d_{\cal P}(\eta_\cl)+\eta_\cl\circ\eta_\cl &=0 \\
d_{\cal P}(\eta_\op)+\eta_\op\circ\eta_\cl+\eta_\op\circ\eta_\op &=0 \\
\end{alignat*}
As a consequence, $\delta_\eta=d_{\cal P}+[\eta,-]$ defines a differential on $\Def(\cal P,\eta)$ and
$(\Def(\cal P,\eta),\delta_\eta)$ is called the deformation complex of $\cal P$. For $\psi=(\psi_\cl,\psi_\op)\in \Def(\cal P,\eta)$, one gets
\[\delta_\eta(\psi)=(d_{\cal P}\psi_\cl+[\eta_\cl,\psi_\cl], d_{\cal P}\psi_\op+[\eta,\psi_\op]+\eta_\op\circ\psi_\cl).\]

It is clear that the closed part of the deformation complex has for differential $\delta_{\eta_\cl}$ and so coincides with the Hochschild cochain complex of the operad $\mathcal P(-;\cl)$ 
under the operad $\Ainf$ (a generalization of the Hochschild cochain complex of a multiplicative operad defined by V. Turchin in \cite{Tu05})

\begin{thm}\label{T:coneoperad} The open part $\Do$ of the deformation complex of the operad $\cal P$ is endowed with the following differential: $\d_H(\psi_\op)=d_\mathcal P \psi_\op+[\eta,\psi_\op].$
As a consequence, the deformation complex of the operad $\cal P$ is the mapping cone of the map
$$\begin{array}{cccc}
F:& (\Dc,\delta_{\eta_\cl})&\rightarrow & (\su\Do,-\su \d_H\su^{-1}) \\
& \psi_\cl & \mapsto &\su \eta_\op\circ\psi_\cl
\end{array}$$
\end{thm}

\begin{proof} Let us prove that $d_H$ is a differential and that $F$ commutes with the differentials.
We already know that $(\delta_\eta)^2=0$ which implies that for any $(\psi_\cl,\psi_\op)$, one has
$$\d_H(\d_H(\psi_\op))+d_H(\eta_\op\circ\psi_\cl)+\eta_\op\circ(\delta_{\eta_\cl}\psi_\cl)=0.$$
Since it is true for any pair, we can choose $\psi_\cl=0$ which implies that $d_H^2=0$ or we can choose $\psi_\op=0$ which implies
that $F$ is a differential graded morphism.
\end{proof}

\begin{rem} In case $(A,V)$ is an $(\hsc_1)_\infty$-algebra, the operad of endomorphism of the $2$-colored dgvs $(A,V)$ is an operad under
$(\hsc_1)_\infty$ and we recover the deformation complex of the previous section.
\end{rem}

\section{An operad acting on the deformation complex}

In this section we build the $\{\cl,\op\}$-colored operad $\rBr$ called the relative brace operad and prove that it acts on the deformation complex of an operad $\mathcal P$ under $(\hsc_1)_\infty$. Before going through the description by planar trees, let us explain algebras over it.

An algebra $(A,R)$ over $\rBr$
is a dgvs where $A$ is a $B_\infty$-algebra (in the sense of Baues in \cite{Baues81}) of a certain kind (it is a Brace algebra in the sense of
Dolgushev-Willwacher in \cite{DolWil15}), $R$ is a differential graded algebra and $A$ acts on $R$.
More precisely the pair $(\mathbf BA=T^c(\su A),\mathbf BR=T^c(\su R))$ is a differential graded $\hsc_1$-algebra in the category of free conilpotent coassociative coalgebras,
with some extra properties. Namely an $\rBr$-algebra $(A,R)$ amounts to the following data:

$\bullet$ $(\mathbf BA,M,\delta)$ is a differential graded bialgebra, where the multiplication $M$, is uniquely determined by degree $0$ maps
$M_{p,q}: (\su A)^{\otimes p}\otimes (\su A)^{\otimes q}\rightarrow \su A$; we ask that $M_{p,q}=0$ for $p\geq 2$, that is, the family of maps restricts to the family
$M_{1,q}:\su A\otimes   (\su A)^{\otimes q}\rightarrow \su A$ and  $M_{1,0}=M_{0,1}=Id$ and $M_{0,q}=0, q\not=1$.

$\bullet$ $(\mathbf BR,\partial)$ is a differential graded coalgebra such that $\partial=d_{\su R}+\partial_2$, that is, $R$ is a differential graded associative algebra.

$\bullet$ $\mathbf BR$ is a right differential graded $\mathbf BA$-module in the category of dg-coalgebras with structure map $K=\mathbf BR\otimes\mathbf BA\rightarrow \mathbf BR$. This structure is uniquely determined by degree $0$ maps
$K_{p,q}:  (\su R)^{\otimes p}\otimes (\su A)^{\otimes q}\rightarrow \su R$; we ask that $K_{p,q}=0$ for $p\not\in\{0,1\}$, that is, the family of maps restricts to the families

$$\begin{array}{ccccccc}
K_{1,q}:& \su R&\otimes &  (\su A)^{\otimes q}&\rightarrow&  \su R, & \text{ with } K_{1,0}=Id \\
K_{0,q}: &\kfield&\otimes& (\su A)^{\otimes q}&\rightarrow& \su R, &\text{ with } K_{0,0}=0.
\end{array}$$

The action of the algebra $\mathbf BA$ reads, $\forall  r\in \mathbf BR,\forall a_1,a_2\in \mathbf BA$
\begin{equation*}
\begin{aligned}
K(r,M(a_1,a_2))=&K(K(r,a_1),a_2)
\qquad \partial K=&K(\partial\otimes Id+Id\otimes \delta)&
\end{aligned}
\end{equation*}

$\bullet$ Notice that the family $(K_{0,q})_{q\geq 1}$ gives rise to a map of dgcoalgebras $K_0:\mathbf BA\rightarrow \mathbf BR$, that is, $K_0$ defines an $\Ainf$-map from $A$ to $R$, satisfying in addition
$$K_0(M(a_1,a_2))=K(K_0(a_1),a_2),\forall a_1,a_2\in \mathbf BA.$$

$\bullet$ In the sequel we will make use of the maps 
$G_n=\su^{-1} K_{0,n}: (\su A)^{\otimes n}  \rightarrow  R$ of degree $-1$ and
$\Gamma_{1,q}= R\otimes (\su A)^{\otimes q} \rightarrow R$ of degree $0$
 defined by
$$\Gamma_{1,q}(r\otimes\underline a)=\su^{-1} K_{1,q}(\su r\otimes\underline a).$$

In order to avoid signs in the construction we will describe the operad $\Lambda_\cl\rBr$ so that 
an $\rBr$-structure on the pair of dgvs $(A,R)$ is equivalent to a $\Lambda_\cl\rBr$-structure on the pair of dgvs $(\su A,R)$.

\subsection{Definition of the operad of relative braces $\rBr$}\label{sec: RBr def}

\newcommand{\brect}[1]{\draw [black,fill=black] ($#1+(-0.1,-0.1)$) rectangle ($#1+(.1,.1)$)}
\newcommand{\wrect}[1]{\draw [black,fill=white] ($#1+(-0.1,-0.1)$) rectangle ($#1+(.1,.1)$)}
\newcommand{\wrond}[1]{\draw [black,fill=white] #1 circle (0.1);}
\newcommand{\brond}[1]{\draw [black,fill=black] #1 circle (0.1);}

The $\{\cl,\op\}$-colored operad $\Lambda_\cl\rBr$ is described as the free $\Z$-module generated by a set $\mathcal T$ of planar rooted trees and words of planar rooted trees. 
Except for the root edge, every edge of the trees considered has two adjacent vertices (there is no half-edges).
Trees are represented with the root at the bottom. Reading a tree from top to bottom, a vertex has ingoing edges, called {\sl inputs} of the vertex and one outgoing edge 
called {\sl the output} of the vertex. The set $In(v)$ of ingoing edges of a vertex $v$ is totally ordered, by reading the tree from left to right.
The {\sl leaves} of a tree are its extremal vertices,
in particular leaves have no ingoing edges.

  Reading a planar tree $T$ from bottom to top and left to right gives a total order on the vertices of the tree $T$, called the {\sl canonical order} of the 
vertices. Assume $v_1<\ldots <v_n$ is the canonical order in the tree $T$. Later on we will define degrees of vertices. 
The sub-$\Z$-module of the free graded commutative $\Z$-algebra $S(v_1,\ldots,v_n)$ 
generated by polynomials containing each $v_i$ exactly once is free of rank 1 and denoted $\det(T)$; $\epsilon_T=v_1\wedge\ldots\wedge v_n\in\det(T)$ 
is our choice of generator. By abuse of notation, we will often write $\epsilon_T$ as $w_1\wedge\ldots \wedge w_s$ where $w_1<\ldots<w_s$ are the vertices of degree $\not=0$.
Similarly if $v_i$ has degrree $0$, $\epsilon_{T\setminus\{v_i\}}$ corresponds to $v_1\wedge\ldots \wedge\hat{v_i}\wedge\ldots\wedge v_n$.

\subsubsection{On vertices}\label{S: vertex}
There are different kind of vertices:
the round-shaped vertices  model operations having a closed output, whereas the square-shaped vertices model operations having  an open output.
Moreover:

\def\labelitemi{-}
\begin{itemize}
\item The neutral round-shaped vertex \begin{tikzpicture}
 \brond{(0,0)};
 \end{tikzpicture}\
may have $n\geq 2$ inputs and model the operations $\delta_n:(\su A)^{\otimes n}\rightarrow \su A$. It has degree $-1$.
\item The  round-shaped vertex  \begin{tikzpicture}
 \wrond{(1,0)};
 \end{tikzpicture}\ is called a closed vertex and is labeled. It has degree $0$. A corolla of $n+1$ closed vertices labeled by $\{1,\ldots,n+1\}$
 model the operations $M_{1,n}:\su A\otimes (\su A)^{\otimes n}\rightarrow \su A$.
 \item The neutral square-shaped vertex \begin{tikzpicture}
 \brect{(0,0)};
\end{tikzpicture}\ may have $n\geq 1$ inputs and model the operations $G_n:(\su A)^{\otimes n}\rightarrow R$. It has degree $-1$.
\item  The square-shaped vertex  \begin{tikzpicture}
 \wrect{(2,0)};
 \end{tikzpicture}\ is called an open vertex and is labeled. It has  degree $0$.
 The corolla having  $n$  closed leaves and an open root 
 model the operation $\Gamma_{1,n}:R\otimes (\su A)^{\otimes n}\rightarrow R$.
  \end{itemize}

\subsubsection{The closed part of the sets $\mathcal T$} Let $I=(s_1,\ldots,s_n)$ be a sequence in $\{\cl,\op\}$ and $x\in\{\cl,\op\}$. The $\Z$-module 
$\Lambda_\cl\rBr(I;x)$ is freely generated by the set $\mathcal T(I;x)$.

For $x=\cl$, if  there exists $i$ such that $s_i=\op$ then this set is empty. Otherwise $\mathcal T(\cl^{\times n};\cl)$
is the set of  planar trees with round-shaped vertices, having $n$  closed vertices  labeled by $\{1,\ldots,n\}$ and we denote by $\beta_T:
\{\text{closed\ vertices}\}\rightarrow \{1,\ldots,n\}$ the bijection defining the labeling.
By definition 
\[
\Lambda_\cl\rBr(\cl^{\times n};\cl)=\bigoplus_{T\in\mathcal T(\cl^{\times n};\cl)} \Z T\otimes \det(T), 
\]
is a graded $\Z$-module, the grading being given by $\det(T)$.

\subsubsection{The open part of the sets $\mathcal T$}
As in \cite{Q-SCMRL} we associate to the sequence $I=(s_1,\ldots,s_n)$ the labels $(l_1,\ldots,l_n)$ defined by
$l_i=i$ if $s_i=\cl$ and $l_i=\underline i$ if $s_i=\op$. We denote by $L_\cl$ the set $\{l_i|s_i=\cl\}$ and by $L_\op$ the set $\{l_i|s_i=\op\}$.

An element in $\mathcal T(I;\op)$ is a word of planar rooted trees $W=T_1\cdots T_r$,
where the trees $T_j$ have a square-shaped root, either neutral or open and labeled. There is a bijection between the set of open roots and $L_\op$
and the open roots are labeled according to this bijection.
The other vertices are round-shaped vertices. There is a bijection between the set of closed vertices and $L_\cl$
and the closed vertices are labeled accordingly. Let $\beta_W:\{\text{closed\ vertices}\}\sqcup \{\text{open\ vertices}\}
\rightarrow L_\cl\sqcup L_\op$ denote the bijection. By definition $\det(W)=\det(T_1)\otimes\ldots\otimes\det(T_r)$ and
\[
\Lambda_\cl\rBr(I;\op)=\bigoplus_{W\in\mathcal T(I;\op)} \Z W\otimes \det(W), 
\]
is a graded $\Z$-module, the grading being given by $\det(W)$. By abuse of notation we say that a word of trees $W$ has degree $-p$ if it has $p$ neutral vertices.

\subsubsection{Action of the symmetric group} Given $\sigma\in\Sigma_n$, let us describe the action map
$\cal T(s_1,\ldots,s_n;x)\rightarrow \cal T(s_{\sigma(1)},\ldots,s_{\sigma(n)};x)$. Let $L^\sigma$ denote the labels associated to the sequence
$(s_{\sigma(1)},\ldots,s_{\sigma(n)})$. Hence
$\sigma$ determines a bijection
$$\begin{array}{cccc}
\rho_\sigma: &L_\cl\sqcup L_\op&\rightarrow& L^\sigma_\cl\sqcup L^\sigma_\op\\
&i\in L_\cl&\mapsto  & \sigma^{-1}(i) \in L^\sigma_\cl \\
&\underline i \in L_\op & \mapsto & \underline{ \sigma^{-1}(i)} \in L^\sigma_\op
\end{array}$$
For $W\in \cal T(s_1,\ldots,s_n;x)$ the element $W\cdot\sigma\in  \cal T(s_{\sigma(1)},\ldots,s_{\sigma(n)};x)$ is the same tree or word of trees
than $W$ with the labeling: $\beta_{W\cdot\sigma}=\rho_\sigma\beta_W$.

\newcommand{\arbremultsq}[4]{
  \draw [-] #1 -- ($#1 +(0,-.5)$);
 \draw [-] ($#1 +(-2,1)$) -- #1 -- ($#1 +(-1,1)$) ;
 \draw [-] ($#1 +(2,1)$) -- #1 -- ($#1 +(1,1)$) ;
  \draw [black,fill=white] ($#1 +(-2,1)$) circle (0.1)  ;
  \draw [black,fill=white] ($#1 +(-1,1)$) circle (0.1)  ;
  \draw [black,fill=white] ($#1 +(1,1)$) circle (0.1)  ;
  \draw [black,fill=white] ($#1 +(2,1)$) circle (0.1)  ;
  #3{#1};
  \draw ($#1 +(-2,1)$) node [above] {$2$};
  \draw ($#1 +(-1,1)$) node [above] {$3$};
  \draw ($#1 +(1,1)$) node [above] {$k$};
  \draw ($#1 +(2,1)$) node [above] {$k+1$};
  \draw ($#1 +(0,0.6)$) node [above] {$\dots$};
  \draw ($#1 -(0.9,0.4)$) node {#2};
    \draw ($#1 +(0.1,-0.3)$) node [right] {#4};
}

\newcommand{\arbremultsqbis}[4]{
  \draw [-] #1 -- ($#1 +(0,-.5)$);
 \draw [-] ($#1 +(-2,1)$) -- #1 -- ($#1 +(-1,1)$) ;
 \draw [-] ($#1 +(2,1)$) -- #1 -- ($#1 +(1,1)$) ;
  \draw [black,fill=white] ($#1 +(-2,1)$) circle (0.1)  ;
  \draw [black,fill=white] ($#1 +(-1,1)$) circle (0.1)  ;
  \draw [black,fill=white] ($#1 +(1,1)$) circle (0.1)  ;
  \draw [black,fill=white] ($#1 +(2,1)$) circle (0.1)  ;
  #3{#1};
  \draw ($#1 +(-2,1)$) node [above] {$1$};
  \draw ($#1 +(-1,1)$) node [above] {$2$};
  \draw ($#1 +(1,1)$) node [above] {$k-1$};
  \draw ($#1 +(2,1)$) node [above] {$k$};
  \draw ($#1 +(0,0.6)$) node [above] {$\dots$};
  \draw ($#1 -(0.9,0.4)$) node {#2};
    \draw ($#1 +(0.1,-0.3)$) node [right] {#4};
}

\newcommand{\arbrecsq}[3]{
 \draw [-] #1+(-0.4,0.9) -- ($#1+(-0.4,0.5)$);
  \draw [-] #1+(0.4,0.9) -- ($#1+(0.4,0.5)$);
 \wrect{($#1+(-0.4,1)$)};
 \wrect{($#1+(0.4,1)$)};
  \draw ($#1+(-0.4,1)$) node [above] {$#2$};
  \draw ($#1+(0.4,1)$) node [above] {$#3$};
}

\begin{figure}
 \begin{center}
\begin{tikzpicture}[>=stealth, arrow/.style={->,shorten >=3pt}, point/.style={coordinate}, pointille/.style={draw=red, top color=white, bottom color=red},scale=0.8]
 \coordinate (A) at (-5,0);
 \coordinate (B) at (0.3,0);
 \coordinate (C) at (0.3,-4);
 \coordinate (D) at (-5,-4);
 \coordinate (E) at (5,0);
 \coordinate (F) at (5,-4);
 \arbremultsqbis{(A)}{$\partial_k$}{\brond}{};
 \arbremultsqbis{(B)}{$G_k$}{\brect}{};
  \draw [-] (-5,3) -- (-5,2.5);
  \draw [black,fill=white] (-5,3) circle (0.1)  ;
  \draw (-5.7,2.3) node [above] {$id_c$};
  \draw (-5,3.1) node [above] {$1$};
  \draw [-] (5,3) -- (5,2.5);
  \draw [black,fill=white] (4.9,2.9) rectangle (5.1,3.1)  ;
  \draw (4.2,2.3) node [above] {$id_\op$};
  \draw (5,3.1) node [above] {$\un{1}$};
 \arbremultsq{(C)}{$\Gamma_{1,k}$}{\wrect}{$\un{1}$};
\arbremultsq{(D)}{$M_{1,k}$}{\wrond}{$1$};
  \arbrecsq{(E)}{\underline{1}}{\underline{2}};
  \draw (4.1,-0.2) node [above] {$\mu_{\op}$};
   \draw [black,-] (-7.4,-5) rectangle (-2.3,4) ;
  \draw [black,-] (-2.1,-5) rectangle (6.4,4) ;
  \draw [black,-] (3.6,-1) rectangle (6.3,3.9) ;
  \draw (-5,-6) node [above] {Closed part};
  \draw (2.5,-6) node [above] {Non-closed part};
  \draw (5,-2) node [above] {Open part};
  \draw [-] (-5,-5.3) -- (-5,-5);
  \draw [-] (2.5,-5.3) -- (2.5,-5);
  \draw [-] (5,-1.3) -- (5,-1);
\end{tikzpicture}
\end{center}\caption{Generators of $\rBr$.}\label{fig: generators}
\end{figure}

\subsubsection{The operad structure} We follow closely the definition of the minimal operad $M$ of Kontsevich and Soibelman in \cite{KS00}, except for the signs.
To a tree $T$ in $\mathcal T(I;x)$ we associate the ordered set of its angles  $A(T)$. To a word $W=T_1\cdots T_r$ of rooted trees the ordered set of its angles $A(W)$
is the ordered set $A(T_1)\cup\ldots\cup A(T_r)$ where for $x\in A(T_i)$ and $y\in A(T_j),$ one has $i<j\implies x<y$.

Let $W\in\cal T(s_1,\ldots,s_n;x)$ and $X\in\cal T(J;s_i)$; let $v$ be the vertex in $W$ labeled by $l_i$;  the operadic composition in $\rBr$ takes the following form
\[
(W\otimes \epsilon_W)\circ_v (X\otimes \epsilon_X)=\sum_{f:In(v)\rightarrow A(X)} (W\circ_v^f X)\otimes \epsilon_{W\setminus\{v\}}\wedge \epsilon_X,
\]
where $f$ runs over  nondecreasing maps from $In(v)\rightarrow A(X)$. The word  $W\circ_v^f X$ is obtained from $W$ and $X$ by substituting $X$ for the vertex $v\in W$
and by attaching the inputs of $v$ to the vertices of $X$ according to the function $f$.

Notice that if $X=U_1\cdots U_q$ is a word of trees, then necessarily the vertex $v$ of $W$ is a square-shaped vertex, thus  a root. Writing 
 $W=T_1\ldots T_r$ where $v$ is the root of $T_j$ for some $j$,  we can then replace the root $v$ by the word $U_1\cdots U_q$.

Below we give  examples of composition. In the first example we don't precise
the signs of the trees since there is only one neutral vertex. In the second example the vertices $e$ and $f$ are the only ones of degree $-1$.

\newcommand{\arbrel}[3]{
 \draw [-] ($#1+(-1,1)$) -- #1 -- ($#1+(1,1)$);
 \draw [-] #1 -- ($#1-(0,.5)$);
 \brect{#1};  
 \draw [black,fill=white] ($#1+(-1,1)$) circle (0.1);
 \draw [black,fill=white] ($#1+(1,1)$) circle (0.1);
  \draw ($#1+(-1,1)$) node [right] {$#2$};
  \draw ($#1+(1,1)$) node [above] {$#3$};
}
\newcommand{\arbrer}[3]{
 \draw [-] ($#1+(-1,1)$) -- #1 -- ($#1+(1,1)$);
 \draw [-] #1 -- ($#1-(0,.5)$);
 \brect{#1};  
 \draw [black,fill=white] ($#1+(-1,1)$) circle (0.1);
 \draw [black,fill=white] ($#1+(1,1)$) circle (0.1);
  \draw ($#1+(-1,1)$) node [above] {$#2$};
  \draw ($#1+(1,1)$) node [left] {$#3$};
}
\newcommand{\arbrec}[3]{
 \draw [-] ($#1+(-1,1)$) -- #1 -- ($#1+(1,1)$);
 \draw [-] #1 -- ($#1-(0,.5)$);
 \brect{#1};  
 \draw [black,fill=white] ($#1+(-1,1)$) circle (0.1);
 \draw [black,fill=white] ($#1+(1,1)$) circle (0.1);
  \draw ($#1+(-1,1)$) node [above] {$#2$};
  \draw ($#1+(1,1)$) node [above] {$#3$};
}
\newcommand{\segm}[2]{ 
  \draw [-] #1 -- ($#1+(0,1)$);
  \draw [black,fill=white] ($#1+(0,1)$) circle (0.1);
  \draw ($#1+(0,1)$) node [above] {$#2$};
  }

 \begin{center}
\begin{tikzpicture}[scale=0.7]
 \coordinate (A) at (-1,0);
  \draw [-] (-1,-.5) -- (A) -- (-1,1);
  \draw [black,fill=white] (-1.1,-.1) rectangle (-.9,.1) ; 
  \draw [black,fill=white] (-1,1) circle (0.1)  ;
  \draw (-1,1.1) node [above] {$2$};
  \draw (-1,0) node [left] {$\un{1}$};
\draw (0,0.5) node [left] {$\circ$};
\draw (-.25,0.3) node [right] {$\un{1}$};
  \arbrec{(1.5,0)}{1}{2};
  \draw (3.5,0.5) node [left] {$=$};
  \segm{(5,0)}{1};
  \arbrec{(5,0)}{3}{2};
  \draw (7,0.5) node [left] {$+$};
  \segm{(7,1)}{3};
  \arbrel{(8,0)}{1}{2};
  \draw (10,0.5) node [left] {$+$};
 \segm{(11,0)}{3};
  \arbrec{(11,0)}{1}{2};
  \draw (13,0.5) node [left] {$+$};
 \segm{(15,1)}{3};
  \arbrer{(14,0)}{1}{2};
  \draw (16,0.5) node [left] {$+$};
  \segm{(17,0)}{2};
  \arbrec{(17,0)}{1}{3};
\end{tikzpicture}
\end{center}
%


  \newcommand{\petitsq}[1]{
 \draw [-] ($#1+(0,-0.5)$) -- #1;
  \draw [black,fill=white] ($#1+(0.1,-0.1)$) rectangle ($#1+(-0.1,0.1)$) ; 
   \draw  #1node [left] {$\un{1}$};
  }

\newcommand{\ligne}[4]{
 \draw [-] ($#1+(0,-0.5)$) -- #1 -- ($#1+(0,1)$);
  \draw [black,fill=#2] ($#1+(0.1,-0.1)$) rectangle ($#1+(-0.1,0.1)$) ; 
   \draw [black,fill=white] ($#1+(0,1)$)  circle (0.1)   ;
    \draw  ($#1+(0,1.1)$) node [above] {$#3$};
  \draw  #1 node [left] {$#4$};}
  
  \newcommand{\ligneg}[4]{
 \draw [-] ($#1+(0,-0.5)$) -- #1 -- ($#1+(0,1)$);
  \draw [black,fill=#2] ($#1+(0.1,-0.1)$) rectangle ($#1+(-0.1,0.1)$) ; 
   \draw [black,fill=white] ($#1+(0,1)$)  circle (0.1)   ;
    \draw  ($#1+(0,1.1)$) node [left] {$#3$};
  \draw  #1 node [left] {$#4$};}
  
\newcommand{\arbreV}[3]{
 \draw [-] ($#1+(-0.5,1)$) -- #1 -- ($#1+(0.5,1)$);
 \draw [-] #1 -- ($#1-(0,.5)$);
 \brect{#1};  
 \draw [black,fill=white] ($#1+(-0.5,1)$) circle (0.1);
 \draw [black,fill=white] ($#1+(0.5,1)$) circle (0.1);
   \draw #1 node [left] {$f$};
  \draw ($#1+(-0.5,1.1)$) node [above] {$#2$};
  \draw ($#1+(0.5,1.1)$) node [above] {$#3$};
}


 \begin{center}

\begin{tikzpicture}[baseline=+2ex,scale=0.65]
  \ligne{(-2,0)}{white}{2}{\un{1}}
\ligne{(-1,0)}{black}{3}{e}
\draw (0,0.5) node [left] {$\circ$};
\draw (-.25,0.3) node [right] {$\un{1}$};
 \petitsq{(1,0)}
 \ligne{(2,0)}{black}{2}{f}
  \draw (3.5,0.5) node [left] {$=$};
  \end{tikzpicture}
 \Bigg(
  \begin{tikzpicture}[baseline=+2ex,scale=0.65]
  \ligne{(4,0)}{white}{3}{\un{1}}
\ligne{(5,0)}{black}{2}{f}
\ligne{(6,0)}{black}{4}{e}
  \draw (7,0.5) node [left] {$+$};
  \petitsq{(7.5,0)}
  \arbreV{(8.5,0)}{3}{2};
  \ligne{(9.5,0)}{black}{4}{e}
   \draw (10.5,0.5) node [left] {$+$};
 \petitsq{(11.2,0)}
 \ligneg{(12,0)}{black}{2}{f}
 \ligne{(13,0)}{black}{4}{e}
  \segm{(12,1.1)}{3};
   \draw (14,0.5) node [left] {$+$};
 \petitsq{(14.5,0)}
   \arbreV{(15.5,0)}{2}{3};
 \ligne{(16.5,0)}{black}{4}{e}
 \end{tikzpicture} \Bigg)
\begin{tikzpicture}[baseline=+2ex,scale=0.7]
  \draw (19,0.5) node [left] {$\otimes\ e\wedge f$};
\end{tikzpicture}

\end{center}
%

The fact that these compositions endow $\Lambda_\cl\rBr$ with a $\{\cl,\op\}$-colored operad structure is a consequence of the computation in \cite{KS00}.

\subsubsection{The differential} In addition {$\Lambda_\cl\rBr$}  is a differential graded operad and we describe the differential in this section.
Let $W\in\cal T(I;x)$ and let $\epsilon_W=v_1\wedge\ldots\wedge v_n$ where $\{v_1,\ldots,v_n\}$ is the set of  vertices of $W$.
The differential is described as follows
\[
d_{\Lambda_\cl\rBr}(W\otimes\epsilon_W)=\sum_{i=1}^k (-1)^{\sum_{j=1}^{i-1}|v_j|}d_{v_i}(W)\otimes (v_1\wedge \ldots v_{i-1}\wedge v_i^1\wedge v_i^2\wedge\cdots\wedge v_n),
\]
where $d_{v_i}(W)$ is a sum of word of trees depending on the nature of the vertex $v_i$ and $v_i^1$, $v_i^2$ are new vertices obtained from a blow-up of $v_i$.
We have four type of vertices to consider.
For the neutral round-shaped  vertices one has

\newcommand{\arbrepartial}[2]{
\draw ($#1 +(0.05,.85)$) node {\tiny$\cdots$};
\draw ($#1-(0,0.2)$) node [left] {\tiny$v$};
 \draw [-] ($#1+(-1,1)$) -- #1 -- ($#1+(1,1)$);
 \draw [-] ($#1+(-.5,1)$) -- #1 -- ($#1+(.5,1)$);
 \draw [-] #1 -- ($#1-(0,.5)$);
#2{#1};
\draw ($#1-(0,.5)$) node [below] {\tiny$T_0$};
\draw ($#1+(-1,1)$) node [above] {\tiny$T_1$};
  \draw ($#1+(1,1)$) node [above] {\tiny$T_k$};
}

\newcommand{\arbrepartialsq}[2]{
\draw ($#1 +(0.05,.85)$) node {\tiny$\cdots$};
\draw ($#1-(0,0.2)$) node [left] {\tiny$v$};
 \draw [-] ($#1+(-1,1)$) -- #1 -- ($#1+(1,1)$);
 \draw [-] ($#1+(-.5,1)$) -- #1 -- ($#1+(.5,1)$);
 \draw [-] #1 -- ($#1-(0,.5)$);
#2{#1};
\draw ($#1+(-1,1)$) node [above] {\tiny$T_1$};
  \draw ($#1+(1,1)$) node [above] {\tiny$T_k$};
}

\newcommand{\arbrepartialsqlibre}[5]{
\draw ($#1 +(0.05,.85)$) node {\tiny$\cdots$};
\draw ($#1-(0,0.2)$) node [left] {\tiny #5};
 \draw [-] ($#1+(-1,1)$) -- #1 -- ($#1+(1,1)$);
 \draw [-] ($#1+(-.5,1)$) -- #1 -- ($#1+(.5,1)$);
 \draw [-] #1 -- ($#1-(0,.5)$);
#2{#1};
\draw ($#1+(-1,1)$) node [above] {\tiny$T_{#3}$};
  \draw ($#1+(1,1)$) node [above] {\tiny$T_{#4}$};
}

\newcommand{\arbreblow}[3]{
\draw ($#1-(0,0.2)$) node [left] {\tiny$v^1$};
\draw ($#1 +(-0.7,.85)$) node {\tiny$\cdots$};
\draw ($#1 +(1.45,.85)$) node {\tiny$\cdots$};
 \draw [-] ($#1+(-1.5,1)$) -- #1 -- ($#1+(2.2,1)$);
 \draw [-] ($#1+(-.2,1)$) -- #1 -- ($#1+(.5,1)$);
 \draw [-] #1 -- ($#1-(0,.5)$);
\coordinate (A) at ($#1+(.5,1)$);
\draw ($(A)-(0,0.2)$) node [right] {\tiny$v^2$};
\draw [-] ($(A)+(-1,1)$) -- (A) -- ($(A)+(1,1)$);
 \draw [-] ($(A)+(-.5,1)$) -- (A) -- ($(A)+(.5,1)$);
 \draw ($(A) +(0.05,.85)$) node {\tiny$\cdots$};
 \draw ($#1-(0,.5)$) node [below] {\tiny$T_0$};
 \draw ($#1+(-1.5,1)$) node [above] {\tiny$T_1$};
 \draw ($#1+(-.2,.9)$) node [above] {\tiny$T_{i-1}$};
  \draw ($#1+(2.2,1)$) node [above] {\tiny$T_k$};
  \draw ($(A)+(-1,1)$) node [above] {\tiny$T_{i}$};
  \draw ($(A)+(1,1)$) node [above] {\tiny$T_{i+j}$};
  #2{#1};
  #3{(A)};
}

\newcommand{\arbreblowcyc}[3]{
\draw ($#1-(0,0.2)$) node [left] {\tiny$v^1$};
\draw ($#1 +(-0.7,.85)$) node {\tiny$\cdots$};
\draw ($#1 +(1.45,.85)$) node {\tiny$\cdots$};
 \draw [-] ($#1+(-1.5,1)$) -- #1 -- ($#1+(2.2,1)$);
 \draw [-] ($#1+(-.2,1)$) -- #1 -- ($#1+(.5,1)$);
 \draw [-] #1 -- ($#1-(0,.5)$);
\coordinate (A) at ($#1+(.5,1)$);
\draw ($(A)-(0,0.2)$) node [right] {\tiny$v^2$};
\draw [-] ($(A)+(-1,1)$) -- (A) -- ($(A)+(1,1)$);
 \draw [-] ($(A)+(-.5,1)$) -- (A) -- ($(A)+(.5,1)$);
 \draw ($(A) +(0.05,.85)$) node {\tiny$\cdots$};
 \draw ($#1-(0,.5)$) node [below] {\tiny$T_0$};
 \draw ($#1+(-1.5,1)$) node [above] {\tiny$T_1$};
 \draw ($#1+(-.2,.9)$) node [above] {\tiny$T_{i}$};
  \draw ($#1+(2.2,1)$) node [above] {\tiny$T_k$};
  \draw ($(A)+(-1,1)$) node [above] {\tiny$T_{i+1}$};
  \draw ($(A)+(1,1)$) node [above] {\tiny$T_{j-1}$};
  #2{#1};
  #3{(A)};
}

\newcommand{\arbreblowsq}[3]{
\draw ($#1-(0,0.2)$) node [left] {\tiny$v^1$};
\draw ($#1 +(-0.7,.85)$) node {\tiny$\cdots$};
\draw ($#1 +(1.45,.85)$) node {\tiny$\cdots$};
 \draw [-] ($#1+(-1.5,1)$) -- #1 -- ($#1+(2.2,1)$);
 \draw [-] ($#1+(-.2,1)$) -- #1 -- ($#1+(.5,1)$);
 \draw [-] #1 -- ($#1-(0,.5)$);
\coordinate (A) at ($#1+(.5,1)$);
\draw ($(A)-(0,0.2)$) node [right] {\tiny$v^2$};
\draw [-] ($(A)+(-1,1)$) -- (A) -- ($(A)+(1,1)$);
 \draw [-] ($(A)+(-.5,1)$) -- (A) -- ($(A)+(.5,1)$);
 \draw ($(A) +(0.05,.85)$) node {\tiny$\cdots$};
 \draw ($#1+(-1.5,1)$) node [above] {\tiny$T_1$};
 \draw ($#1+(-.2,.9)$) node [above] {\tiny$T_{i-1}$};
  \draw ($#1+(2.2,1)$) node [above] {\tiny$T_k$};
  \draw ($(A)+(-1,1)$) node [above] {\tiny$T_{i}$};
  \draw ($(A)+(1,1)$) node [above] {\tiny$T_{i+j}$};
  #2{#1};
  #3{(A)};
}

\newcommand{\arbrebloww}[3]{
\coordinate (A) at ($#1+(-1.5,1)$);
\coordinate (B) at ($#1+(1.5,1)$);
\draw ($(A) +(0.05,.85)$) node {\tiny$\cdots$};
\draw ($(A)-(0,0.2)$) node [left] {\tiny$v^1$};
 \draw [-] ($(A)+(-1,1)$) -- (A) -- ($(A)+(1,1)$);
 \draw [-] ($(A)+(-.5,1)$) -- (A) -- ($(A)+(.5,1)$);
 \draw ($(B) +(0.05,.85)$) node {\tiny$\cdots$};
\draw ($(B)-(0,0.2)$) node [right] {\tiny$v^2$};
 \draw [-] ($(B)+(-1,1)$) -- (B) -- ($(B)+(1,1)$);
 \draw [-] ($(B)+(-.5,1)$) -- (B) -- ($(B)+(.5,1)$);
 \draw [-] (A) -- #1 -- (B);
 \draw [-] #1 -- ($#1- (0,.5)$);
 \draw ($#1-(0,.5)$) node [below] {\tiny$T_0$};
\draw ($(A)+(-1,1)$) node [above] {\tiny$T_1$};
  \draw ($(A)+(1,1)$) node [above] {\tiny$T_i$};
  \draw ($(B)+(-1,1)$) node [above] {\tiny$T_{i+1}$};
  \draw ($(B)+(1,1)$) node [above] {\tiny$T_k$};
  #2{(A)};
#3{(B)};
\brond{#1};
}
\newcommand{\arbreblowu}[3]{
\draw ($#1-(0,0.2)$) node [left] {\tiny$v^1$};
\draw ($#1 +(-0.7,.85)$) node {\tiny$\cdots$};
\draw ($#1 +(1.45,.85)$) node {\tiny$\cdots$};
 \draw [-] ($#1+(-1.5,1)$) -- #1 -- ($#1+(2.2,1)$);
 \draw [-] ($#1+(-.4,1)$) -- #1 -- ($#1+(.1,1)$);
 \draw [-] ($#1+(.9,1)$) -- #1 -- ($#1-(0,.5)$);
\coordinate (A) at ($#1+(.1,1)$);
 \draw ($(A)$) node [above] {\tiny$v^2$};
 \draw ($#1-(0,.5)$) node [below] {\tiny$T_0$};
 \draw ($#1+(-1.5,1)$) node [above] {\tiny$T_1$};
 \draw ($#1+(-.4,.9)$) node [above] {\tiny$T_i$};
 \draw ($#1+(.9,.9)$) node [above] {\tiny$T_{i+1}$};
  \draw ($#1+(2.2,1)$) node [above] {\tiny$T_k$};
  #2{#1};
  #3{(A)};
}

\begin{equation}\label{deltaroundneutral}
d_v\Bigg(
\begin{tikzpicture}[baseline=-0.65ex,scale=.8]
 \arbrepartial{(0,0)}{\brond};
 \end{tikzpicture}
 \Bigg)
=-\sum_{\substack{1\leq i<i+j\leq k\\ j\leq k-2}}
\begin{tikzpicture}[baseline=-0.65ex,scale=.8]
 \arbreblow{(5,0)}{\brond}{\brond}; 
\end{tikzpicture} 
\end{equation}

For the closed vertices one has

\begin{multline}\label{deltaclosed}
d_v\Bigg(
\begin{tikzpicture}[baseline=-0.65ex,scale=0.7]
 \arbrepartial{(0,0)}{\wrond};
 \end{tikzpicture}
 \Bigg)
= -\sum_{\substack{1\leq i+1,\\ j-1\leq k,\\ i<j}}
\begin{tikzpicture}[baseline=-0.65ex,scale=0.7]
 \arbreblowcyc{(5,0)}{\brond}{\wrond};
\end{tikzpicture}
+\sum_{{1\leq i<i+j\leq k}}
\begin{tikzpicture}[baseline=-0.65ex,scale=0.7]
 \arbreblow{(5,0)}{\wrond}{\brond};
\end{tikzpicture}.
\end{multline}

For the open vertices one has

\begin{multline}\label{deltaopen}
d_v\Bigg(
\begin{tikzpicture}[baseline=-0.65ex,scale=0.7]
 \arbrepartialsq{(0,0)}{\wrect};
 \end{tikzpicture}
 \Bigg)
=+\sum_{1\leq i<i+j\leq k}
\begin{tikzpicture}[baseline=-0.65ex,scale=0.7]
 \arbreblowsq{(5,0)}{\wrect}{\brond};
\end{tikzpicture}\\
+
\sum_{0\leq  i< k}
\begin{tikzpicture}[baseline=-0.65ex,scale=0.7]
 \arbrepartialsqlibre{(5,0)}{\wrect}{1}{i}{$v^1$};
  \arbrepartialsqlibre{(8,0)}{\brect}{i+1}{k}{$v^2$};
\end{tikzpicture}
-\sum_{0< i\leq k}
\begin{tikzpicture}[baseline=-0.65ex,scale=0.7]
 \arbrepartialsqlibre{(5,0)}{\brect}{1}{i}{$v^1$};
  \arbrepartialsqlibre{(8,0)}{\wrect}{i+1}{k}{$v^2$};
\end{tikzpicture}
\end{multline}

and 
\begin{multline}\label{deltaneutralsquare}
d_v\Bigg(
\begin{tikzpicture}[baseline=-0.65ex,scale=0.7]
 \arbrepartialsq{(0,0)}{\brect};
 \end{tikzpicture}
 \Bigg)
=-\sum_{1\leq i<i+j\leq k}
\begin{tikzpicture}[baseline=-0.65ex,scale=0.65]
 \arbreblowsq{(5,0)}{\brect}{\brond};
\end{tikzpicture}
-\sum_{0< i<k}
\begin{tikzpicture}[baseline=-0.65ex,scale=0.65]
 \arbrepartialsqlibre{(5,0)}{\brect}{1}{i}{$v^1$};
  \arbrepartialsqlibre{(8,0)}{\brect}{i+1}{k}{$v^2$};
\end{tikzpicture}
\end{multline}

We refer to Appendix \ref{A:A} for details on the generating operations, their differential and compositions.

For later use, let us introduce the following notation for these blow-ups. 
Let $v$ be a vertex of a tree $T$. 
For $v$ a neutral round-shaped vertex, one writes 
\begin{equation}\label{eq: notation bu nrs} 
d_v(T)= -\sum_{\substack{1\leq i<i+j\leq k\\ j\leq k-2}}  d_{v}^{[i,i+j]}(T), \tag{\ref{deltaroundneutral}b}
\end{equation}
for the blow-up \eqref{deltaroundneutral}.
For $v$ a closed vertex, one writes 
\begin{equation}\label{eq: notation bu cl}
d_v(T)=  -\sum_{\substack{i,j\in\{0,\ldots,k\},\\ j<j+1<\ldots<k<0<1<\ldots< i}} d_{v}^{[j,i]}(T)+ \sum_{{1\leq i<i+j\leq k}}d_{v}^{[i,i+j]}(T), \tag{\ref{deltaclosed}b}
\end{equation}
for the blow-up \eqref{deltaclosed},  where the interval $[j,i]$ is thought of as an interval of a cyclic ordered set, as in Definition \ref{def: strict interval}.
For $v$ an open vertex, one writes 
\begin{equation}\label{eq: notation bu op}
d_v(T)= \sum_{1\leq i<i+j\leq k} d_{v}^{[i,i+j]}(T)
+ \sum_{0\leq  i< k} d_{v}^{[i+1,\text{right}]}(T)  -\sum_{0< i\leq k} d_{v}^{[\text{left},i]}(T), \tag{\ref{deltaopen}b}
\end{equation}
for the blow-up \eqref{deltaopen}. 
For $v$ a neutral square-shaped  vertex, one writes 
\begin{equation}\label{eq: notation bu opcl}
d_v(T)= -\sum_{1\leq i<i+j\leq k} d_{v}^{[i,i+j]}(T)
 -\sum_{1\leq  i< k} d_{v}^{[\lft,i]}(T), \tag{\ref{deltaneutralsquare}b}
\end{equation}
for the blow-up \eqref{deltaneutralsquare}.

\begin{Convention}\label{conv:rbr} The two operads $\Lambda_\cl\rBr$ and $\rBr$  only differ by gradings and signs in the composition and differential. By abuse of notation, the generators $\partial_k$, $G_k$, $M_{1,k}$ and $\Gamma_{1,k}$ will be also considered as generators in $\rBr$ of degree $k-2$,  $k-1$, $k$ and $k$ respectively.
\end{Convention}

\subsection{The deformation complex as an $\rBr$-algebra}

\begin{thm}\label{T:action} Let $\eta: (\hsc_1)_\infty\rightarrow \mathcal P$ be an operad morphism.  
The operad $\rBr$ acts on the deformation complex $\Def(\cal P,\eta)$. 
\end{thm}

\begin{proof} The deformation complex $\Def(\cal P,\eta)$ writes as $\Dc\times \Do$ with 

$$\Dc=\prod_{n\geq 1} \Lambda_\cl \cal P(\cl^{\times n};\cl) \text{ and }\Do=
\prod_{m\geq 1} \Lambda_\cl \cal P(\cl^{\times m};\op)\prod_{p\geq 0}\Lambda_\cl  \cal P(\cl^{\times p},\op;\op)$$

We denote by $\eta:(\hsc_1)_\infty\rightarrow \cal P$ the operad map as well as the twisting element in $\Def(\cal P,\eta)$.
We recall that $\Dc$ is endowed with the differential $\delta_{\eta_\cl}=d_{\cal P}+[\eta_\cl,-]$

Indeed we will describe an action of $\Lambda_\cl\rBr$ on $\Def(\cal P,\eta)$.

\medskip

\noindent {\it First step: the closed part}

Since $\Lambda_\cl \cal P(-;\cl)$ is an operad we know from Gerstenhaber and Voronov in \cite{GerVor95} that $\Dc$ is endowed with brace operations
\begin{align*}
M_{1,n}:\Dc\otimes \Dc^{\otimes n} & \rightarrow \Dc\\
x\otimes y_1\otimes\ldots\otimes y_n &\mapsto  x\{y_1,\ldots,y_n\}
\end{align*}
defining an associative product which is a morphism of coalgebras
\begin{align*}
M:T^c(\Dc)\otimes T^c(\Dc)&\rightarrow T^c(\Dc)\\
 a\otimes b & \mapsto  a* b
\end{align*}
Notice that 
$$M_{1,1}(a,b)=a\{b\}=a\circ b.$$

The differential of the operad $\Lambda_c\cal P$ can be extended as a coderivation of the coalgebra  $T^c(\Dc)$, denoted by $d_{\cal P}$.
We claim that the map
$$\delta_\cl:T^c(\Dc)\rightarrow T^c(\Dc)$$
defined by
$$\delta_\cl(a)=d_{\cal P} a+\eta_\cl * a-(-1)^{|a|} a*\eta_\cl$$
is a coderivation of cofree coalgebras, satisfies $\delta_\cl^2=0$ so that $(T^c(\Dc),M,\delta_\cl)$ is a differential graded bialgebra.

The ingredients to prove that claim are the following. If $a,b\in\Dc$ then $a*b=a\circ b+a\otimes b+(-1)^{|a||b|} b\otimes a$. Since $\eta_\cl\in\Dc$ is of degree $-1$ the relation 
$d_{\cal P}\eta_\cl+\eta_\cl\circ\eta_\cl=0$ writes $d_{\cal P}\eta_\cl+\eta_\cl*\eta_\cl=0$. This proves that $\delta_\cl^2=0$.
Furthermore $\eta_\cl$ is a primitive element in $T^c(\Dc)$ which implies that $\delta_\cl$ is a coderivation.
Finally the associativity of the product $M$ implies that $\delta$ is a derivation with respect to $M$.

Notice that is $a\in\Dc$ then $\delta_\cl(a)=d_{\cal P} a+\eta_\cl\circ a-(-1)^{|a|} a\circ\eta_\cl=\delta_{\eta_\cl}(a)$.

\medskip

\noindent{\it Second step: the open part}

The definition of the product $\circ$ given in Section \ref{S:preLie} for a colored operad implies that given two elements
$a,b$ in $\Do$ the product $a\circ b$ is indeed associative, because it consists in grafting $b$ into the only, if exists, open part of $a$.

In Theorem \ref{T:coneoperad} we have proved that $\Do$ is endowed with the differential
$$d_H(a)=d_\mathcal Pa-(-1)^{|a|} a\circ\eta_\cl + \eta_\op\circ a-(-1)^{|a|} a\circ \eta_\op.$$
A short computation shows that $d_H(a\circ b)=d_H(a)\circ b+(-1)^{|a|} a\circ d_H(b)$ making $\Do$ into a dg associative algebra.

We denote by $\delta_\op$ the induced square zero coderivation on $T^c(\su \Do)$.

The degree $0$ maps

\begin{align*}
K_{1,q}: \su \Do\otimes   (\Dc)^{\otimes q}&\rightarrow  \su \Do,  \\
K_{0,q}:  (\Dc)^{\otimes q}&\rightarrow \su \Do, 
\end{align*}

are obtained from the degree $0$ maps
\begin{align*}
\Gamma_{1,q}:  \Do\otimes   (\Dc)^{\otimes q}&\rightarrow   \Do \\
 x \otimes y_1\otimes\ldots\otimes y_q &\mapsto  x\{y_1,\ldots,y_q\}
\end{align*}
which are Gerstenhaber-Voronov brace operations
and the degree $-1$ maps
\begin{align*}
G_n:   (\Dc)^{\otimes n}&\rightarrow   \Do \\
 y_1\otimes\ldots\otimes y_n &\mapsto  \eta_\op\{y_1,\ldots,y_n\}
\end{align*}
That the induced morphism of coalgebras
$$K:T^c(\su \Do)\otimes T^c(\Dc)\rightarrow T^c(\su\Do)$$
satisfies
$$K(a,M(b,c))=K(K(a,b),c))$$
is a consequence of the relations of the brace operations.
The fact that $K$ commutes with the differentials is a computation similar to the one described in Appendix \ref{A:A}.
\end{proof}

The following corolla is a direct consequence of the computation of the deformation complex of an $(\hsc_1)_\infty$-algebra in Section
\ref{S:defcomplex} together with the fact that $\rBr$ is an $\SC_2$-operad (section 5) so that $H_*(\rBr)=\hsc_2$. 

\begin{cor} Let $(A,V)$ be an $\hsc_1$-algebra. The pair $(CH^*(A,A),CH^*(A,\End(V)^+)$ is an $\rBr$-algebra. At the level of homology this action gives the following $\hsc_2$-structure on $(HH^*(A,A),HH^*(A,\End(V)^+)$:
\begin{enumerate}
\item $HH^*(A,A)$ has its usual Gerstenhaber structure.
\item $HH^*(A,\End(V)^+)$ inherits the associative structure of $\End(V)^+$
\item $\varphi_*: HH^*(A,A)\rightarrow HH^*(A,\End(V)^+)$ is the central morphism of algebras induced by the morphism $\varphi:A\rightarrow \End(V)^+$.
\end{enumerate}
\end{cor}

\begin{rem}\label{R:DTT} Notice that if we start with an associative algebra $A$ acting on itself on the right, it induces a morphism of algebras $\alpha: A\rightarrow \End(A)^{op}$. In that case Theorem \ref{T:action} holds, that is, the pair $\Def_{\SC}(A)=(CH^*(A,A),CH^*(A,\End(A)^{op}))$ is an $\rBr$-algebra. If $A$ is a unital algebra then the homology of $\Def_{\SC}(A)$ is the pair $(HH^*(A,A),A)$. It is then
an $\hsc_2$-algebra where the central morphism $f:HH^*(A,A)\rightarrow A$ is given by the inclusion $Z(A)\rightarrow A$.
We get then a lift of the action considered by D. Dolgushev, D. Tamarkin and B. Tsygan in \cite{DTT11} at the level the deformation complex $\Def_{\SC}(A)$.
The cochain considered in \cite{DTT11} is  the pair $(CH^*(A,A),A)$, which gives the same homology.

Hence, our approach is totally different in spirit. Here  we consider (right) braces operations
\[ \Gamma_{1,k}: CH^*(A,\End(A)^{op})\otimes CH^*(A,A)^{\otimes k}\rightarrow CH^*(A,\End(A)^{op})\]
whereas in \cite{DTT11} the authors consider (left) braces operations
\[ CH^*(A,A)\otimes A^{\otimes k}\rightarrow A\]
modelling the action of the operad $\End_A$ onto $A$.
Clearly, our chain complex is not an algebra over their $\SC_2$-operad and their chain complex is not an algebra over
$\rBr$. Hence our solution of the Deligne's Swiss-cheese conjecture is different from \cite{DTT11}.

\end{rem}

\section{The Operad $ \rBr $  is an $ {\rm SC}_2 $ operad}

 \newcommand{\zero}[1]{
\coordinate (A) at (0,0);
 \draw [-] ($(A)+ (0,-.5)$) -- (A);
  #1{(A)};
}
 \newcommand{\une}[2]{
\coordinate (A) at (0,0);
 \draw [-] ($(A)+ (0,-.5)$) -- (A);
 \draw [-] ($(A)+(0,1)$)-- (A);
 \draw ($(A)+(0,1)$) node [above] {\tiny$#2$};
  #1{(A)};
}
\newcommand{\deux}[3]{
\coordinate (A) at (0,0);
 \draw [-] ($(A)+ (0,-.5)$) -- (A);
 \draw [-] ($(A)+(-.35,1)$)-- (A);
 \draw [-] ($(A)+(0.35,1)$)-- (A);
 \draw ($(A)+(-.35,1)$) node [above] {\tiny$#2$};
 \draw ($(A)+(.35,1)$) node [above] {\tiny$#3$};
  #1{(A)};
}
\newcommand{\trois}[4]{
\coordinate (A) at (0,0);
 \draw [-] ($(A)+ (0,-.5)$) -- (A);
 \draw [-] ($(A)+(-.5,1)$)-- (A);
 \draw [-] ($(A)+(0,1)$)-- (A);
 \draw [-] ($(A)+(0.5,1)$)-- (A);
 \draw ($(A)+(-.5,1)$) node [above] {\tiny$#2$};
 \draw ($(A)+(0,1)$) node [above] {\tiny$#3$};
 \draw ($(A)+(.5,1)$) node [above] {\tiny$#4$};
  #1{(A)};
}
\newcommand{\quatre}[5]{
\coordinate (A) at (0,0);
 \draw [-] ($(A)+ (0,-.5)$) -- (A);
 \draw [-] ($(A)+(-1,1)$)-- (A);
 \draw [-] ($(A)+(-.35,1)$)-- (A);
 \draw [-] ($(A)+(0.35,1)$)-- (A);
 \draw [-] ($(A)+(1,1)$)-- (A);
 \draw ($(A)+(-1,1)$) node [above] {\tiny$#2$};
 \draw ($(A)+(-.35,1)$) node [above] {\tiny$#3$};
 \draw ($(A)+(.35,1)$) node [above] {\tiny$#4$};
 \draw ($(A)+(1,1)$) node [above] {\tiny$#5$};
  #1{(A)};
}
\newcommand{\cinq}[6]{
\coordinate (A) at (0,0);
 \draw [-] ($(A)+ (0,-.5)$) -- (A);
 \draw [-] ($(A)+(-1,1)$)-- (A);
 \draw [-] ($(A)+(-.5,1)$)-- (A);
 \draw [-] ($(A)+(0,1)$)-- (A);
 \draw [-] ($(A)+(0.5,1)$)-- (A);
 \draw [-] ($(A)+(1,1)$)-- (A);
 \draw ($(A)+(-1,1)$) node [above] {\tiny$#2$};
 \draw ($(A)+(-.5,1)$) node [above] {\tiny$#3$};
 \draw ($(A)+(0,1)$) node [above] {\tiny$#4$};
 \draw ($(A)+(.5,1)$) node [above] {\tiny$#5$};
 \draw ($(A)+(1,1)$) node [above] {\tiny$#6$};
  #1{(A)};
}
\newcommand{\deuxdeux}[4]{
\coordinate (A) at (0,0);
 \draw [-] ($(A)+ (0,-.5)$) -- (A);
 \draw [-] ($(A)+(-.35,1)$)-- (A);
 \draw [-] ($(A)+(0.35,1)$)-- (A);
 \draw [-] ($(A)+(0.35,1)+(-.35,.5)$) -- ($(A)+(0.35,1)$)-- ($(A)+(0.35,1)+(.35,.5)$);
 \draw ($(A)+(-.35,1)$) node [above] {\tiny$#2$};
 \draw ($(A)+(0.35,1)+(-.35,.5)$) node [above] {\tiny$#3$};
 \draw ($(A)+(0.35,1)+(.35,.5)$) node [above] {\tiny$#4$};
  #1{(A)};
}
\newcommand{\unevert}[2]{
\coordinate (A) at #2;
 \draw [-] ($(A)+ (0,-.5)$) -- (A);
 \draw [-] ($(A)+(0,1)$)-- (A);
  #1{(A)};
}
\newcommand{\deuxvert}[4]{
\coordinate (A) at #4;
 \draw [-] ($(A)+ (0,-.5)$) -- (A);
 \draw [-] ($(A)+(-1,1)$)-- (A);
 \draw [-] ($(A)+(0,1)$)-- (A);
 \draw [-] ($(A)+(1,1)$)-- (A);
 \draw ($(A)+(-1,1)$) node [above] {\tiny$#2$};
 \draw ($(A)+(1,1)$) node [above] {\tiny$#3$};
  #1{(A)};
}
\newcommand{\cinqvert}[6]{
\coordinate (A) at #6;
 \draw [-] ($(A)+ (0,-.5)$) -- (A);
 \draw [-] ($(A)+(-1.5,.5)$)-- (A);
 \draw [-] ($(A)+(-.75,.5)$)-- (A);
 \draw [-] ($(A)+(0,1)$)-- (A);
 \draw [-] ($(A)+(0.75,.5)$)-- (A);
 \draw [-] ($(A)+(1.5,.5)$)-- (A);
 \draw ($(A)+(-1.5,.5)$) node [above] {\tiny$#2$};
 \draw ($(A)+(-.75,.5)$) node [above] {\tiny$#3$};
 \draw ($(A)+(.75,.5)$) node [above] {\tiny$#4$};
 \draw ($(A)+(1.5,.5)$) node [above] {\tiny$#5$};
  #1{(A)};
}
\newcommand{\quatreleft}[6]{
\coordinate (A) at (0,0);
 \draw [-] ($(A)+ (0,-.5)$) -- (A);
 \draw [-] ($(A)+(-1,1)$)-- (A);
 \draw [-] ($(A)+(-.5,.5)$)-- ($(A)+(-.5,.5)+(0,.5)$);
 \draw [-] ($(A)+(-.5,.5)$)-- ($(A)+(-.5,.5)+(.5,.5)$);
 \draw [-] ($(A)+(1,1)$)-- (A);
 \draw ($(A)+(-1,1)$) node [above] {\tiny$#2$};
 \draw ($(A)+(-.5,.5)+(0,.5)$) node [above] {\tiny$#3$};
 \draw ($(A)+(-.5,.5)+(.5,.5)$) node [above] {\tiny$#4$};
 \draw ($(A)+(1,1)$) node [above] {\tiny$#5$};
  #1{(A)};
  #6{($(A)+(-.5,.5)$)};
}
\newcommand{\quatreright}[6]{
\coordinate (A) at (0,0);
 \draw [-] ($(A)+ (0,-.5)$) -- (A);
 \draw [-] ($(A)+(-1,1)$)-- (A);
 \draw [-] ($(A)+(.5,.5)$)-- ($(A)+(.5,.5)+(-0,.5)$);
 \draw [-] ($(A)+(.5,.5)$)-- ($(A)+(.5,.5)+(-.5,.5)$);
 \draw [-] ($(A)+(1,1)$)-- (A);
 \draw ($(A)+(-1,1)$) node [above] {\tiny$#2$};
 \draw ($(A)+(.5,.5)+(-0,.5)$) node [above] {\tiny$#4$};
 \draw ($(A)+(.5,.5)+(-.5,.5)$) node [above] {\tiny$#3$};
 \draw ($(A)+(1,1)$) node [above] {\tiny$#5$};
  #1{(A)};
  #6{($(A)+(.5,.5)$)};
}
 
\newcommand{\undeuxun}[4]{
\coordinate (A) at (0,0);
 \draw [-] ($(A)+ (0,-.5)$) -- (A);
 \draw [-] ($(A)+(0,1)$)-- (A);
 \draw (A) node [left] {\tiny$#3$};
 \draw ($(A)+(0,1)$) node [above] {\tiny$#4$};
  #1{(A)};
  #2{($(A)+(0,1)$)};
}
\newcommand{\undeuxdeux}[6]{
\coordinate (A) at (0,0);
 \draw [-] ($(A)+ (0,-.5)$) -- (A);
 \draw [-] ($(A)+(0,.5)$)-- (A);
 \draw [-] ($(A)+(0,.5)+(-0.5,.5)$)-- ($(A)+(0,.5)$)-- ($(A)+(0,.5)+(0.5,.5)$);
 \draw [-] ($(A)+(0,.5)+(.5,.5)+(-0.5,.5)$)-- ($(A)+(0,.5)+(.5,.5)$)-- ($(A)+(0,.5)+(.5,.5)+(0.5,.5)$);
 \draw ($(A)+(0,.5)+(-0.5,.5)$) node [above] {\tiny$#4$};
 \draw ($(A)+(0,.5)+(.5,.5)+(-0.5,.5)$) node [above] {\tiny$#5$};
 \draw ($(A)+(0,.5)+(.5,.5)+(0.5,.5)$) node [above] {\tiny$#6$};
  #1{(A)};
  #2{($(A)+(0,.5)$)};
  #3{($(A)+(0,.5)+(0.5,.5)$)};
}

\newcommand{\boxa}[2]{
\coordinate (A) at #1;
 \draw [black,fill=white] ($(A)+(-.5,-.25)$) rectangle ($(A)+(.5,.25)$) ;
 \draw ($(A)$) node  {\tiny #2};
}

\newcommand{\vertex}[2]{
\draw ($#1 +(0.05,.85)$) node {\tiny$\cdots$};
\draw ($#1-(0,0.2)$) node [left] {\small$v$};
 \draw [-] ($#1+(-1,1)$) -- #1 -- ($#1+(1,1)$);
 \draw [-] ($#1+(-.5,1)$) -- #1 -- ($#1+(.5,1)$);
 \draw [-] #1 -- ($#1-(0,.5)$);
#2{#1};
\draw ($#1+(-1,.85)$) node [above] {\small$1$};
  \draw ($#1+(1,.85)$) node [above] {\small$k$};
}
\newcommand{\arbrecrien}[4]{
 \draw [-] ($#1+(-1,1)$) -- #1 -- ($#1+(1,1)$);
 \brond{#1};  
 \draw ($#1-(0,.5)$) node [below] {\tiny$#2$};
\draw ($#1+(-1,1)$) node [above] {\tiny$#3$};
\draw ($#1+(1,1)$) node [above] {\tiny$#4$};
}
\newcommand{\arbrecrienrootl}[4]{
 \draw [-] ($#1+(-1,1)$) -- #1 -- ($#1+(1,1)$);
  \draw [-] #1 -- ($#1-(0,.5)$);
 \brond{#1};  
 \brond{($#1+(1,1)$)}; 
 \draw ($#1-(0,.5)$) node [below] {\tiny$#2$};
\draw ($#1+(-1,1)$) node [above] {\tiny$#3$};
\draw ($#1+(1,1)$) node [above] {\tiny$#4$};
}
\newcommand{\arbrecrienrootr}[4]{
 \draw [-] ($#1+(-1,1)$) -- #1 -- ($#1+(1,1)$);
  \draw [-] #1 -- ($#1-(0,.5)$);
 \brond{#1};  
 \brond{($#1+(-1,1)$)};
 \draw ($#1-(0,.5)$) node [below] {\tiny$#2$};
\draw ($#1+(-1,1)$) node [above] {\tiny$#3$};
\draw ($#1+(1,1)$) node [above] {\tiny$#4$};
}

\newcommand{\arbrepartialll}[2]{
\draw ($#1 +(0.05,.85)$) node {\tiny$\cdots$};
\draw ($#1-(0,0.2)$) node [right] {\tiny$v_2$};
 \draw [-] #1 -- ($#1+(1,1)$);
 \draw [-] ($#1+(-.5,1)$) -- #1 -- ($#1+(.5,1)$);
 \draw [-] #1 -- ($#1-(0,1.5)$);
#2{#1};
\draw ($#1-(0,1.5)$) node [below] {\tiny$T_0$};
\draw ($#1+(-.5,1)$) node [above] {\tiny$T_2$};
  \draw ($#1+(1,1)$) node [above] {\tiny$T_k$};
}
\newcommand{\arbrepartialrr}[2]{
\draw ($#1 +(0.05,.85)$) node {\tiny$\cdots$};
\draw ($#1-(0,0.2)$) node [left] {\tiny$v_2$};
 \draw [-] ($#1+(-1,1)$) -- #1 ;
 \draw [-] ($#1+(-.5,1)$) -- #1 -- ($#1+(.5,1)$);
 \draw [-] #1 -- ($#1-(0,1.5)$);
#2{#1};
\draw ($#1-(0,1.5)$) node [below] {\tiny$T_0$};
\draw ($#1+(-1,1)$) node [above] {\tiny$T_1$};
  \draw ($#1+(1,1)$) node [above] {\tiny$T_{k-1}$};
}

\newcommand{\arbreblowas}[3]{
\draw ($#1-(0,0.2)$) node [left] {\tiny$v^1$};
\draw ($#1 +(-0.7,.85)$) node {\tiny$\cdots$};
\draw ($#1 +(1.45,.85)$) node {\tiny$\cdots$};
 \draw [-] ($#1+(-1.5,1)$) -- #1 -- ($#1+(2.2,1)$);
 \draw [-] ($#1+(-.2,1)$) -- #1 -- ($#1+(.5,1)$);
 \draw [-] #1 -- ($#1-(0,.5)$);
\coordinate (A) at ($#1+(.5,1)$);
\draw ($(A)-(0,0.2)$) node [right] {\tiny$v^2$};
\draw [-] ($(A)+(-1,1)$) -- (A) -- ($(A)+(1,1)$);
 \draw ($#1-(0,.5)$) node [below] {\tiny$T_0$};
 \draw ($#1+(-1.5,1)$) node [above] {\tiny$T_1$};
 \draw ($#1+(-.2,.9)$) node [above] {\tiny$T_{i-1}$};
  \draw ($#1+(2.2,1)$) node [above] {\tiny$T_k$};
  \draw ($(A)+(-1,1)$) node [above] {\tiny$T_{i}$};
  \draw ($(A)+(1,1)$) node [above] {\tiny$T_{i+1}$};
  #2{#1};
  #3{(A)};
}

In \cite{Q-SCMRL} a dg-operad $\RS$ was defined and shown to be an $\text{SC}_2$-operad. 
In this section we prove that there exists a weak equivalence of operads $\Phi:\rBr\to \RS$.

In Section \ref{sec: def rs} we recall the operad $\RS$ and fix notation. In particular, we consider $\RS$ as generated by a set  $\Tas$.  We also introduce the morphism $\Phi$. In Section \ref{sec: bu and nested} and \ref{sec: bu and nested as} we describe the blow-up components of the complexes $\rBr(I;x)$ and $\RS(I;x)$ in terms of \emph{nested intervals}. In those sections, the contraction map is defined.
In order to prove that $\Phi$ is a weak equivalence, we show that the complexes $\rBr(I;x)$ and $\RS(I;x)$ admit a cellular decomposition (see Definition \ref{de: cell decompo}) over the same poset $\Tas(I;x)$.  In the case of $\RS$, the cells involved are isomorphic to the cellular complexes of the standard simplexes. In the case of $\rBr$ the standard simplexes are replaced by the Stasheff's polytopes \cite{Sta1} and the cyclohedra.  
 With this geometric description, the morphism $\Phi$ may be seen as induced by the contraction of the Stasheff polytope onto a point. Consequently  $\Phi$ sends the $n$-dimensional cyclohedron onto the $n$-simplex. This is performed combinatorially by the contraction map.

Section \ref{sec: cell decompo} is devoted to the description of the cellular decompositions and Section \ref{sec: proof closed} is the proof  that $\rBr(I;x)$ has a cellular decomposition indexed by $\Tas(I;x)$. 

\newcommand{\Tsb}{\cal{T}_{sb}}

\subsection{Definition of the relative surjections operad} \label{sec: def rs}

We define an operad $\RS$ which is, up to signs, the \emph{relative surjections operad} obtained in \cite{Q-SCMRL} by using the filtration from Remark 3.5 in \emph{loc. cit.}
\\

The $\{\cl;\op\}$-colored operad $\RS$ is described as the free $\Z$-module generated by a set $\Tas$ of planar rooted trees and words of planar rooted trees defined as follows. Let $\mathcal{T}_{2,1}\subset \mathcal{T}$ be the set of elements in $\cal T$ whose neutral round-shaped vertices have  $2$ inputs and  neutral square-shaped vertices have  one input. Let $R_{As}$ be the equivalence relation generated by 
\begin{equation}\label{eq: Ras}
\begin{tikzpicture}[baseline=1.ex,scale=0.5]
  \arbrecrienrootr{(0,0)}{T_0}{}{T_3};
  \arbrecrien{(-1,1)}{}{T_1}{T_2};
\end{tikzpicture}
\sim
\begin{tikzpicture}[baseline=1.ex,scale=0.5]
  \arbrecrienrootl{(0,0)}{T_0}{T_1}{};
  \arbrecrien{(1,1)}{}{T_2}{T_3};
\end{tikzpicture}
\text{ and }
\begin{tikzpicture}[baseline=1.ex,scale=0.65]
 \coordinate (A) at (-1,0);
 \coordinate (B) at (0,0);
  \draw [-] ($(A)+(0,1)$) -- (A) -- ($(A)-(0,.5)$);
  \draw [-] ($(B)+(0,1)$) -- (B) -- ($(B)-(0,.5)$);
 \brect{($(A)$)};
 \brect{($(B)$)};
 \draw ($(A)+(0,1)$) node [above] {\tiny$T_1$};
 \draw ($(B)+(0,1)$) node [above] {\tiny$T_2$};
\end{tikzpicture}
\sim
\begin{tikzpicture}[baseline=1.ex,scale=0.65]
 \coordinate (A) at (0,0);
 \draw [-] ($(A)+(0,1)+(-1,1)$) -- ($(A)+(0,1)$) -- ($(A)+(0,1)+(1,1)$);
  \draw [-] ($(A)+(0,1)$) -- (A) -- ($(A)-(0,.5)$);
 \brect{(A)};  
 \brond{($(A)+(0,1)$)};
 \draw ($(A)+(-1,2)$) node [above] {\tiny$T_1$};
 \draw ($(A)+(1,2)$) node [above] {\tiny$T_2$};
\end{tikzpicture}.
\end{equation}
Let us define 
\[
\Tas:=\mathcal{T}_{2,1}/R_{As}.  
\]
Notice that since the neutral round-shaped and the neutral square-shaped vertices have degree $0$, the relations $R_{As}$ are homogeneous. 
As in Section \ref{sec: RBr def}, for $I=(s_1,\ldots,s_n)$ a sequence in $\{\cl,\op\}$ and $x\in\{\cl,\op\}$, the $\Z$-modules $\RS(I;x)$ and
$\Lambda_c\RS(I;x)$ are freely generated by the set $\Tas(I;x)$. The composition and differential in $\RS$ and $\Lambda_c\RS$ only differ by signs. In order to define the differential, it is more convenient to define it for the operad $\Lambda_c\RS$, especially since one can forget about
the sign determinant of a tree $T$.

Hence, let us write
\[
\Lambda_\cl\RS(I;x)=\bigoplus_{[T]\in\Tas(I;x)} \Z [T].
\]

The operadic compositions and the differential in $\Lambda_\cl\RS$ are defined in the same way as for $\Lambda_c\rBr$
respecting the constraint of the set $\mathcal T_{2,1}$. Thus the blow-ups  \eqref{deltaclosed} and \eqref{deltaopen} are respectively replaced by 
\begin{multline}\label{deltaclosedAS}
d_v\Bigg(
\begin{tikzpicture}[baseline=-0.65ex,scale=0.7]
 \arbrepartial{(0,0)}{\wrond};
 \end{tikzpicture}
 \Bigg)
=-
\begin{tikzpicture}[baseline=-0.65ex,scale=0.7]
 \arbrepartialll{(0,0)}{\wrond};
 \coordinate (A) at (0,-.8);
 \brond{(A)};
 \draw [-] ($(A)+(-0.8,.8)$) -- (A);
 \draw ($(A)+(-0.8,.8)$) node [above] {\tiny $T_1$};
 \draw (A) node [right] {\tiny $v_1$};
\end{tikzpicture}
-
\begin{tikzpicture}[baseline=-0.65ex,scale=0.7]
 \arbrepartialrr{(0,0)}{\wrond};
 \coordinate (A) at (0,-.8);
 \brond{(A)};
 \draw [-] ($(A)+(0.8,.8)$) -- (A);
 \draw ($(A)+(0.8,.8)$) node [above] {\tiny $T_k$};
 \draw (A) node [left] {\tiny $v_1$};
\end{tikzpicture}
+\sum_{1\leq i\leq k-1}
\begin{tikzpicture}[baseline=-0.65ex,scale=0.7]
 \arbreblowas{(5,0)}{\wrond}{\brond};
\end{tikzpicture},
\end{multline}

and

\begin{multline}\label{deltaopenAS}
d_v\Bigg(
\begin{tikzpicture}[baseline=-0.65ex,scale=0.7]
 \arbrepartialsq{(0,0)}{\wrect};
 \end{tikzpicture}
 \Bigg)
=+\sum_{1\leq i\leq k-1}
\begin{tikzpicture}[baseline=-0.65ex,scale=0.7]
 \arbreblowas{(5,0)}{\wrect}{\brond};
\end{tikzpicture}
+
\begin{tikzpicture}[baseline=-0.65ex,scale=0.7]
 \arbrepartialsqlibre{(5,0)}{\wrect}{1}{k-1}{$v^1$};
  \coordinate (B) at (7,0);
  \draw [-] ($(B)+(0,1)$) -- (B) -- ($(B)-(0,.5)$);
 \brect{($(B)$)};
 \draw ($(B)+(0,1)$) node [above] {\tiny$T_k$};
 \draw (B) node [left] {\tiny $v_2$};
\end{tikzpicture}
-
\begin{tikzpicture}[baseline=-0.65ex,scale=0.7]
 \arbrepartialsqlibre{(7,0)}{\wrect}{2}{k}{$v^2$};
  \coordinate (B) at (5,0);
  \draw [-] ($(B)+(0,1)$) -- (B) -- ($(B)-(0,.5)$);
 \brect{($(B)$)};
 \draw ($(B)+(0,1)$) node [above] {\tiny$T_1$};
 \draw (B) node [left] {\tiny $v_1$};
\end{tikzpicture}.
\end{multline}

\newcommand{\arbrelr}[3]{
 \draw [-] ($#1+(-1,1)$) -- #1 -- ($#1+(1,1)$);
 \draw [-] #1 -- ($#1-(0,.5)$);
 \brond{#1};  
 \wrect{($#1+(-1,1)$)};
 \wrect{($#1+(1,1)$)}; 
  \draw ($#1+(-1,1)$) node [right] {$#2$};
  \draw ($#1+(1,1)$) node [above] {$#3$};
}
\newcommand{\arbrerr}[3]{
 \draw [-] ($#1+(-1,1)$) -- #1 -- ($#1+(1,1)$);
 \draw [-] #1 -- ($#1-(0,.5)$);
 \brond{#1};  
 \wrect{($#1+(-1,1)$)};
 \wrect{($#1+(1,1)$)}; 
  \draw ($#1+(-1,1)$) node [above] {$#2$};
  \draw ($#1+(1,1)$) node [left] {$#3$};
}
\newcommand{\arbrecr}[3]{
 \draw [-] ($#1+(-1,1)$) -- #1 -- ($#1+(1,1)$);
 \draw [-] #1 -- ($#1-(0,.5)$);
 \brond{#1};  
 \wrect{($#1+(-1,1)$)};
 \wrect{($#1+(1,1)$)}; 
  \draw ($#1+(-1,1)$) node [above] {$#2$};
  \draw ($#1+(1,1)$) node [above] {$#3$};
}
\begin{figure}[h]
 \begin{center}
\begin{tikzpicture}[baseline=1ex,scale=0.8]
 \coordinate (A) at (-1,0);
  \draw [-] (-1,-.5) -- (A) -- (-1,1);
  \draw [black,fill=white] (-1.1,-.1) rectangle (-.9,.1) ; 
  \draw [black,fill=white] (-1,1) circle (0.1)  ;
  \draw (-1,1.1) node [above] {\tiny$2$};
  \draw (-1,0) node [left] {\tiny$\un{1}$};
\end{tikzpicture}
$\circ_{\un{1}}$
\begin{tikzpicture}[baseline=1ex,scale=0.8]
\zero{\wrect};
\draw (0,0) node [left] {\tiny$\un{1}$};
  \end{tikzpicture}
  \begin{tikzpicture}[baseline=1ex,scale=0.8]
\zero{\wrect};
\draw (0,0) node [left] {\tiny$\un{2}$};
  \end{tikzpicture}
  $=$
  \begin{tikzpicture}[baseline=1ex,scale=0.8]
\undeuxun{\wrect}{\wrond}{\un{1}}{3};
\end{tikzpicture}
\begin{tikzpicture}[baseline=1ex,scale=0.8]
\zero{\wrect};
\draw (0,0) node [left] {\tiny$\un{2}$};
  \end{tikzpicture}
 $ +$
\begin{tikzpicture}[baseline=1ex,scale=0.8]
\zero{\wrect};
\draw (0,0) node [left] {\tiny$\un{1}$};
  \end{tikzpicture}
\begin{tikzpicture}[baseline=1ex,scale=0.8]
\undeuxun{\wrect}{\wrond}{\un{2}}{3};
\end{tikzpicture}
\end{center}\caption{Example of a composition in $\Lambda_c\RS$.}\label{fig: compo sS2}
\end{figure}

%
 For  $T\in \mathcal{T}_{2,1}$, we denote by $d_v^i(T)$ the $i$-th component of the blow-up \eqref{deltaclosedAS} or \eqref{deltaopenAS}. Thus  equations \eqref{deltaclosedAS} and  \eqref{deltaopenAS} are written respectively as
 \[
d_v(T)= -d^0_v(T)-d^k_v(T)+ \sum_{1\leq i\leq k-1} d^i_v(T)\ \text{ and } \
d_v(T)= \sum_{1\leq i\leq k-1} d^i_v(T)+d^k_v(T) -d^0_v(T).   
 \]

 Notice that the blow-up components of \eqref{deltaclosedAS} and \eqref{deltaopenAS}, defined on $\mathcal{T}_{2,1}$ with value in $\mathcal{T}_{2,1}$, induce blow-up components on $\Tas$. This means that, for $T,T'\in \mathcal{T}_{2,1}$ such that  $[T]=[T']\in \Tas$, one has $[d_v^i(T)]=[d_v^i(T')]$. 
 Let us define 
 \begin{equation*}
d_v^i([T]):=[d_v^i(T)] \text{ for any } T\in \mathcal{T}_{2,1}.  
 \end{equation*}

The map
  \[ 
  \Phi:   \rBr\to  \RS,
  \]
associates to  $T\in \mathcal{T}$ the element
\begin{equation*}
 \Phi(T)=
 \begin{cases}
 [T]  \text{ if } T \in \mathcal{T}_{2,1},\\
 0 \text{ otherwise}.
 \end{cases}
 \end{equation*}

 \begin{lem}\label{lem: morph dg-operad}
 The map $\Phi$ is a morphism of dg-operads. 
\end{lem}
\begin{proof}
Straightforward. 
 \end{proof}

\subsection{Blow-up components in $\cal T$ as nested intervals}\label{sec: bu and nested}
In  \cite{Sta1} was given an interpretation of the $(n-k+1)$-dimensional faces of the $(n-2)$-dimensional Stasheff polytope $K_{n}$ in terms of trees. 
With our notation, those trees may be identified to rooted planar trees  with $n$ leaves (= terminal vertices) and $k$ neutral round-shaped vertices, that is, to iterations of blow-up components of the tree $\partial_{n}$ (see Figure \ref{fig: generators}).    
 We give a similar description of the generators  $M_{1,n}$, $G_{n}$ and $\Gamma_{1,n}$ as well as of their blow-up components. 
The two generators $G_{n}$ of degree $n-1$ and $\Gamma_{1,n}$ of degree $n$ correspond in this description to the top cell of $K_{n+1}$ and $K_{n+2}$ respectively. 
The generator $M_{1,n}$ of degree $n$ corresponds to the top cell of the $n$-dimensional cyclohedron $W_n$.

\subsubsection{On the poset structure of $\cal T$}

\begin{nota}
 Let $(\mathcal{P},\leq)$ be a partially ordered set (poset for short).  
 Any subset $S$ of $\cal{P}$ is a poset for: $a\leq b \in S \Leftrightarrow a\leq b \in \cal{P}$. 
 For $a\in \mathcal{P}$, let $(\mathcal{P})_{\leq a}$ be the set of elements $b\in \cal{P}$ such that $b\leq a$.  
The meet of $a$ and $b$ (the greatest lower bound of the set $\{a,b\}$) is denoted by $a\wedge b$.

 The product of two posets $\cal{P}_1$ and $\cal{P}_2$ is a poset for the following structure. For $(a_1,a_2)$ and $(b_1,b_2)$ in $\cal{P}_1\times \cal{P}_2$, one has  $(a_1,a_2)\leq (b_1,b_2) \Leftrightarrow a_1\leq b_1$ and $a_2\leq b_2$.
 
  A poset $\mathcal{P}$ naturally defines a category (also denoted $\mathcal{P}$) with objects the elements of $\mathcal{P}$ and  morphism $\cal{P}(a,b)=*$ if $a\leq b$ and $\cal{P}(a,b)=\emptyset$ otherwise.  
\end{nota}

\begin{defn}
The set $\mathcal{T}(I;x)$ has naturally a poset structure given by the blow-up components. 
Explicitly, $T'\leq T$ if and only if $T'$ is obtained from $T$ by an iteration of blow-up components of the form \eqref{deltaroundneutral}, \eqref{deltaclosed}, \eqref{deltaopen} or \eqref{deltaneutralsquare}. 
\end{defn}

For $T\in \cal{T}$, let us denote by $<(\mathcal{T})_{\leq T}>$ the sub complex of $\rBr$ generated by $(\mathcal{T})_{\leq T}$. 

\begin{prop}{\cite[Proposition 9.2.8.]{LodVal}}\label{prop: 9.2.8}
For each $n\geq 2$,   $<(\cal{T})_{\leq \partial_n}>$ is the cellular chain complex of $K_n$, $C^{cell}_*(K_n)$. 
\end{prop}

For $n\geq 0$, let us denote by $W_n$ the $n$-dimensional cyclohedron, see definition in \cite{Devadoss}.  

\begin{prop}\label{P:corolles}
For each $n\geq 0$, one has the following three isomorphisms. 
\begin{enumerate}
 \item $<(\mathcal{T})_{\leq G_{n+1}}>=  C^{cell}_*(K_{n+2})$.
 \item $<(\mathcal{T})_{\leq \Gamma_{1,n}}>=  C^{cell}_*(K_{n+2})$. 
 \item $<(\mathcal{T})_{\leq M_{1,n}}>=  C^{cell}_*(W_{n})$. 
\end{enumerate}
\end{prop}

\begin{proof}
\textit{a)} Let 
\begin{equation}\label{bij: becarre}
\natural,\natural':(\mathcal{T})_{\leq G_n}\to (\mathcal{T})_{\leq\partial_{n+1}}     
    \end{equation}
be the following maps. 

For a  tree $T\in (\mathcal{T})_{\leq G_n}$, one defines $\natural(T)$ (resp. $\natural'(T)$)  as the tree obtained from $T$ by:
\begin{itemize}
 \item replacement of the neutral square-shaped vertex {$v$} by a neutral round-shaped vertex {$\natural{v}$ (resp. $\natural'{v}$}); and,
 \item grafting to {$\natural{v}$ (resp. $\natural'{v}$}) the unit tree 
\begin{tikzpicture}[baseline=-0.65ex,scale=0.9]
  \draw (0,-.4) -- (0,0);
  \wrond{(0,0)};
 \end{tikzpicture} 
 as its first (resp. last) input. 
\end{itemize}
  
For a word $W=T_1\cdots T_p\in (\mathcal{T})_{\leq G_n}$, one defines 
\begin{equation*}
\natural(W):=((\cdots (\natural(T_p)\circ_{v^p_{min}} \natural(T_{p-1}))\circ_{v^{p-1}_{min}} \cdots )\circ_{v^2_{min}} \natural(T_1)),
\end{equation*}
where $v^i_{min}$ denotes  the first closed vertex of $\natural(T_i)$, and \begin{equation*}
\natural'(W):=((\cdots (\natural'(T_1)\circ_{v^1_{max}} \natural'(T_{2}))\circ_{v^{2}_{max}} \cdots )\circ_{v^{p-1}_{max}} \natural'(T_p)),
\end{equation*}
where  $v^i_{max}$ denotes  the last closed vertex of $\natural'(T_i)$.

For each neutral square-shaped vertex $v$, one has

\begin{equation}
 \begin{split}\label{eq: becarre}
 \natural\circ d_{v}^{[i,i+j]}&=d_{\natural v}^{[i+1,i+j+1]}\circ \natural, \\
 \natural\circ d_{v}^{i}&=d_{\natural v}^{[1,i+1]}\circ \natural, ~~\text{ for all $i,j$ as in } \eqref{deltaneutralsquare}, \\
\natural'\circ d_{v}^{[i,i+j]}&=d_{\natural' v}^{[i,i+j]}\circ \natural', \\
 \natural'\circ d_{v}^{i}&=d_{\natural' v}^{[i+1,k+1]}\circ \natural', ~~\text{ for all $i,j$ as in } \eqref{deltaneutralsquare}.
\end{split}
\end{equation}

For each neutral round-shaped vertex $v$, one has 
\begin{equation}\label{eq: becarre 2}
\natural\circ d_{v}^{[i,i+j]}=d_{v}^{[i,i+j]}\circ \natural  ~~\text{ and  } ~~ \natural'\circ d_{v}^{[i,i+j]}=d_{v}^{[i,i+j]}\circ \natural'
 ~~\text{ for all $i,j$ as in } \eqref{deltaroundneutral}. 
\end{equation}
It follows that $\natural$ and $\natural'$ are bijections of posets. In $\Lambda_\cl\rBr$ these maps extend as
\begin{equation*}
 \natural(W\otimes \epsilon_W)=(-1)^p\natural(W)\otimes \natural\epsilon_W \text{ and }
\end{equation*}
\begin{equation*}
 \natural'(W\otimes \epsilon_W)=\natural'(W)\otimes \natural'\epsilon_W,
\end{equation*}
where $W$ is a word of length $p$ and
$\natural\epsilon_W:= v_\lft\wedge\natural v_1\wedge ... \wedge\natural v_k$ and
$\natural'\epsilon_W:=\natural v_1\wedge ... \wedge\natural v_k\wedge v_\rgt$
if $\epsilon_W= v_1\wedge ... \wedge v_k$ and $v_\lft$ and $v_\rgt$ denote the additional closed vertex from the definitions of $\natural$ and $\natural'$.  

It follows that   $\natural$  induces an isomorphism of complexes 
\begin{equation*}
\natural_*:(<(\mathcal{T})_{\leq G_n}>,d_{\rBr})\to (< (\mathcal{T})_{\leq \partial_{n+1}}>,d_{\rBr}).
\end{equation*}
The conclusion follows by Proposition  \ref{prop: 9.2.8}.
\\

 \textit{b)} 
The map
 \begin{equation}\label{eq: bij bemol}
\flat:(\mathcal{T})_{\leq \Gamma_{1,n}}\to (\mathcal{T})_{\leq \partial_{n+2}} ,
\end{equation}
associates to a tree $T\in (\mathcal{T})_{\leq \Gamma_{1,n}}$, the element $\flat(T)$ obtained from $T$ as follows:
\begin{itemize}
 \item replacement of the open vertex $v$ of $T$ by a neutral round-shaped vertex $\flat v$; and,
 \item grafting two unit trees 
\begin{tikzpicture}[baseline=-0.65ex,scale=0.9]
  \draw (0,-.4) -- (0,0);
  \wrond{(0,0)};
 \end{tikzpicture} above $\flat v$, 
 one as its first input, the other one as its last input. 
\end{itemize}
 
Any word $W$  in $(\mathcal{T})_{\leq \Gamma_{1,n}}$ writes uniquely as $W=W_lT_0W_r$ where $T_0$ is an open rooted tree and $W_l$ ad $W_r$ are words of neutral square-shape rooted trees.
We define 
\begin{equation*}
\flat(W):=\big( \flat(T_{0})\circ_{v_{max}}  \natural'(W_r) \big) \circ_{v_{min}} \natural(W_l), 
\end{equation*}
 where $v_{min}$ and  $v_{max}$ denote respectively the first and last closed vertex of $\flat(T_{0})$. 

 If $v$ denotes the open vertex of $W$, 
 then one has 
\begin{equation}
 \begin{split}\label{eq: flatt}
\flat\circ d_{v}^{[i,i+j]}&=d_{\flat v}^{[i+1,i+j+1]}\circ \flat, \\
 \flat\circ d_{v}^{[i+1,\rgt]}&=d_{\flat v}^{[i+2,k+2]}\circ \flat, \\
  \flat\circ d_{v}^{[\lft,i]}&=d_{\flat v}^{[1,i+1]}\circ \flat,
 ~~\text{ for all $i,j$ as in } \eqref{deltaopen}. 
\end{split}
\end{equation}
By combining the relations \eqref{eq: becarre}, \eqref{eq: becarre 2}, \eqref{eq: flatt} and $\flat(\Gamma_{1,n})=\partial_{n+2}$,  one obtains that $\flat$ is a bijection of posets.

Let us extend $\flat$ as a chain map. 
For $W=W_lT_0W_r\in (\mathcal{T})_{\leq \Gamma_{1,n}}$, with $W_l$ of length $p$, we write $\epsilon_W=\epsilon_{W_l}\wedge\epsilon_{T_0}\wedge \epsilon_{W_r}$ with $\epsilon_{T_0}=v_0\wedge V_0$, $v_0$ being the root of $T_0$. With this notation
\begin{equation*}
\flat_*(W\ot \epsilon_W):=(-1)^{p+|W|} \flat(W)\ot \flat\epsilon_{W}, 
\end{equation*}
where $\flat\epsilon_{W}:= \flat v_{0}\wedge \natural\epsilon_{W_l} \wedge V_0\wedge\natural'\epsilon_{W_r}$. 
It yields  an isomorphism of complexes 
\begin{equation*}
\flat_*:(<(\mathcal{T})_{\leq \Gamma_{1,n}}>,d_{\rBr})\to (< (\mathcal{T})_{\leq \partial_{n+2}}>,d_{\rBr}).
\end{equation*}
One concludes with Proposition  \ref{prop: 9.2.8}.
\\

\textit{c)}
From the proposition \ref{P: bij cyclohedra} below, there is a bijection
 $\Upsilon:\Nest(S(n))\to (\mathcal{T})_{\leq M_{1,n}}.$
 From Markl in \cite{Markl99}
 the poset $\Nest(S(n))$ (see definition below) labels the cyclohedron $W_{n}$ (see also \cite{Devadoss}). 
  It follows that $<(\mathcal{T})_{\leq M_{1,n}}>= C_*^{cell}(W_n)$. 
\end{proof}

\subsubsection{The poset $\cal{T}$ and the blow-up components}

In what follows, we describe the posets  $(\mathcal{T})_{\leq \Gamma_{1,n}}$, $(\mathcal{T})_{\leq M_{1,n}}$  and
$(\mathcal{T})_{\leq G_n}$ in terms of \emph{nested intervals}. 
More precisely, we prove a correspondence between nested intervals and  blow-up components of vertices. Given $T\in \cal{T}$, we
show that these nested intervals fully encode the poset $(\cal{T})_{\leq T}$, cf. Proposition \ref{prop: bij T no neutral}.

\newcommand{\sq}{
  \begin{tikzpicture}[baseline=-0.45ex,scale=1]
  \draw [black,fill=black] (-0.06,-0.06) rectangle (.06,.06);
 \end{tikzpicture}} 

We generalize the notation \eqref{eq: notation bu nrs}, \eqref{eq: notation bu cl}, \eqref{eq: notation bu op} and \eqref{eq: notation bu opcl} by mean of nested intervals.

\begin{defn}\label{def: strict interval} Let $S$ be a totally ordered set.
A \emph{strict interval} of $S$ is an interval which is not a singleton nor $S$.
For $n\geq 0$ the total ordered sets $\{\lft<1<...<n<\rgt\}$ and $\{1<...<n<\rgt\}$ are denoted respectively by $I(n)$ and $I^{right}(n)$.

Let $S(n)$ be the set $\{0,1,...,n\}$ endowed with the canonical cyclic order.  For $i,j\in\{0,\ldots,n\}$ with $i\not=j$ the
 \emph{strict interval} $[i,j]$ of $S(n)$ is either $[i,j]=\{i<i+1<...<j\}$ if $i< j$ or $[i,j]=\{i<i+1<...<n<0<1<...<j-1<j\}$ if $i>j$.
 Notice that the latter interval is denoted $j)(i$ in \cite{Markl99}.

\end{defn}

Let $S$ be a total ordered set or a cyclic ordered set.  The power set of the set of strict intervals in $S$ is denoted $ P^{2+}(S)$. It is a poset ordered by reverse inclusion, that is, $x\leq y\Leftrightarrow y\subset x$.
Any two elements $\cal{A},\cal{B}$ of   $ P^{2+}(S)$ have a meet which is $\cal{A}\wedge \cal {B}=\cal{A}\cup \cal{B}$.

An element $\cal{A}=\{A_t\}_{t\in K}$ of $P^{2+}(S)$ is said to be \emph{nested} if,  for each $r,s \in K$, one has $A_r\cap A_s\in \{ \emptyset, A_r, A_s\}$.  Let $\Nest(S)$
denote the set of nested elements of $P^{2+}(S)$ ; notice that the meet of nested elements $\cal{A}$ and $\cal{B}$ is not necessarily nested.

\begin{prop}\label{P: bij cyclohedra} There is a  bijection of posets 
 $\Upsilon:\Nest(S(n)) \to (\mathcal{T})_{\leq M_{1,n}}$. 
\end{prop}

\begin{proof} 
From a nested family $\mathcal{A} \in \Nest(S(n))$, we can obtain a new such family by adjoining an interval $[a,b]$ satisfying the nesting condition, i.e., the interval $[a,b]$ is either contained in some interval of $\mathcal{A}$ or disjoint with any other interval $A  \in \mathcal{A}$. Under this condition, $\mathcal{A}' = \mathcal{A} \cup \{[a,b]\}$ is a new nested family. 

On the other hand, any tree in $(\cal{T})_{\leq M_{1,n}}$ can be obtained by iterating the two types of blow-ups appearing in (\ref{eq: notation bu cl}). We recall that the leaves of $M_{1,n}$ are labelled by $\{1,\ldots,n\}$ and the root is labelled by $v$.
We will describe the correspondence $\Upsilon$ between including an interval in a family of $\Nest(S(n))$ and blowing up a tree in $(\mathcal{T})_{\leq M_{1,n}}$.
If $\mathcal{A} = \emptyset$ is the empty family, then: $\Upsilon(\emptyset) = M_{1,n}$. 
Then we define $\Upsilon$ in such a way that every new interval added corresponds to a neutral  vertex connected to the vertices in the associated tree that are labeled by the non-zero elements in the interval. Those vertices whose labels do not belong to any interval in the family are directly connected to the  vertex $v$. 

If $\Upsilon(\mathcal{A}) = T$, let us consider that $\mathcal{A}' = \mathcal{A} \cup \{[\alpha, \beta]\}$ is a new nested family.   
The nesting condition says that the interval $[\alpha, \beta]$ is either contained or disjoint to the intervals of $\mathcal{A}$. In the case where it is disjoint, we blow up the vertex $v$ of the tree $T$. In the case where $[\alpha, \beta]$ is contained in some interval, we blow up the  neutral vertex of $T$ associated to that interval. In both cases, we get a tree $T'$. 
Its new neutral vertex is joined to those  vertices labeled in $[\alpha, \beta] - \{ 0 \}$. If $0\in [\alpha,\beta]$, then the new neutral vertex is the root of  $T'$.
We define the map $\Upsilon$ such that: $\Upsilon(\mathcal{A}') = T'$.

If we include more intervals in a nested family $\mathcal{A}$, we obviously get a new family $\mathcal{B}$ such that $\mathcal{A} \subseteq \mathcal{B}$, and the corresponding trees are obtained by iterated blow-ups. This shows that $\Upsilon$ is a map of posets. On the other hand, any tree in $(\cal{T})_{\leq M_{1,n}}$ is obtained by a sequence of iterated blow-ups respecting the nesting condition. This gives an inverse map to $\Upsilon$.
 Hence, it is a bijection of posets.
\end{proof}

\begin{nota} Fix an element $T\in\cal T$. Let $v$ be a vertex of $T$. As described in Section \ref{S: vertex}, $v$ has a certain type, denoted by $c_v$ with 
$c_v\in\{\bullet,\cl,\sq,\op\}$. Let $n=In(v)$. To the vertex $v$ is associated the set $I_v$ and 
the corolla $\Cor_v$, defined by the following

\[ I_v=\begin{cases} \{1<\ldots<n\}, & \text { if } c_v=\bullet, \\
S(n), & \text { if } c_v=\cl, \\
I^{\lft}(n), & \text { if } c_v=\sq, \\
I(n), & \text { if } c_v=\op, \end{cases} \quad \text{and} \quad 
Cor_v=\begin{cases} \delta_n, & \text { if } c_v=\bullet, \\
M_{1,n,} & \text { if } c_v=\cl, \\
G_n, & \text { if } c_v=\sq, \\
\Gamma_{1,n}, & \text { if } c_v=\op. \end{cases}\]

Notice that $I_v$ is a total ordered set except for $c_v=\cl$ where $I_v$ is a cyclic ordered set. 

\end{nota}

Let $S$ be a total ordered set on $k$ letters. Let $\kappa: \Nest(S)\to  (\mathcal{T})_{\leq \partial_{k}}$ be the bijection in which  $\Nest(S)$ is identified 
with  the poset of parenthesis of $k$ letters. This bijection of posets induces the following ones

\begin{equation}\label{eq: bij vertex}
 D_{c_v}:\Nest(I_v)\to (\mathcal{T})_{\leq \Cor_v},
\end{equation}

where $D_\bullet=\kappa$,  $D_\op=\flat^{-1}\circ \kappa$ and $D_{\sq}= \natural^{-1}\circ \kappa$, with $\flat$ and $\natural$ are the bijections
defined in \eqref{bij: becarre} and \eqref{eq: bij bemol} and $D_\cl$ is the bijection $\Upsilon: \Nest(S(n))\to (\mathcal{T})_{\leq M_{1,n}}$ from Proposition \ref{P: bij cyclohedra}.

 For $\cal{A}_v\in \Nest(I_v)$, let us define
 \begin{equation}\label{eq: not dva}
  d_v^{\cal{A}_v}(T)=T\circ_v^{\mathcal{T}} D_{c_v}(\cal{A}_v) \in \cal T,
 \end{equation}
obtained from $T$ by substituting the vertex $v$ by the word of trees $D_{c_v}(\cal{A}_v)$. (see  Figure \ref{fig: extension PRT} for $c_v=\sq$).
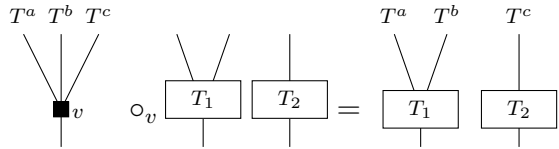
\begin{figure}[H]
 \begin{equation*}
 \begin{tikzpicture}[baseline=-0.65ex,scale=1]
 \trois{\brect}{T^a}{T^b}{T^c};
 \draw (0,-0.05) node [right] {\tiny$v$};
 \end{tikzpicture}
 \circ_v
 \begin{tikzpicture}[baseline=-0.65ex,scale=1]
 \deux{\brond}{}{};
 \boxa{(0,.13)}{$T_1$} ;
 \end{tikzpicture}
 ~~
 \begin{tikzpicture}[baseline=-0.65ex,scale=1]
 \une{\brond}{};
 \boxa{(0,.13)}{$T_2$} ;
 \end{tikzpicture}
 =
 \begin{tikzpicture}[baseline=-0.65ex,scale=1]
 \deux{\brond}{T^a}{T^b};
 \boxa{(0,0)}{$T_1$} ;
 \end{tikzpicture}
 ~~
 \begin{tikzpicture}[baseline=-0.65ex,scale=1]
 \une{\brond}{T^c};
 \boxa{(0,0)}{$T_2$} ;
 \end{tikzpicture}
 \end{equation*}\caption{Composition $T\circ_v^{\cal T} T_1T_2$.}\label{fig: extension PRT}
 \end{figure}
 
The substitution can be performed independently for every vertex $v$ in $T$, that is, $\forall v\not=w$ vertices of $T$, one has
\[  d_w^{\cal{A}_w}(d_v^{\cal{A}_v}(T))= d_v^{\cal{A}_v}(d_w^{\cal{A}_w}(T))\]
so that we can define $\Nest(T)= \prod_v \Nest(I_v)$ and a map 
 \begin{equation*}
 \begin{split}
\gamma^{\cal T}_T: \Nest(T) &\to (\cal{T})_{\leq T} \\
\cal{A}& \mapsto   d^{\cal{A}}(T):=\gamma^{\mathcal{T}}(T;D_{c_{v_1}}(\cal{A}_{v_1}),...,D_{c_{v_k}}(\cal{A}_{v_k})).
 \end{split}
 \end{equation*}

 \newcommand{\vertexhigh}[2]{
 \draw ($#1-(0,0.2)$) node [left] {\small$v$};
 \draw [-] ($#1+(-2,1)$) -- #1 -- ($#1+(2,1)$);
 \draw [-] ($#1+(-1,1)$) -- #1 -- ($#1+(1,1)$);
 \draw [-] ($#1+(-.5,1)$) -- #1 -- ($#1+(.5,1)$);
 \draw [-] ($#1-(0,.5)$) -- #1 -- ($#1+(0,1)$);
#2{#1};
}
\newcommand{\vertexbound}[3]{
 \draw [-]  #1 -- ($#1+(.8,.8)$);
 \draw [-] ($#1+(1,1)+(-.5,.5)$) -- ($#1+(.8,.8)$) -- ($#1+(1,1)+(.25,.5)$);
 \draw [-] ($#1+(1,1)+(.6,-.1)$) -- ($#1+(.8,.8)$) -- ($#1+(1,1)+(.6,.3)$);
 \draw [-] ($#1+(-.5,1)$) -- #1; 
 \draw [-] ($#1-(0,1)$) -- #1 ;
 \draw [-] ($#1-(0,.5)+(-1,-.1)$) -- ($#1-(0,.5)$) -- ($#1-(0,.5)+(-.8,.6)$) ;
#2{#1};
#3{($#1-(0,.5)$)};
#3{($#1+(.8,.8)$)};
}

 For example,
 \begin{equation*}
d_v^{[0,2],[4,7]}\left(
\begin{tikzpicture}[baseline=-0.65ex,scale=0.8]
 \vertexhigh{(0,0)}{\wrond};
 \end{tikzpicture}   
 \right) =
 \begin{tikzpicture}[baseline=-0.65ex,scale=0.75]
 \vertexbound{(0,0)}{\wrond}{\brond};
 \end{tikzpicture}   .
 \end{equation*}

\begin{rem}
The notation \eqref{eq: not dva} matches with the previous one, that is, if $A\in\Nest(I_v)$ is reduced to one interval, then $d_v^A$ corresponds to the notation introduced in equations \eqref{eq: notation bu nrs},  \eqref{eq: notation bu cl}, \eqref{eq: notation bu op} and \eqref{eq: notation bu opcl}. 
\end{rem}

\begin{prop}\label{prop: bij T no neutral}
 For $T\in \cal{T},$ the map $\gamma^{\cal T}_T$  is a bijection of posets. 
 Hence $<(\mathcal{T})_{\leq T}>$ is contractible.

\end{prop}
\begin{proof}
Since the maps \eqref{eq: bij vertex} are bijection of posets, $\cal{A}\leq \cal{B}$ if and only if  $D_{c_v}(\cal{A})\leq D_{c_v}(\cal{B})$,  if and only if $T\circ_v^{\mathcal T} D_{c_v}(\cal{A})\leq T\circ_v^{\mathcal T} D_{c_v}(\cal{B})$. Thus $\gamma^{\cal T}_T$ is a bijection of posets.
As a consequence
$<(\mathcal{T})_{\leq T}>=\otimes_v <(\mathcal{T})_{\leq \Cor_v}>$ which is contractible, due to Propositions \ref{prop: 9.2.8} and  \ref{P:corolles}
\end{proof}

\subsubsection{The contraction map}

\begin{defn}\label{def: contraction map}
 The contraction map 
$c:\mathcal{T}\to \mathcal{T},$  
 is defined via the following rule:
 merge adjacent neutral vertices into one round-shaped vertex if the two vertices are round-shaped or into one square-shaped vertex if one of them is square-shaped.
  \end{defn}
  
 For example, 
  \begin{equation*}
  c\Bigg(
  \begin{tikzpicture}[baseline=.85ex,scale=.7]
 \une{\brect}{1};
 \end{tikzpicture}
 \begin{tikzpicture}[baseline=0.85ex,scale=.7]
 \deuxdeux{\brect}{2}{3}{4};
 \brond{($(A)+(0.35,1)$)};
 \end{tikzpicture}
 \Bigg)= 
 \begin{tikzpicture}[baseline=0.85ex,scale=.7]
 \quatre{\brect}{1}{2}{3}{4};
 \end{tikzpicture}.
\end{equation*}

  Let ${\cal T}_{\sq}$ be the subset 
 of $\cal T$ formed by words of neutral square-shaped rooted trees only. Let ${\cal T}_\bullet$ be the subset of $\cal T$ formed by trees $T$ such that its canonical ordered set $v_1<\ldots <v_k$ of vertices satisfy the following condition: there exists a closed leaf $v_i$ such that  $c_{v_j}=\bullet$ for every $j<i$. 
 
 \begin{prop}\label{prop: c et becarre} The bijection $\natural$ defined in \eqref{bij: becarre} extends to a bijection of posets
 \[\natural: {\cal T}_{\sq} \rightarrow {\cal T}_\bullet,\]
 commuting with $c$.
  \end{prop}

  \begin{proof} The construction of the map $\natural$ extends to words of neutral square-shaped rooted trees since this construction is performed on the roots only, and thus the image of $\natural$ is precisely ${\cal T}_\bullet$. The two inclusions $c({\cal T}_{\sq})\subset {\cal T}_{\sq}$ and $c({\cal T}_\bullet)\subset {\cal T}_\bullet$ and the definition of $c$ imply that $c\natural=\natural c$.
  \end{proof}

 \subsection{Blow-up components in $\Tas$ as restrictive nested intervals}\label{sec: bu and nested as}
 
 \subsubsection{On the poset structure of $\Tas$}
 
\begin{defn}
The set $\Tas(I;x)$ has naturally a poset structure given by the blow-up components. 
Explicitly,  $[T']\leq [T]$ in $\Tas(I;x)$ if  $[T']$ is obtained from $[T]$ by iterations of blow-ups of the form \eqref{deltaclosedAS} or \eqref{deltaopenAS}  \ie   if 
\begin{equation}\label{eq: decompo boundary}
[T']=d_{v_r}^{i^{r,k_r}} \cdots d_{v_2}^{i^{2,1}} d_{v_1}^{i^{1,k_1}} \cdots  d_{v_1}^{i^{1,2}} d_{v_1}^{i^{1,1}} ([T]) 
\end{equation}
 for some $v_j$ non neutral vertices,  $i^{j,k}\geq 0$, $k_j\geq 0$ and $r\geq 0$. 
 
For $[T]\in \Tas$, let us denote by $<(\Tas)_{\leq [T]}>$ the sub complex of $\RS$ generated by $(\Tas)_{\leq [T]}$. 
\end{defn}

 \begin{rem}\label{rem: simplicial}
 For a fixed vertex $v$ the components $d_v^i$ satisfy the simplicial relations. Thus in \eqref{eq: decompo boundary}, the indices   $i^{s,{k_s}},...,i^{s,1}$ can be chosen such that $i^{s,{k_s}}\geq \cdots \geq i^{s,1}$. 
 Moreover, for two different vertices $v$ and $w$, one has $d_v^id_w^j=d_w^jd_v^i$, $\forall i,j$.
\end{rem}

 \begin{prop}\label{prop: M1n simplex} 
 For $n\geq 0$, one has 
 \[ <(\Tas)_{\leq [M_{1,n}]}>=C_*^{cell}(\Delta^n) \quad\text{ and } \quad <(\Tas)_{\leq [\Gamma_{1,n}]}>=C_*^{cell}(\Delta^{n}). \]
\end{prop}
 
 \begin{proof}
 From Remark \ref{rem: simplicial} it follows that $(\Tas)_{\leq [M_{1,n}]}$ is a semi-simplicial set and $<(\Tas)_{\leq [M_{1,n}]}>$ is its simplicial chain complex. 
 The first result follows from 
 \[ \{ d^{i_k}_v\cdots d^{i_1}_v([M_{1,n}])\}_{i_k\geq \cdots \geq i_1\geq 0}=\{f:[n-k]\to [n]~|~ f \text{ is an injective poset map}\},\forall k\geq 0,\] where $v$ denotes the root vertex of $M_{1,n}$. 
 The second statement is proved with the same arguments. 
 \end{proof}

 Let ${\Tasc}$ (resp. ${\Tasr}$) be the sets ${\cal T}_{\sq}\cap \cal T_{2,1}/\sim$ (resp. ${\cal T}_{\bullet}\cap \cal T_{2,1}/\sim$).
 
 \begin{prop}\label{prop: c et becarre as} The bijection $\natural: \cal T_{\sq}\rightarrow \cal T_{\bullet}$ induces a bijection
 \[ \natural: {\Tasc}\rightarrow {\Tasr}\]
 \end{prop}
 
 \begin{proof} $\natural$ induces a bijection $ \cal T_{\sq}\cap {\cal T}_{2,1} \rightarrow \cal T_{\bullet}\cap \cal T_{2,1}$ of posets, compatible with the equivalence relation $R_{As}$.
 \end{proof}

\subsubsection{The poset $\Tas$ and the blow-up components}

\begin{defn}\label{def: completion}

We begin by considering $S$ as a {linearly} ordered or a {cyclicaly} ordered set and let
 \begin{equation}\label{eq: map F}
F: \Nest(S)\to \Nest(S) 
\end{equation}
be the map that sends $\cal{A}=\{A_t\}_{t\in K}$ to the subset of its maximal elements (for the inclusion of strict intervals). Since $\cal A$ is nested, then $F\cal A$ contains disjoint strict intervals.
For a tree $T$, one extends $F$ to a map $F:\Nest(T)\rightarrow \Nest(T)$ defined by 
\[F(\prod_v {\cal A_v})=\prod_v F(\cal A_v)\] 

  A nested element $\cal{A}=\{A_t\}_{t\in K}\in \Nest(S)$ is said to be \emph{restrictive} if, for each $t\in K$,  either $|A_t|=2$ or, there exists $r\in K\setminus \{t\}$ such that 
  $A_t\setminus A_r\in \cal A$ or $A_t\setminus A_r= \{p\}$ for some $p\in S$. 
 Let us denote by $\Nest_{2}(S)$ the subset of $\Nest(S)$ formed by restrictive nested elements.
 
 \end{defn}

Let $\Nestas(S)$ be the following poset. 
 As a set, one has $\Nestas(S):=\Nest_{2}(S)/ \sim$ with
  \begin{equation*}
 \cal{A}\sim\cal{B} \Leftrightarrow F\cal{A}=F\cal{B}, 
 \end{equation*}
 The poset structure on $\Nestas(S)$ is induced by that of $\Nest_2(S)$: $[\cal{B}]\leq [\cal{A}]$ if and only if for  all representative set $\cal{A}_a$ of $[\cal{A}]$ there exists a representative set $\cal{B}_b$ of $[\cal{B}]$ such that  $\cal{B}_b\leq \cal{A}_a$ in $\Nest_2(S)$.

 Let $\cal A\in \Nest(S);$ by adding strict sub-intervals of intervals in $\cal A$ one obtains $\cal A'\in \Nest_2(S)$; the class $[\cal A']\in\Nestas(S)$ is independent from the sub-intervals added and is called the  {\it completion of $\cal A$} and denoted by $[\widehat{\cal A}]$.
  
 \begin{ex}
  For $n\geq 3$, the element $\cal{A}=\{[\lft, 3]\}$ is not in $\Nest_2(I(n))$. We can choose for representative of $[\widehat{\cal A}]$ either
  $\cal{B}=\{[\lft, 3],[\lft,1], [2,3]\}$ or $\cal{C}=\{[\lft, 3],[\lft,2], [1,2]\}.$
 \end{ex}

Notice that if $T\in\mathcal T_{2,1}$, then $\Nest(I_v)=\{\emptyset\}$ for every neutral vertex $v$ of $T$ so that 
\[\Nest(T)=\prod_{v| c_v\in\{\cl,\op\}} \Nest(I_v).\]
Restricting this set to $\prod_v \Nest_2(v)$ and modding out by the equivalence relation, the set $\prod_v \Nestas(I_v)$ is independent from a representative of $[T]$ so that one can define
\[\Nestas([T])=\prod_{v} \Nestas(I_v).\]
The bijections $D_{c_v}$ defined in \eqref{eq: bij vertex}
 restrict to poset bijections
\begin{equation*}
 D_{c_v}:\Nest_2(I_v)\to (\mathcal{T})_{\leq \Cor_v}\cap \mathcal T_{2,1},
\end{equation*}
inducing poset bijections
 \begin{equation*}
  D^{As}_{c_v}: \Nestas(I_v)\to (\Tas)_{\leq[Cor_v]} .
 \end{equation*}
 Similarly, for $\cal A_v\in \Nest_2(T)$ the class $[d_v^{\cal{A}_v}(T)]$ is independent from representatives of $[T]$ and  $[\cal A_v]$, hence defines maps
 \begin{equation}
  d_v^{[\cal{A}_v]}([T]):=[T]\circ_v^{\Tas} D^{As}_{c_v}([\cal A_v]):=[T\circ_v^{\mathcal{T}} D_{c_v}(\cal{A}_v)],
 \end{equation}
and

  \begin{equation*}
 \begin{split}
\gamma^{\Tas}_{[T]}:   \Nestas([T]) &\to (\Tas)_{\leq [T]} \\
[\cal{A}]& \mapsto   d^{[\cal{A}]}([T])=\gamma^{\Tas}([T];D_{c_1}([\cal{A}_{v_1}]),...,D_{c_k}([\cal{A}_{v_k}])) , 
 \end{split}
 \end{equation*}

\begin{rem}
 Given a closed vertex $v$ of $[T]\in \Tas$, the blow-up $d_v^{[i,j]}([T])$ corresponds to the composition of the $j-i$ face maps, $d_v^{i} \cdots d_v^i([T])$. 
\end{rem}

\begin{prop}\label{prop: bij Tas no neutral} The map $\gamma^{\Tas}_{[T]}$  is a bijection of posets. Hence $<(\Tas)_{\leq [T]}>$ is contractible for every element $[T]\in\Tas$.
\end{prop}

\begin{proof} The proof is similar to the proof of Proposition \ref{prop: bij T no neutral}
\end{proof}

\subsection{Cellular decompositions}\label{sec: cell decompo}

We provide a cellular decomposition of each complex $\rBr(I;x)$ and $\RS(I;x)$ over the same poset $\Tas(I;x)$.

\begin{nota}
 The category of chain complexes of $\Z$-modules  is denoted by $\Ch$. 
 We consider the projective model structure on $\Ch$, see \cite[Section 2.3]{Hov}, where weak equivalences are quasi-isomorphisms and cofibrations are monomorphisms with projective cokernels.
 A complex $C\in \Ch$ is said \emph{contractible} if it is homotopic to the trivial complex $\Z$.  

 The \emph{$\delZ$-realization} is given by $|-|:=-\ot_{\triangle}\delZ$ where $\delZ: \triangle\to \Ch$  is defined by $\delZ([n])=C_*(\Hom(-,[n]);\Z)$. 
 \end{nota}

We refer to \cite{Batanin-Berger-Lattice} for the following definition.
 
 \begin{defn}\label{de: cell decompo}
  Let $\cal{A}$ be a poset and $C\in \Ch$ be a non-negatively graded chain complex. 
  We say that $C$ admits an $\cal{A}$-\emph{cellular decomposition} if there is a functor $\Theta:\cal{A}\to \Ch$ such that:
\begin{enumerate}
 \item $C\cong \colim_{\alpha \in \cal{A}} \Theta(\alpha)$;\label{item 1}
 \item  $\forall\alpha \in \cal{A} , ~\colim_{\beta<\alpha}\Theta(\beta)\rightarrow \Theta(\alpha)$ is a cofibration;\label{item 3}
 \item for each $\alpha\in \cal{A}$, the ``cell''  $\Theta(\alpha)$ is contractible.\label{item 4}
\end{enumerate}
 \end{defn}

\begin{lem}\label{lem: weak equiv BA X}
 Let $C$ be a chain complex with an $\cal{A}$-cellular decomposition. We have the weak equivalences 
  \begin{center}
\scalebox{1}{
\begin{tikzpicture}[>=stealth,thick,draw=black!50, arrow/.style={->,shorten >=1pt}, point/.style={coordinate}, pointille/.style={draw=red, top color=white, bottom color=red},scale=0.8]
\matrix[row sep=5mm,column sep=8mm,ampersand replacement=\&]
{
 \node (00) {$C\cong \emph{colim}_{\alpha \in \cal{A}} \Theta(\alpha)$};  \& \node (01){$\emph{hocolim}_{\alpha \in \cal{A}} \Theta(\alpha)$} ; \& \node (10){$\emph{hocolim}_{\alpha \in \cal{A}} (\Z)\cong \Ba \cal{A}$,} ;\\
}; 
\path
   	  (00)     edge[above,<-]      node {$\sim$}  (01)
 	  (01)     edge[above,->]      node {$\sim$}  (10)
 	  ; 
\end{tikzpicture}}
\end{center}
where $\Ba \cal{A}$ denotes the $\delZ$-realization of the nerve of the category $\cal{A}$.
\end{lem}
\begin{proof}
 This is \cite[Section 3.1]{Batanin-Berger-Lattice} reformulated in the chain complexes framework. 
 See also \cite[Theorem 19.9.1]{Hirschhorn}. 
\end{proof}

\subsubsection{Cellular decomposition for the relative surjections}
Let 
\begin{align*}
\Theta:\Tas(I;x)&\to \Ch\\
 [T]&\mapsto \Theta([T]):= <(\Tas)_{\leq [T]}>.
 \end{align*}

 \begin{prop}
  $\Theta:\Tas(I;x)\to \Ch$ defines a $\Tas(I;x)$-cellular decomposition of $\RS(I;x)$.
 \end{prop}
\begin{proof}
The structural maps $\{i_a:\Theta(a)\to \RS(I;x)\}_{a\in \Tas(I;x)}$ are induced by the canonical inclusions 
$(\Tas)_{\leq a}\to \Tas(I;x)$. 
It follows that they are cofibrations. 
Notice that,  for each $a\leq b$, the inclusion $i_{a}^b: \Theta(a) \to \Theta(b)$ is a cofibration, as well as
the inclusions $\colim_{a<b} \Theta(a)\to \Theta(b)$.

Let us show that the complex $\RS(I;x)$ is isomorphic to $\colim_{[T]\in\Tas(I;x)}\Theta([T])$. By the universal property of the colimit,  it is sufficient to show that, given any complex $C$ together with a system of maps $\{j_{a}:\Theta(a)\to C\}_{a\in \Tas(I;x)}$ compatible with the morphisms of $\Tas(I;x)$,  there is a unique complex morphism $g:\RS(I;x) \to C$ making the usual diagram commute.

For $a$ a generator of  $\RS(I;x)$, since $a=i_a(a)$, set $g(a)=j_a(a)$ and extend it as a linear map. It follows that
$dg(a)=dj_a(a)=j_a(da)$ and $da=\sum_{j} \pm a_j$ where $\forall j$ $a_j\leq a$. Since $a_j=i_{a_j}^{a}(a_j)$, one has
$j_a(a_j)=j_{a_j}(a_j)$ so that $j_a(da)=\sum_{j} \pm j_{a_j}(a_j)=g(da)$. 

The contractibility of $\Theta(a)$ is given by Proposition \ref{prop: bij Tas no neutral}.
\end{proof}

 \subsubsection{Cellular decomposition for the relative braces}

 \newcommand{\cont}{\mathcal{P}^{[T]}}
  \newcommand{\contbecarre}{\mathcal{P}^{[\natural T]}}
 
 For $[T]\in \Tas$, we consider the sub poset of $\mathcal{T}$,
 \begin{equation*}
\cont=\{ c(T')\in \mathcal{T} ~|~ [T']\leq [T]\in \Tas \}.   
 \end{equation*} 
 Let us define  functors

\begin{equation*} 
  \begin{aligned}
  C^{[T]}: \cont &\to \Ch &\quad\text{ and }\quad  \Theta_\infty:\Tas(I;x)&\to \Ch \\
    a &\mapsto C_a:=<(\mathcal{T})_{\leq a}> & [T]& \mapsto \Theta_{\infty}([T]):= \colim_{a\in \cont} C_a. 
 \end{aligned}
\end{equation*}

\begin{prop}\label{prop: decompo braces}
  The functor $\Theta_\infty:\Tas(I;x)\to \Ch$ defines a $\Tas(I;x)$-cellular decomposition of $\rBr(I;x)$. 
 \end{prop}
\begin{proof}
See Section \ref{sec: proof closed}. 
\end{proof}

\begin{thm} \label{th: equiv}
The operad $\rBr$ is weakly equivalent to the operad of chains with coefficients in $\Z$ of the operad $\SC_2$.
\end{thm}
\begin{proof}
 By \cite{Q-SCMRL}, it is sufficient to show that $\Phi: \rBr\to \RS$ is a weak equivalence of dg-operads. 
 In order to do this we prove that for every $(I;x)$ the map of chain complexes $\Phi(I;x):\rBr(I;x)\to\RS(I;x)$ is a quasi-isomorphism. Let us prove that $\Phi$ induces a natural transformation of functors from $\Theta_\infty$ to $\Theta$,
 that is for each $(I;x)$,  the map $\Phi_{I;x}: \rBr(I;x)\to \RS(I;x)$ satisfies $\Phi_{I;x}(\Theta_\infty([T]))\subset\Theta([T]),\forall [T]\in \Tas(I;x)$ and commutes with the inclusion maps.
 Indeed any generator $a\in \Theta_\infty([T])$ 
 can be viewed as an element in $(\cal T)_{\leq c(T')}$  with $[T']\leq [T]$. If $a\not\in \cal T_{2,1}$ then $\Phi_{I;x}(a)=0$. If 
 $a\in \cal T_{2,1}$ then $a$ is obtained by blowing up the neutral vertices of $c(T')$ so that $[a]\leq [T']\leq [T]$. Hence $\Phi_{I;x}(a)\in\Theta([T])$. If $[T]\leq [T']$, then the structure diagram commutes. We denote by $\Phi_{[T]}$ the 
 induced map $\Theta_\infty(T)\to\Theta([T])$.
 complexes 
 
 The cells $\Theta_\infty([T])$ and $\Theta([T])$ being contractible, it follows that $\Phi_{[T]}$ is a quasi-isomorphism for every $[T]$.
 By homotopy invariance we deduce that $\Phi$ induces a quasi-isomorphism
 $\hocolim_{[T]} \Theta_\infty([T])\to\hocolim_{[T]}  \Theta([T]).$

 Thus, $\Phi_{I;x}$ is a weak equivalence of complexes by Lemma \ref{lem: weak equiv BA X}. 
 \end{proof}

\subsection{Proof of  the cellular decomposition of $\rBr$ (Proposition \ref{prop: decompo braces})}\label{sec: proof closed}

\subsubsection{Decomposition as a colimit}

We prove the existence of an isomorphism
\begin{equation}\label{eq: iso colim}
 \rBr(I;x)\cong \colim_{\Tas(I;x)} \Theta_\infty. 
\end{equation}

 Let us begin by a few preliminary considerations.

\begin{nota}  The tree obtained by grafting $k$ trees $\alpha_1,\ldots,\alpha_k$  onto a vertex $v$ is denoted by $B(v;\alpha_1,\ldots,\alpha_k)$. Notice that every tree can be written
in that way.
\end{nota}

 \begin{lem}\label{lem: lem separation 1 new} Let $T\in\ \mathcal{T}$.
 \begin{enumerate}
 \item  there exists $T' \in  \mathcal{T}_{2,1}$ such that $T\leq c(T')$. If $T\not\in \cal T_{\sq}$, then $T'$ can be chosen with no neutral vertices so that $c(T')=T'$.
  \item If $T\in \cal{T}_{2,1}\setminus \cal T_{\sq}$,  there exists a tree $T'\in \cal{T}_{2,1}$ with no neutral vertices such that $T\leq T'$ and
  \[ T=d^{\cal{A}}(T'),\ c(T)=d^{F\cal{A}}(T') \text{ with }  \cal{A}\in \Nest_2(T').\]
\end{enumerate}
 \end{lem}

 \begin{proof}
 Proof of (a):
 
 i) Every neutral round-shaped vertex in a tree  $T\in \cal T\setminus \cat T_{\sq}$ may be seen as coming from the blow-up of a closed vertex, so that $T$  is the boundary of a tree with no neutral vertices.
 
 ii)  If $T\in\cal T(I;\op)$ and $T\not\in\cal T_{\sq}$ then $T$ writes
 $T=T_1\cdots T_{2p+1}$ where $T_{2k}$ is an open rooted tree and $T_{2k+1}$ is either the ``empty" word or a word of neutral square-shaped rooted trees. For $1\leq k<p,$ there exists an open rooted tree $T'_k$ such that $T_{2k-1}T_{2k}\leq T'_k$ and by i) we can choose $T'_k$ with no neutral vertices. Similarly $T_{2p-1}T_{2p}T_{2p+1}\leq T'_p$ with $T'_p$ an open rooted tree with no neutral vertices.
 Hence $T\leq T'=T'_1\ldots T'_p$ where $T' \in  \mathcal{T}_{2,1}$ has no neutral vertices.
 
 iii) If $T=T_1\ldots T_p\in \mathcal T_{\sq}$ then $T\leq c(T)=B(v;\alpha_1,\ldots,\alpha_k)$ where $v$ is neutral square-shaped. 
 By i) there exist $\alpha'_1,\ldots,\alpha'_k$ with no neutral vertices such that $\alpha_i\leq \alpha'_i$.
The word $T'=B(v_1;\alpha'_1)\ldots B(v_k;\alpha'_k)$ with $v_i$ neutral square-shaped is in $\cal T_{2,1}$ and satisfies 
$T\leq c(T')=B(v;\alpha'_1,\ldots,\alpha'_k)$.

 Proof of (b): we apply part (a) to $c(T)$, so that there exists $T' \in  \mathcal{T}_{2,1}$ with no neutral vertices such that $c(T)\leq T'$.  It follows that $T\leq c(T)\leq T'$.
    If $\cal{A}\in \Nest_2(T')$ is such that $T=d^{\cal{A}}(T')$, then $c(T)=d^{F\cal{A}}(T')$.  
   Indeed,  let $\cal{B}$ be such that $c(T)=d^{\cal{B}}(T')$. Since $T$ is obtained from $c(T)$ by blow-ups of neutral vertices only, it follows that $\cal{A}\leq \cal{B}$ is obtained from $\cal{B}$ by adding sub intervals only. 
  \end{proof}

  \begin{prop}\label{prop: fonda} Let $[T_1],\ldots,[T_n] \in \Tas$. If $u\in \bigcap\limits_{i=1}^n (\cal T)_{\leq c(T_i)}$, then every element $P_u\in \cal T_{2,1}$ obtained by blowing up the neutral vertices of $u$ satisfies  
  \begin{align*}
P_u&\leq u\leq c(P_u) \\
  [P_u]&\leq [T_i], \ \forall 1\leq i\leq n.
  \end{align*}
\end{prop}

\begin{proof} First notice that if the intersection is non empty, then either $T_i\in\cal T\setminus\cal T_{\sq}$ for all $i$ or
$T_i\in\cal T_{\sq}$ for all $i$.

Assume first that $\forall i, T_i\not\in\cal T_{\sq}$. Let $u\in\cap_{i=1}^n(\cal T)_{\leq c(T_i)}$, and let $P_u \in \cal T_{2,1}$ obtained by blowing up the neutral vertices of $u$.  In particular $u\leq c(P_u)$.
By Lemma \ref{lem: lem separation 1 new},
there exists $T'$ with no neutral vertices such that $T_1=d^{\cal{A}}(T'),\ c(T_1)=d^{F\cal{A}}(T') \text{ with }  \cal{A}\in \Nest_2(T')$.
As a consequence, since $u\leq c(T_1)\leq T'$, there exists $\cal{B}$ such that $u=d^{\cal{B}}(T')$. In particular, $\cal{B}\leq F\cal{A}$. 
Since $P_u\leq u\leq T'$,
one has $P_u=d^{\cal{B}'}(T')$ where  $\cal B'\in\Nest_2(T')$ is a completion of $\cal B$ by sub intervals as in Definition \ref{def: completion}. It implies in particular that $\cal B'\leq F\cal A$ and that the intervals of $F\cal A$ have been completed in $\cal B'$ so that $[\cal B']\leq [\cal A]$.
As a consequence $[P_u]\leq [T_1]$.
The same arguments show that $[P_u]\leq [T_i]$, for all $i$.

Assume that $\forall i, T_i\in\cal T_{\sq}$ and let $u\in\cap_{i=1}^n(\cal T)_{\leq c(T_i)}$. Applying the bijection $\natural$ and Propositions \ref{prop: c et becarre}  and \ref{prop: c et becarre as}, one has 
\begin{align*}
P_{\natural u}&\leq \natural u\leq c(P_{\natural u})  \\
  [P_{\natural u}]&\leq [\natural T_i], \ \forall 1\leq i\leq n.
  \end{align*}
  Notice that $\natural^{-1} P_{\natural u}$ is obtained by blowing up the neutral vertices of $u$, and then is of the form $P_u$.
  Hence applying the bijection $\natural^{-1}$ one gets the result.
\end{proof}

\begin{proof}[Proof of \eqref{eq: iso colim}.]
 
Let us notice first that  $\rBr$ is the colimit of the functor $C:\cal T\rightarrow \Ch$ which associated to $T$ the complex 
$<(\cal T)_{\leq T}>$.
It induces, by restriction of categories, canonical inclusions $i_a:\Theta_\infty(a)\rightarrow \rBr$ so that the system $(\rBr,i_a)$ commutes with the structural maps $\{i^b_a:  \Theta_\infty(a) \subseteq  \Theta_\infty(b)\}_{a\leq b}.$

Let $(D,j_a)$ be a system in $\Ch$ that commutes with those structural maps. Let us define $g:(\rBr(I;x),i_a)\to (D,j_a)$ as follows. 

For a generator  $u\in \rBr(I;x)$, there is an $a\in \Tas(I;x)$ and a $u_a\in \Theta_\infty(a)$ such that $i_a(u_a)=u$, cf.  Lemma \ref{lem: lem separation 1 new}.

Let us define $g(u)=j_a(u_a)$. 

To see that $g$ is well defined it is sufficient to show that, if there are $a,b\in \Tas(I;x)$ and $u_a,u_b$ such that $u=i_a(u_a)=i_b(u_b)$, then there exist  $c\in \Tas(I;x)$  and $u_c   \in \Theta_\infty(c)$ such that $u=i_a(i_c^a(u_c))=i_b(i_c^b(u_c))$. 
Since $i_a,i_b$ are canonical inclusions, let us denote by $u:=i_a(u_a)=u_a=i_b(u_b)=u_b$. 
There exist  $[T_a]\leq a$ and $[T_b]\leq b$ such that $u\in (\mathcal{T})_{\leq c(T_a)}\cap (\mathcal{T})_{\leq c(T_b)}$. 
By Proposition \ref{prop: fonda}, $u\in (\cal T)_{\leq c(P_u)}$ with $[P_u]\leq [T_a]\leq a$ and $[P_u]\leq [T_b]\leq b$, hence $c=[P_u]$
fulfills the conditions.

Let us extend $g$ by linearity and let $u\in\rBr$ and $u_a\in \Theta_\infty(a)$ such that $i_a(u_a)=u$. Since $du=\sum_i\pm u_i$ with $u_i\leq u$ there exists $u_{i,a}\in\Theta_\infty(a)$ such that $u_i=i_a(u_{i,a})$. Hence $g(du)=j_a(du)=dj_a(u)=dg(u)$.
\end{proof}

\subsubsection{Contractibility of the cells}

  \begin{defn}
  Let $n,p\geq 1$. Let $\cal{A}_1,...,\cal{A}_p$ be $p$ elements of $\Nest(S(n))$ (resp. in $\Nest(I(n))$). 
 For $1\leq k\leq p$, let  $A_k^0$ (resp. $A_k^{\lft}$) be the largest interval of $\cal{A}_k$ containing zero (resp. $\lft$) with the convention that it is the empty set if such an interval does not exist in $\cal{A}_k$. 
   The sets $\cal{A}_1,...,\cal{A}_p$ are called \emph{separated} if  $\bigcap_{1\leq k\leq p} A_k^0=\emptyset$ 
   (resp. if  $\bigcap_{1\leq k\leq p} A_k^{\lft}=\emptyset$). 
  \end{defn}

\begin{rem}
 Let $T\in \cal{T}(I;x)$ with no neutral vertices and $T'\leq T$. Let $\cal{A}\in \Nest(T)$ such that $T'=d^{\cal{A}}(T)$. 
 If  $\cal{A}$ is separated for each vertex, then $c(T')=d^{F\cal{A}}(T)$. 
 \end{rem}

Let $n\geq 2$ and $T,T_1,...,T_n\in \mathcal{T}_{2,1}(I;x)$ be such that $[T_1],...,[T_n]\in (\Tas)_{\leq [T]}$. Assume $T\not\in\cal T_{\sq}$.
By Lemma \ref{lem: lem separation 1 new}, there is an element $T'$ with no neutral vertices such that $[T]\leq [T']$  and then $[T_1],...,[T_n]\in (\Tas)_{\leq [T']}$. 
It results that,  for each $1\leq i\leq n$, there exists  $\cal{A}_i\in \Nest_2(T')$  such that $[T_i]=d^{[\cal{A}_i]}([T'])$. %

\begin{lem}\label{lem: zero}
 With the notation above, suppose $\bigwedge_{1\leq i\leq n} F\cal{A}_i\in \Nest(T')$. %
 There exist $Y\in \mathcal{T}_{2,1}(I;x)$ and, for each $1\leq i\leq n$, an element $\cal{B}_i\in \Nest(Y)$ such that:
 \begin{itemize}
  \item $[T_i]=d^{[\cal{B}_i]}([Y])\in \Tas(I;x)$;
  \item $\bigwedge_{1\leq i\leq n} F\cal{B}_i \in \Nest(Y)$  and 
  $F\cal{B}_1,...,F\cal{B}_n$ are separated for each vertex.
 \end{itemize}
\end{lem}

\begin{proof} Recall that the vertices of $T'$ are either closed or open, and are ordered. We build $Y$ by backward induction
on the vertices $v$ of $T'$.
For each $1\leq i\leq n$, let $K_i:=d^{F\cal{A}_i}(T')\in \cal{T}$. 
We concentrate on the vertices $v$  such that the $F\cal{A}_1^v,...,F\cal{A}_n^v$ are not separated.

\noindent{\sl First case.}
 If $v$ is an open vertex, then $T'=R_1\ldots R_k$ is a word of open rooted trees. Let us write $R_i=B(v_i,\alpha_i^1,\ldots,\alpha_i^{l_i})$. By assumption, the set
 $J_v=\bigcap_{1\leq k\leq p} A_k^{v,\lft}=\{\lft,1,\ldots,s_v\}$ is non empty and is the maximal interval such that all the inputs of $v$ in $T'$ that are indexed by elements of $J_v$ are at the left of $v$ in $K_1,...,K_n$. 
 Starting from $T'$, we construct $Y$ locally  as follows. 
 \begin{enumerate}
  \item In case there exists $i>1$ such that $v_i=v$, we set $Y=R_1\ldots R'_{i-1}R'_i\ldots R_k$ with $R'_{i-1}=B(v_{i-1};\alpha_{i-1}^1,\ldots,\alpha_{i-1}^{l_{i-1}},\alpha_i^1,\ldots,\alpha_i^{s_v})$ and $R'_i=B(v_i;\alpha_i^{s_v+1},\ldots,\alpha_i^{l_i})$.
 \item If $v=v_1$, we set $Y=B(\sq\,;\alpha_1^1)\ldots B_(\sq\,;\alpha_1^{s_v})B(v_1;\alpha_1^{s_v+1},\ldots,\alpha_1^{l_1})R_2\ldots R_k$ (Figure \ref{fig: zero annulation 1}).
 \end{enumerate}
 
\noindent{\sl Second case.} If $v$ is a closed vertex, let $T'_v=B(v,\alpha_1,\ldots,\alpha_l)$ be the maximal subtree of $T'$ whose root is $v$.
 The set $J_v=\bigcap_{1\leq k\leq p} A_k^{v,0}=\{i_v<\ldots<0<\ldots <j_v\}$ is non empty and is the maximal interval such that 
 all the inputs of $v$ in $T'$ indexed by $J_v$ are below the vertex $v$ in $K_1,...,K_n$. 
 Starting from $T'$, we construct $Y$ locally  as follows. 
 \begin{enumerate}
  \item If the vertex $v$ is not  the root of $T'$, then there exists a vertex $w$ in $T'$ such that
  $T'_w=B(w,\beta_1,\ldots,\beta_m)$ with $T'_v=\beta_i$ for some $i$. The element $Y$ is obtained from $T'$ by replacing $T'_w$ by
 \[B(w;\beta_1,\ldots,\beta_{i-1},\alpha_{1},\ldots,\alpha_{j_v},B(v,\alpha_{j_v+1},\ldots,\alpha_{i_v-1}),\alpha_{i_v},\ldots,\alpha_l,\beta_{i+1},\ldots,\beta_m)\]
  \item If $v$ is the root vertex of $T'$, then  $T'=T'_v$. Let $T''_v=B(v,\alpha_{j_{v+1}},\ldots,\alpha_{i_{v-1}})$ and $X$ be
  a binary tree with neutral round shape vertices with $|J_v|$ leaves. The tree $Y$ is obtained from $X$  by attaching 
  $\alpha_1,\ldots,\alpha_{j_v},T''_v,\alpha_{i_v},\ldots,\alpha_m$ to its leaves (Figure \ref{fig: zero annulation 3}).
\end{enumerate}

Let us point out that $K_i$ is not necessarily  smaller than $Y$. However, for each $1\leq i\leq n$, there exists $P_i$ such that $P_i\leq K_i$ and $P_i\leq Y$. 
Indeed,  $P_i$ is obtained by adding $J_v$  to $F\cal{A}_i^v$, for each vertex $v$. 

For each $1\leq i\leq n$, let $\cal{C}_i$ be such that $P_i=d^{\cal{C}_i}(Y)$. 
The elements $\cal{C}_1,...,\cal{C}_n$ are separated by construction and  $\bigcup_{1\leq i\leq n} \cal{C}_i \in \Nest(Y)$.  
For $1\leq i\leq n$, the element $\cal{B}_i\in \Nest_2(Y)$ is obtained from $\cal{C}_i$ by adding subintervals. %

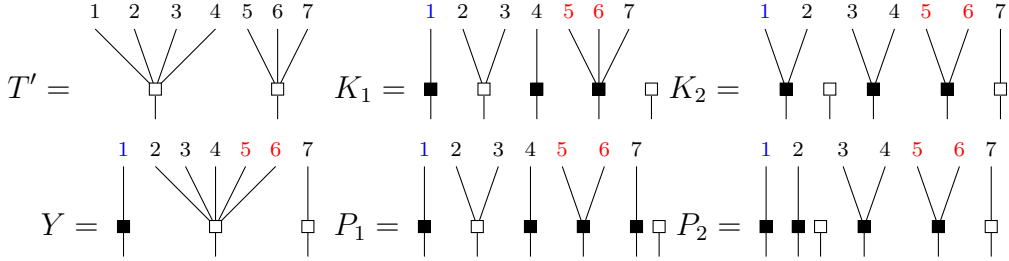
\begin{figure}[H]
\begin{align*}
T'=
\begin{tikzpicture}[baseline=-0.65ex,scale=.8]
\quatre{\wrect}{1}{2}{3}{4};
 \end{tikzpicture}
 \begin{tikzpicture}[baseline=-0.65ex,scale=.8]
 \trois{\wrect}{5}{6}{7};
 \end{tikzpicture}
 ~~& K_1=
 \begin{tikzpicture}[baseline=-0.65ex,scale=.8]
 \une{\brect}{\textcolor{blue}{1}};
 \end{tikzpicture}
 \begin{tikzpicture}[baseline=-0.65ex,scale=.8]
 \deux{\wrect}{2}{3};
 \end{tikzpicture}
 \begin{tikzpicture}[baseline=-0.65ex,scale=.8]
 \une{\brect}{4};
 \end{tikzpicture}
 \begin{tikzpicture}[baseline=-0.65ex,scale=.8]
 \trois{\brect}{\textcolor{red}{5}}{\textcolor{red}{6}}{7};
 \end{tikzpicture}
 \begin{tikzpicture}[baseline=-0.65ex,scale=.8]
 \zero{\wrect};
 \end{tikzpicture}
~~~~ K_2= 
\begin{tikzpicture}[baseline=-0.65ex,scale=.8]
 \deux{\brect}{\textcolor{blue}{1}}{2};
 \end{tikzpicture}
 \begin{tikzpicture}[baseline=-0.65ex,scale=.8]
 \zero{\wrect};
 \end{tikzpicture}
 \begin{tikzpicture}[baseline=-0.65ex,scale=.8]
 \deux{\brect}{3}{4};
 \end{tikzpicture}
 \begin{tikzpicture}[baseline=-0.65ex,scale=.8]
 \deux{\brect}{\textcolor{red}{5}}{\textcolor{red}{6}};
 \end{tikzpicture}
 \begin{tikzpicture}[baseline=-0.65ex,scale=.8]
 \une{\wrect}{7};
 \end{tikzpicture}
 \\
 Y=
 \begin{tikzpicture}[baseline=-0.65ex,scale=.8]
 \une{\brect}{\textcolor{blue}{1}};
 \end{tikzpicture}
 \begin{tikzpicture}[baseline=-0.65ex,scale=.8]
 \cinq{\wrect}{2}{3}{4}{\textcolor{red}{5}}{\textcolor{red}{6}};
 \end{tikzpicture}
 \begin{tikzpicture}[baseline=-0.65ex,scale=.8]
 \une{\wrect}{7};
 \end{tikzpicture}
 ~~& P_1=
 \begin{tikzpicture}[baseline=-0.65ex,scale=.8]
 \une{\brect}{\textcolor{blue}{1}};
 \end{tikzpicture}
 \begin{tikzpicture}[baseline=-0.65ex,scale=.8]
 \deux{\wrect}{2}{3};
 \end{tikzpicture}
 \begin{tikzpicture}[baseline=-0.65ex,scale=.8]
 \une{\brect}{4};
 \end{tikzpicture}
 \begin{tikzpicture}[baseline=-0.65ex,scale=.8]
 \deux{\brect}{\textcolor{red}{5}}{\textcolor{red}{6}};
 \end{tikzpicture}
 \begin{tikzpicture}[baseline=-0.65ex,scale=.8]
 \une{\brect}{7};
 \end{tikzpicture}
 \begin{tikzpicture}[baseline=-0.65ex,scale=.8]
 \zero{\wrect};
 \end{tikzpicture}
 ~~~~P_2= 
 \begin{tikzpicture}[baseline=-0.65ex,scale=.8]
 \une{\brect}{\textcolor{blue}{1}};
 \end{tikzpicture}
 \begin{tikzpicture}[baseline=-0.65ex,scale=.8]
 \une{\brect}{2};
 \end{tikzpicture}
 \begin{tikzpicture}[baseline=-0.65ex,scale=.8]
 \zero{\wrect};
 \end{tikzpicture}
 \begin{tikzpicture}[baseline=-0.65ex,scale=.8]
 \deux{\brect}{3}{4};
 \end{tikzpicture}
 \begin{tikzpicture}[baseline=-0.65ex,scale=.8]
 \deux{\brect}{\textcolor{red}{5}}{\textcolor{red}{6}};
 \end{tikzpicture}
 \begin{tikzpicture}[baseline=-0.65ex,scale=.8]
 \une{\wrect}{7};
 \end{tikzpicture}
\end{align*}\caption{The element $Y$ from $T'$, $K_1$ and $K_2$ (open case).}\label{fig: zero annulation 1}
\end{figure}

\begin{figure}[H]
\begin{equation*}
T'=
\begin{tikzpicture}[baseline=-0.65ex,scale=.8]
 \cinqvert{\wrond}{1}{2}{3}{4}{(0,0)};
 \cinqvert{\wrond}{5}{6}{7}{8}{(0,1.5)};
 \end{tikzpicture}
 K_1=
 \begin{tikzpicture}[baseline=-0.65ex,scale=.8]
 \deuxvert{\brond}{}{}{(0,.6)};
 \deuxvert{\wrond}{}{}{(0,0)};
 \draw (-.9,1.65) node [left] {\tiny$2$};
 \draw (-.9,1.) node [left] {\tiny$1$};
 \draw (.9,1.65) node [right] {\tiny$3$};
 \draw (.9,1.) node [right] {\tiny$4$};
 \draw (-.9,2.45) node [left] {\tiny$\textcolor{red}{5}$};
 \draw (-.9,3.1) node [left] {\tiny$6$};
 \draw (.9,2.45) node [right] {\tiny$\textcolor{red}{8}$};
 \draw (.9,3.1) node [right] {\tiny$7$};
 \deuxvert{\brond}{}{}{(0,1.4)};
 \deuxvert{\wrond}{}{}{(0,2)};
 \end{tikzpicture}
 K_2=
 \begin{tikzpicture}[baseline=-0.65ex,scale=.8]
 \cinqvert{\wrond}{1}{2}{3}{4}{(0,0)};
 \cinqvert{\brond}{\textcolor{red}{5}}{6}{7}{\textcolor{red}{8}}{(0,.9)};
 \unevert{\wrond}{(0,1.5)};
 \end{tikzpicture}
Y=
 \begin{tikzpicture}[baseline=-0.65ex,scale=.8]
 \cinqvert{\wrond}{1}{}{}{4}{(0,0)};
 \draw (-.7,0.55) node [left] {\tiny$2$};
 \draw (.7,0.55) node [right] {\tiny$3$};
 \deuxvert{\wrond}{\textcolor{red}{5}}{\textcolor{red}{8}}{(0,0)};
 \deuxvert{\wrond}{6}{7}{(0,1.5)};
 \end{tikzpicture}
 \end{equation*}\caption{The element $Y$ from $T'$, $K_1$ and $K_2$ (closed case).}\label{fig: zero annulation 3}
 \end{figure}
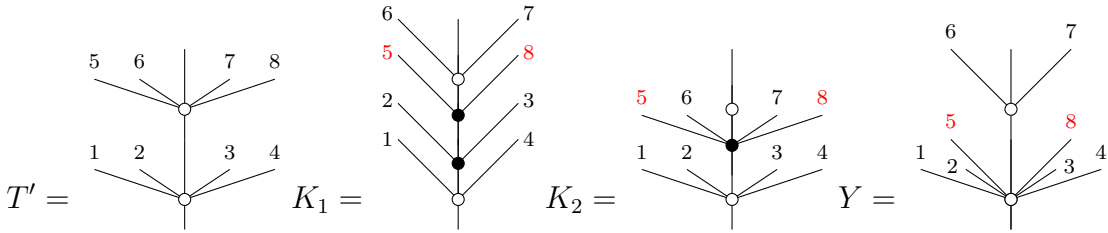
 \end{proof}

 \begin{rem}\label{rm: separation}
 With the notations above, for each open (resp. closed) vertex $v$ of $Y$, there is at least one $1\leq j\leq n$ such that  $\cal{B}^v_j$ has no intervals containing $\lft$ (resp. $0$).  Indeed, this corresponds to the $1\leq j\leq n$ such that the intersection interval $J_v$ is the interval of $F\cal{A}_j^v$ containing $\lft$ (resp. $0$). 
 \end{rem}

\begin{prop}\label{prop: PT contract}
 Let $n\geq 2$ and $T,T_1,...,T_n\in \mathcal{T}_{2,1}(I;x)$ be such that $[T_1],...,[T_n]\in (\Tas)_{\leq [T]}$. Then
 either $\bigcap_{1\leq i\leq n} (\cal T)_{\leq c(T_{i})}=\emptyset$ or, 
 there exist  $P,Q\in \cal T$  such that  
 
 \begin{align*}
\bigcap_{1\leq i\leq n} (\cal T)_{\leq c(T_{i})}&=(\cal T)_{\leq P}, \\
\bigcap_{1\leq i\leq n} (\cal T)_{\leq c(T_{i})}\bigcap (\cal T)_{\leq c(T)}&=(\cal T)_{\leq Q}
 \end{align*}
\end{prop}

\begin{proof} 
Assume first that $T\not\in{\cal T}_{\sq}$ and that the intersection $ \bigcap_{1\leq i\leq n} (\cal T)_{\leq c(T_{i})}$ is non empty.
Let $u$ be an element in the intersection.

 By Lemma \ref{lem: lem separation 1 new}, there exists  $T'\geq T$ with no neutral vertices, and $\cal A\in \Nest_2(T')$ such that $T=d^{\cal A} T'$ and $c(T)=d^{\cal FA} T'$.
 For $1\leq i\leq n$, let $\cal{A}_i\in \Nest_2(T')$ be such that $[T_i]=d^{[\cal{A}_i]}([T'])$. 
 For $1\leq i\leq n$, let us define $K_i:= d^{F\cal{A}_i}(T')$. Notice that $K_i\leq c(T_i)$. 
 
Let $P_u$ be obtained as in Proposition \ref{prop: fonda}, that is, 
$P_u$ is obtained as a blow-up of neutral vertices of $u\leq c(T_i)$ and $[P_u]\leq [T_i]$ for all $i$, so that
$P_u\leq K_i$ for all $i$. Writing $P_u$ as $d^{\cal B} T'$ with $\cal B\in \Nest_2(T')$, one gets
 $\cal B\leq F\cal A_i$ for all $i$ and thus $\cal B\leq \wedge_i F\cal A_i$. As a consequence $\wedge_i F\cal A_i$ is nested.

By Lemma \ref{lem: zero}, we can assume that $F\cal{A}_1,...,F\cal{A}_n$ are separated. Let us define $P=d^{\wedge_iF\cal{A}_i}(T)$. It follows 
 from Proposition \ref{prop: bij T no neutral}  that $P=\wedge_i K_i$. 
In particular,   for $1\leq i\leq n$,  $P\leq K_i\leq c(T_i),$ so that
 $P\in \bigcap_{1\leq i\leq n} (\mathcal{T})_{\leq c(T_{i})}$.  

 Let us show that $ \bigcap_{1\leq i\leq n} (\mathcal{T})_{\leq c(T_{i})}\subset (\cal T)_{\leq P}.$ In order to do this, we prove that any $a\in \bigcap_{1\leq i\leq n} (\mathcal{T})_{\leq c(T_{i})}$ satisfies
$a\leq K_i$ for all $i$.

Notice that if $K_i<c(T_i)$, then we are necessarily in one of the two following situations:
there exist  vertices $w<v$ in $T'$ such that $F\mathcal A_i^v$ has an interval containing  $0$ and $F\mathcal A_i^w$ has an interval containing the labelling of the edge joining $w$ to $v$; or there exist  vertices $w<v$ in $T'$ such that $F\mathcal A_i^v$ has an interval containing  $\lft$ and $F\mathcal A_i^w$ has an interval containing $\rgt$ (see the example of the trees $K_1, K_2$ in Figure \ref{fig: zero annulation 1} and the tree $K_1$ in Figure \ref{fig: zero annulation 3}).
Hence,  $c(T_i)$ has a neutral vertex $z$ with some inputs in $In(v)$ and some others in $In(w)$ and satisfying $w<z<v$.
Remark \ref{rm: separation} implies that there is at least one $1\leq j_0\leq n$ such that  
 $F\cal{A}_{j_0}^v$ has no interval containing $0$ or $\lft$. Consequently, there is no vertex $z\in c(T_{j_0})$ satisfying
 the same conditions than the one for $c(T_i)$. Since $a\leq c(T_i)$ and $a\leq c(T_{j_0})$, it is obtained as a blow-up of
 $c(T_i)$ and $c(T_{j_0})$, which means that $z$ must be blowned up in order to sepate the inputs in $In(v)$ from the inputs in $In(w)$. Applying this argument for every vertices in $T'$ one gets  $a\leq K_i$ for every $i$. Hence $a\leq \wedge_i K_i=P$.

Assume $T\in {\cal T}_{\sq}$. Since $\natural$ commutes with the contraction map (Proposition \ref{prop: c et becarre}) and passes to the quotient in $\Tas$
(Proposition \ref{prop: c et becarre as}),
one has
\[ \natural\left(\bigcap_{1\leq i\leq n} (\cal T)_{\leq c(T_{i})}\right)=\bigcap_{1\leq i\leq n} ({\cal T})_{\leq c(\natural T_{i})}\]
which is either empty or of the form $(\cal T)_{\leq P}$.

Let us prove the second part of the Proposition. We assume that $\bigcap_{1\leq i\leq n} (\cal T)_{\leq c(T_{i})}$  is non empty
and let $u$ be an element in the intersection. Applying 
Proposition \ref{prop: fonda}, there exists 
$P_u\in \mathcal{T}_{2,1}$ so that  $P_u\leq u$, $[P_u]\leq [T_i]\leq [T]$ and $P_u\leq T\leq c(T)$. As a consequence,
 $P_u\in \bigcap_{1\leq i\leq n} (\cal T)_{\leq c(T_{i})}\bigcap (\cal T)_{\leq c(T)}$. Applying the first part of the Proposition, there exists $Q\in\cal T$ such that
 \[ \bigcap_{1\leq i\leq n} (\cal T)_{\leq c(T_{i})}\bigcap (\cal T)_{\leq c(T)}=(\cal T)_{\leq Q}\]
\end{proof}

\begin{prop} For all $[T]\in\Tas$, $\Theta_\infty([T])$ is contractible.
\end{prop}

\begin{proof} In this proof we use the results of Borsuk and Bjorner in  \cite{Bor48} or \cite{bjorner}, and more specifically the nerve theorem. 
Let $[T]$ be in $\Tas$. Following  \cite{Bor48} or \cite{bjorner} we associate to the covering of $\Theta_\infty([T])$ by 
$<(\cal T)_{\leq a}>_{a\in P^{[T]}}$ the nerve $N^{[T]}$ which is the simplicial complex whose vertices are labelled by $a\in P^{[T]}$.  
Simplices are given by $\sigma\subset P^{[T]}$ such that $\cap_{a\in\sigma} (\cal T)_{\leq a}\not=\emptyset$.
Applying Propositon \ref{prop: PT contract} and Proposition \ref{prop: bij T no neutral} we get that a non empty intersection is contractible.
The nerve theorem asserts that $\Theta_\infty([T])$  has the homotopy type of that of $N^{[T]}$. Proposition \ref{prop: PT contract} implies that $N^{[T]}$  is a cone whose apex is the vertex $c(T)$, hence  is contractible.
\end{proof}

 \appendix

\section{Generating operations of $\Lambda_\cl\rBr$}\label{A:A}
 
 We recall here our notation for the generating operations of $\Lambda_\cl\rBr$. This operad acts on a pair of  dgvs $((L,d_L),(R,d_R))$ with

 $\bullet$ $(R,\mu,d_R)$ is a dg algebra, whose structure maps are equivalently given by a square zero coderivation $d_{\su R}+\partial_2$ on $T^c(\su R)$;
 
 $\bullet$   For $n\geq 2$, the maps $\delta_n: L^{\otimes n}\rightarrow L$ assemble with $d_L$ to a square zero coderivation on $T^c(L)$. This implies that
 $$\sum_{i+j=n+1} \delta_i\circ\delta_j+ \delta_{n}\circ d_L+d_L\circ\delta_n =0,$$ that is, 
 $$d_{\Lambda_\cl\rBr}(\delta_n)=-\sum_{i+j=n+1} \delta_i\circ\delta_j,$$  which is Relation (\ref{deltaroundneutral}).

$\bullet$ $M_{1,q}:L\otimes   L^{\otimes q}\rightarrow L$  defines a morphism of coalgebras
$M:T^cL\otimes T^cL\rightarrow T^cL$ which is associative and is compatible with $\delta$.
Notice that for $x\in L$ and $\underline y\in L^{\otimes n}$, one has
$$M(x;\underline y)=\sum\pm y_{(1)}\otimes M(x;y_{(2)})\otimes y_{(3)},$$
which implies that
$$d_L M_{1,n}+\sum_{a+b=n+1}\delta_b\circ M_{1,a}=M\circ d_L+\sum_{a+b=n+1} M_{1,a}\circ\delta_b$$
that is
$$d_{\Lambda_\cl\rBr} M_{1,n}=\sum_{a+b=n+1} M_{1,a}\circ \delta_b-\delta_b\circ M_{1,a},$$
which is Relation (\ref{deltaclosed}).

$\bullet$ Recall that 
$$K:T^c{\su R}\otimes T^c(L)\rightarrow T^c{\su R},$$ 
is the unique coalgebra morphism so that
$K_{0,n}=\su G_n$ and  $K(\su r,\underline l)=\su\Gamma_{1,q}(r,\underline l)$.

In addition $K$ is compatible with the differential.
For $\su r\in\su R$ and $\underline l\in L^{\otimes b}$, one has
\begin{multline*}
K(\su r,\underline l)=K_{1,n}(\su r,\underline l)+\sum_{\substack{i+j=n\\ i>0}}(-1)^{|l_{(1)}||\su r|} K_{0,i}(l_{(1)})\otimes K_{1,j}(\su r,l_{(2)})+
\sum_{\substack{i+j=n\\ j>0}} K_{1,i}(\su r,l_{(1)})\otimes  K_{0,j}(l_{(2)})\\+\text{ higher terms }
\end{multline*}

As a consequence, the compatibility with the differential gives
\begin{multline*}
d_{\su R}K_{1,n}(\su r,\underline l)+\partial_2\left(
\sum_{\substack{i+j=n\\ i>0}}(-1)^{|l_{(1)}||\su r|} K_{0,i}(l_{(1)})\otimes K_{1,j}(\su r,l_{(2)})+
\sum_{\substack{i+j=n\\ j>0}} K_{1,i}(\su r,l_{(1)})\otimes  K_{0,j}(l_{(2)})\right)
=\\
K(d_{\su R}\su r,\underline l)+( K\circ (1_{\su R}\otimes (d_L+\delta))(\su r,\underline l)
\end{multline*}

which is equivalent to
\[-d_R\Gamma_{1,n}(r,\underline l)-\sum_{i+j=n}\left (\mu(G_i,\Gamma_{1,j}\tau_{1,i}\cdot(r,\underline l)+\mu(\Gamma_{1,j},G_i)(r,\underline l)\right)\]
\[=-\Gamma(d_R r,\underline l)- ( \Gamma\circ d_L+\sum_{i+j=n+1}\Gamma_i\circ\delta_j)( r,\underline l)\]
and to
\[d_{\Lambda_\cl\rBr} \Gamma_{1,n}=\sum_{i+j=n+1}\Gamma_i\circ\delta_j-\sum_{i+j=n}\left(\mu(G_i,\Gamma_{1,j})\cdot \tau_{1,i}+\mu(\Gamma_{1,j},G_i)\right)\]
 which is Relation (\ref{deltaopen}), where $\tau_{1,i}$ is the permutation that permutes the first variable with the next $i$ variables.

$\bullet$ The same computation holds for $K_0:T^c(L)\rightarrow T^c(\su R)$:
$$d_{\su R}K_{0,n}(\underline l)+\sum_{a+b=n} \partial_2(K_{0,a}(l_{(1)})\otimes K_{0,_b}(l_{(2)}))=K_0((d_L+\delta)(\underline l)),$$
and
$$-d_RG_n(\underline l)+\sum_{a+b=n}(-1)^{|l_{(1)}|+1}\mu(G_a(l_{(1)}))\otimes G_b(l_{(2)}))=(G\circ d_L+G\circ\delta)(\underline l).$$
As a consequence, we have:
$$d_{\Lambda_\cl\rBr}G_n=-\sum_{a+b=n} \mu(G_a\otimes G_b)-\sum_{i+j=n+1} G_i\circ \delta_j$$
which is Relation (\ref{deltaneutralsquare}).

\def\cprime{$'$} \def\cprime{$'$} \def\cprime{$'$} \def\cprime{$'$}
\providecommand{\bysame}{\leavevmode\hbox to3em{\hrulefill}\thinspace}
\providecommand{\MR}{\relax\ifhmode\unskip\space\fi MR }
\providecommand{\MRhref}[2]{%
  \href{http://www.ams.org/mathscinet-getitem?mr=#1}{#2}
}
\providecommand{\href}[2]{#2}

\end{document}